%% file: utmost_rigidity_-_arXiv_v2.tex
\documentclass[10pt,letterpaper,twoside]{article}

\usepackage{amsthm,amsmath,amssymb,amsfonts}
\usepackage[utf8x]{inputenc}
\usepackage[all]{xy}
\usepackage{enumerate}
\usepackage{lipsum}
\usepackage{pdfpages}
\usepackage{fancyhdr}
\usepackage{hyperref}
  \hypersetup{%
  pdftitle={The utmost rigidity property for quadratic foliations with an invariant line on P2},pdfauthor={Valente Ramirez},%
  colorlinks=true,linktocpage=true,linkcolor=blue,citecolor=red,urlcolor=blue%
  }
\usepackage[left=1.25in, right=1.25in, top=1.25in, bottom=1.5in]{geometry}

\newtheorem{theorem}{Theorem}
\newtheorem{corollary}{Corollary}
\newtheorem{lemma}{Lemma}[section]
\newtheorem{proposition}{Proposition}[section]
\newtheorem{formula}{Formula}[section]
\newtheorem*{key}{Key Lemma}
\newtheorem*{main}{Main Lemma}
\newtheorem*{elimination}{Elimination Lemma}
  \theoremstyle{definition}
\newtheorem{definition}{Definition}[section]
\newtheorem{remark}{Remark}[section]

\numberwithin{equation}{section}

\newcommand{\Z}{\mathbb{Z}}

\newcommand{\C}{\mathbb{C}}
\renewcommand{\P}{\mathbb{P}}

\newcommand{\Diff}{\operatorname{Diff}(\mathbb{C},0)}

\newcommand{\A}{\mathcal{A}}
\newcommand{\F}{\mathcal{F}}
\newcommand{\Ft}{\widetilde{\mathcal{F}}}

\newcommand{\LF}{\mathcal{L}_{\mathcal{F}}}
\newcommand{\LFt}{\mathcal{L}_{\widetilde{\mathcal{F}}}}

\newcommand{\Aff}[2]{\operatorname{Aff}({#1},{#2})}

\renewcommand{\Re}[1]{\operatorname{Re}\,{#1}}
\renewcommand{\Im}[1]{\operatorname{Im}\,{#1}}

\newcommand{\var}[2]{\varphi_{{#1}\{\gamma_{{#2}}\}}(0)}
\newcommand{\vart}[2]{\widetilde{\varphi}_{{#1}\{\gamma_{{#2}}\}}(0)}

\newcommand{\Res}{\operatorname{Res}}

\newcommand\blfootnote[1]{%
  \begingroup
  \renewcommand\thefootnote{}\footnote{#1}%
  \addtocounter{footnote}{-1}%
  \endgroup
}

\newenvironment{changemargin}[2]{%
\begin{list}{}{%
\setlength{\topsep}{0pt}%
\setlength{\leftmargin}{#1}%
\setlength{\rightmargin}{#2}%
\setlength{\listparindent}{\parindent}%
\setlength{\itemindent}{\parindent}%
\setlength{\parsep}{\parskip}%
}%
\item[]}{\end{list}}

\title{The utmost rigidity property for quadratic\\ foliations on $\P^2$ with an invariant line}
\author{Valente Ram\'{i}rez$^\ast$}
\date{\today}

\begin{document}

\maketitle
\blfootnote{$^\ast$This work was supported by the grants UNAM--DGAPA--PAPIIT IN 102413 and CONACYT 219722.}

\pagestyle{fancy}

\input{0-abstract}

\tableofcontents

\section{Introduction}\label{sec:intro}\input{1a-introduction}

  \subsection{Acknowledgments}\input{1b-acknowledgments}\label{sec:acknowledgments}

\section{Structure of the work}\label{sec:structure}\input{2-structure}

\section{Sketch of the proofs}\label{sec:sketch}\input{3-sketch}

\section{Definitions and normalizations}\label{sec:normalizations}\input{4-normalizations}

\section[Analysis of the power series]{Analysis of the power series expansion of the distinguished parabolic germs \texorpdfstring{$f_j$}{}}\label{sec:analysis}\input{5-analysis}

\section{Proof of the Key lemma}\label{sec:keylemma}\input{6-keylemma}

\section{Proof of Elimination lemma}\label{sec:elimination}\input{7-elimination}

\clearpage
\fancyhead[LE,RO]{}\fancyhead[LO,RE]{\slshape APPENDIX: MATHEMATICA SCRIPT}
  \addcontentsline{toc}{section}{Appendix: \emph{Mathematica} script}
    \input{8-appendix}

\clearpage
\fancyhead[LE,RO]{}\fancyhead[LO,RE]{\slshape \leftmark}
  \bibliographystyle{alpha}
  \bibliography{ref}
  \addcontentsline{toc}{section}{References}

~\\[4mm]
Valente Ram\'{i}rez
\begin{center} 
\begin{flushright}
\begin{minipage}{.975\textwidth}
\begin{flushleft}
Department of Mathematics\\
Cornell University\\
120 Malott Hall\\
Ithaca NY 14853\\
USA\\
\textit{e--mail:} \texttt{valente@math.cornell.edu}
\end{flushleft}
\end{minipage}
\end{flushright}
\end{center}

\end{document}

%% file: 0-abstract.tex

\begin{abstract}
In this work we consider holomorphic foliations of degree two on the complex projective plane $\P^2$ having an invariant line. In a suitable choice of affine coordinates these foliations are induced by a quadratic vector field over the affine part in such a way that the invariant line corresponds to the line at infinity. We say that two such foliations are topologically equivalent provided there exists a homeomorphism of $\P^2$ which brings the leaves of one foliation onto the leaves of the other and preserves orientation both on the ambient space and on the leaves. 

The main result of this paper is that in the generic case two such foliations may be topologically equivalent if and only if they are analytically equivalent. In fact, it is shown that the analytic conjugacy class of the holonomy group of the invariant line is the modulus of both topological and analytic classification. We obtain as a corollary that two generic orbitally topologically equivalent quadratic vector fields on $\C^2$ must be orbitally affine equivalent.

This result improves, in the case of quadratic foliations, a well--known result by Ilyashenko that claims that two generic and topologically equivalent foliations with an invariant line at infinity are affine equivalent provided they are close enough in the space of foliations and the linking homeomorphism is close enough to the identity map of $\P^2$. 
\end{abstract}

%% file: 1a-introduction.tex

Any polynomial vector field on $\C^2$ with isolated singularities defines a singular holomorphic foliation by curves which can be analytically extended to the projective plane $\P^2$. Conversely, any holomorphic foliation on $\P^2$ with isolated singularities is given by a polynomial vector field on any affine chart. We are interested in foliations on $\P^2$ with an invariant line. It is convenient to choose affine coordinates such that the invariant line becomes the line at infinity. Since any line can be mapped to any other line by a linear automorphism of $\P^2$ there is no loss of generality in choosing a distinguished line $\mathcal{L}$ and considering only foliations which leave $\mathcal{L}$ invariant. Define $\A_n$ to be the class of those singular foliations on $\P^2$ which in the fixed affine chart $\C^2\approx\P^2\setminus\mathcal{L}$ are induced by a polynomial vector field of degree $n$ and have an invariant line at infinity. Note that the line at infinity with the singularities removed is a leaf of the foliation. We call this leaf the leaf at infinity or the infinite leaf indistinctly.

\begin{remark}\label{remark:degree}
Foliations from the class $\A_n$ have, by definition, \emph{affine degree} $n$, since they are induced by a polynomial vector field on $\C^2$ of degree $n$. The fact that they have an invariant line at infinity implies that such foliations also have \emph{projective degree} $n$. By projective degree $n$ we mean that such foliations have exactly $n$ tangencies with any line not invariant by the foliation (cf.~\cite{GomezMontOrtiz2004,Brunella2004}).
\end{remark}

Two foliations from the class $\A_n$ are \emph{topologically equivalent} if there exists an orientation--preserving homeomorphism of $\P^2$ that brings the leaves of the first foliation onto the leaves of the second one and preserves the natural orientation on these leaves. In case such a map is an affine map on $\C^2$ we say that the foliations are affine equivalent.

Let $\F\in\A_n$ and denote by $\LF$ its leaf at infinity. Given a base point $b\in\LF$, the germ of a cross--section $(\Gamma,b)$ transversal to the leaves of $\F$ and a parametrization $(\C,0)\to(\Gamma,b)$ we obtain the holonomy representation $\Delta\colon\pi_1(\LF,b)\to\Diff$ of the fundamental group of the infinite leaf on the space $\Diff$ of germs of invertible holomorphic maps at $(\C,0)$. Its image is called the \emph{holonomy group at infinity} of $\F$.

\begin{definition}
We say that two foliations $\F$ and $\Ft$ from the class $\A_n$ have \emph{analytically conjugate holonomy groups at infinity} whenever there exist the germ of a conformal map $h\in\Diff$ and a geometric isomorphism\footnote{We say that the isomorphism $H_{\ast}$ is geometric if it is induced by some orientation--preserving homeomorphism $H\colon\LF\to\LFt$.} $H_{\ast}\colon\pi_1(\LF,b)\to\pi_1(\LFt,\tilde{b})$ such that for any loop $\gamma\in\pi_1(\LF,b)$ we have $h\circ\Delta_\gamma=\widetilde{\Delta}_{H_{\ast}\gamma}\circ h$. 
\end{definition}

\subsection{Rigidity of polynomial foliations}

Generic foliations from the class $\A_n$ exhibit a phenomenon known as topological rigidity. Topological rigidity of polynomial foliations was, until now, more a heuristic idea than a formal statement. The idea of topological rigidity is that topological equivalence of foliations implies their analytic equivalence. There are several theorems in the literature asserting that topological equivalence of generic foliations plus some additional hypotheses imply their affine equivalence. The first such rigidity property for generic polynomial foliations was discovered by Ilyashenko in \cite{Ilyashenko1978} and called \emph{absolute rigidity}. 

\begin{definition}
We say that a foliation $\F\in\A_n$ is \emph{absolutely rigid} if there exist a neighborhood $U$ of $\F$ in $\A_n$ and a neighborhood $V$ of the identity map in the space of self homeomorphisms of $\P^2$ such that any foliation from $U$ which is conjugate to $\F$ by a homeomorphism in $V$ is necessarily affine equivalent to $\F$.
\end{definition}

It is proved in \cite{Ilyashenko1978} that a generic polynomial foliation is absolutely rigid. However, their genericity assumptions excluded a dense subset of $\A_n$. These conditions have been substantially weakened by Shcherbakov, Nakai and others (cf.~\cite{Shcherbakov1984,Nakai1994,LinsNetoSadScardua1998}). In the latest works, the key assumption on a foliation is the non--solvability of its holonomy group at infinity. 

Later on, Ilyashenko and Moldavskis proved that generic quadratic foliations exhibit a stronger rigidity property, known as \emph{total rigidity} \cite{IlyashenkoMoldavskis2011}.

\begin{definition}
A polynomial foliation $\F\in\A_n$ is \emph{totally rigid} if there exist only a finite number of foliations (up to affine equivalence) from the class $\A_n$ which are topologically equivalent to $\F$.
\end{definition}

In \cite{IlyashenkoMoldavskis2011} the number of affine classes of foliations which are topologically equivalent to a given generic foliation from $\A_2$ is estimated to be at most 240. This result is proved using the topological invariance of the Baum--Bott indices for generic quadratic foliations. 

In this work we prove for the first time that the paradigm of topological rigidity of polynomial foliations may be formalized, at least in the case of quadratic foliations with an invariant line at infinity: Two generic foliations form $\A_2$ are topologically equivalent if and only if they are affine equivalent. Moreover, this is proved by comparing the holonomy groups at infinity exclusively and we thus conclude that it is the holonomy group that serves as a modulus of analytic (hence also topological) classification.

\subsection{Statement of the theorem}

The following theorem is the main result of this work.

\begin{theorem}\label{thm:main}
Let $\F\in\A_2$ be a generic foliation and suppose its holonomy group at infinity is analytically conjugate to the holonomy group of $\Ft\in\A_2$. There exists an affine map on $\C^2$ that conjugates $\F$ to $\Ft$. 
\end{theorem}

It is well known that generic topologically equivalent foliations have analytically conjugate holonomy groups \cite{Ilyashenko1978}. The next results follows immediately from Theorem \ref{thm:main}.

\begin{corollary}\label{coro:rigidity}
Two generic foliations from $\A_2$ are topologically equivalent if and only if they are affine equivalent.
\end{corollary}

We say that two vector fields are \emph{orbitally topologically equivalent} whenever there exists an orientation--preserving homeomorphism of $\C^2$ that maps the integral curves of the first vector field onto those of the second one. If two quadratic vector fields on $\C^2$ are orbitally topologically equivalent it need not be true that the induced foliations on $\P^2$ are topologically equivalent since the linking homeomorphism need not extend to the line at infinity. However, if the singularities at infinity are hyperbolic, it can be easily proved that such linking homeomorphism takes the separatrix set of the former foliation onto the separatrix set of the latter one (cf.~\cite{TeymuriMafraScardua2013}). Once this has been established we may carry out with no problem an argument by Mar\'{i}n which guarantees that, even though the homeomorphism need not extend to the infinite line, the holonomy groups at infinity are still conjugated (see Theorem A in \cite{Marin2003}). We obtain the following result.

\begin{corollary}\label{coro:qvfs}
Two generic quadratic vector fields on $\C^2$ are orbitally topologically equivalent if and only if they are orbitally affine equivalent.
\end{corollary}

The above results may be summarized as follows.

\begin{corollary}\label{coro:equivalences}
Let $\F,\Ft\in\A_2$ be generic foliations. The following are equivalent.
\begin{enumerate}\parskip -1pt
 \item There exists a homeomorphism of $\C^2$ conjugating $\F$ to $\Ft$,
 \item There exists a homeomorphism of $\P^2$ conjugating $\F$ to $\Ft$,
 \item Foliations $\F$ and $\Ft$ have analytically equivalent holonomy groups at infinity,
 \item There exists an affine map on $\C^2$ conjugating $\F$ to $\Ft$.
\end{enumerate}
\end{corollary}

\subsection{Genericity assumptions}\label{subsec:genericity}

In order to prove Theorem \ref{thm:main} we shall consider exclusively foliations from the class $\A_2$ that satisfy the generic properties listed below. Also, we enumerate once and for all the singular points at infinity in such a way that $\Re{\lambda_1}\geq\Re{\lambda_2}\geq\Re{\lambda_3}$.
\begin{enumerate}[(i)]
 \item The holonomy group at infinity is non--solvable,
 \item The characteristic numbers $\lambda_1$, $\lambda_2$, $\lambda_3$ of the singular points at infinity are pairwise different and do not belong to the set $\frac{1}{3}\Z\cup\frac{1}{4}\Z\cup\frac{1}{5}\Z$,
 \item The commutator of the two holonomy maps corresponding to the \emph{standard geometric generators}\footnote{The standard geometric generators $\mu_i$ are described in Definition \ref{def:geometricgenerators}.} of the fundamental group of the infinite leaf belongs to the class of parabolic germs with non--zero quadratic term (see Remark \ref{rmk:genericity(iii)} in Subsection \ref{subsec:ideas}).
\end{enumerate}
Moreover, there is an additional technical requirement needed to prove Theorem \ref{thm:main}. In Section \ref{subsec:eliminationlemma} we shall construct a dense Zariski open set $\mathcal{U}\subset\A_2$ and assume
\begin{enumerate}[(i)]\setcounter{enumi}{3}
 \item Foliation $\F$ belongs to the set $\mathcal{U}$.
\end{enumerate}
In order to prove Corollaries \ref{coro:qvfs} and \ref{coro:equivalences} we must further assume that the characteristic numbers $\lambda_1$, $\lambda_2$, $\lambda_3$ are non--real (i.e.~the singularities on the line at infinity are hyperbolic). However, this last condition is not needed to prove Theorem \ref{thm:main}.

The genericity of conditions (ii) and (iv) is obvious. Condition (iii) also defines a complex Zariski open set in $\A_2$ (cf.~\cite{Shcherbakov1984}). Genericity of (i) is proved in \cite{Shcherbakov1984} for polynomial foliations of arbitrary degree. For quadratic vector fields we know an even stronger result:

\begin{theorem}[\cite{Pyartli2006}]
Let $\Lambda=(\lambda_1,\lambda_2,\lambda_3)$ be such that $\lambda_1+\lambda_2+\lambda_3=1$. Denote by $\mathcal{B}_\Lambda$ the set of foliations in $\A_2$ with characteristic numbers at infinity $\lambda_1,\lambda_2,\lambda_3$. Assume that $\Re{\lambda_1}\geq\Re{\lambda_2}\geq\Re{\lambda_3}$. Then, if $\lambda_1,\lambda_2\notin\frac{1}{3}\Z\cup\frac{1}{4}\Z$, there exist at least one and at most ten orbits of the group $\Aff{2}{\C}$ in $\mathcal{B}_\Lambda$ whose points correspond to equations with non--commutative solvable holonomy group at infinity.
\end{theorem}

Moreover, for any $\Lambda$, foliations in $\mathcal{B}_\Lambda$ with commutative holonomy group at infinity fall into seven families which are explicitly described in \cite{Pyartli2006}. In particular, it follows from such description (see also Theorem 1 in \cite{Pyartli2000}) that for $\Lambda=(\lambda_1,\lambda_2,\lambda_3)$ satisfying assumption (ii) above there exist exactly two orbits of the group $\Aff{2}{\C}$ in $\mathcal{B}_\Lambda$ corresponding to equations with a commutative holonomy group.

%% file: 1b-acknowledgments.tex

I would like to thank Yulij S.~Ilyashenko for his guidance, for the many discussions that led to this result and for the remarks that improved the text. My deepest gratitude to Nataliya Goncharuk and Yury Kudryashov who carefully reviewed and commented on this paper. I would specially like to thank Laura Ortiz--Bobadilla for the continuous support she provided throughout the whole process of this work.

This work was financially supported by the grants UNAM--DGAPA--PAPIIT IN 102413 and CONACYT 219722. I have also been supported by Consejo Nacional de Ciencia y Tecnolog\'{i}a, Secretar\'{i}a de Educaci\'{o}n P\'{u}blica and the Mexican Government, for which I am grateful.

%% file: 2-structure.tex

\subsection{Ideas behind the proof of Theorem \ref{thm:main}}\label{subsec:ideas}
Any foliation $\F\in\A_2$ is induced, in a neighborhood of the line at infinity $\{z=0\}$, by a rational differential equation
\begin{equation}\label{eq:ratdifeq0}
\frac{dz}{dw}=\frac{z\,P(z,w)}{Q(z,w)},
\end{equation}
such that $Q\vert_{z=0}$ is not identically zero. In fact, the roots of $r(w)=Q(0,w)$ determine the position of the singular points at infinity which from now on will be assumed, without loss of generality, to be given by $w_1=-1$, $w_2=1$ and $w_3=\infty$. Under this assumption the polynomial $r(w):=Q(0,w)$ may be chosen to be $r(w)=w^2-1$. 

In Section \ref{sec:normalizations} we will normalize the above equation using the action of the group $\Aff{2}{\C}$. This normalization was originally introduced in \cite{Pyartli2000}. Any normalized foliation is uniquely defined by five complex parameters: the characteristic numbers $\lambda_1,\lambda_2$ and three more parameters $\alpha_0,\alpha_1,\alpha_2\in\C$. We will write $\F=\F(\lambda,\alpha)$ whenever we wish to emphasize that $\F$ is defined by the parameters $\lambda=(\lambda_1,\lambda_2)$ and $\alpha=(\alpha_0,\alpha_1,\alpha_2)$. 

Let us also consider the solution $\Phi(z,w)$ of  equation (\ref{eq:ratdifeq0}) with initial condition $\Phi(z,0)=z$ and expand it as a power series in $z$ using the variations $\varphi_d$ of the solution $z=0$ in the following way:
\[ \Phi(z,w)=\sum_{d=1}^{\infty} \varphi_d (w)\,z^d. \]
The variations $\varphi_d(w)$ are defined in a neighborhood of the origin and can be analytically continued along any path on $\LF$. Moreover the holonomy map $\Delta_{\gamma}(z)$ with respect to a given loop $\gamma\in\pi_1(\LF,0)$ is given by the power series
\begin{equation}\label{eq:powerseriesDelta}
\Delta_{\gamma}(z)=\var{1}{}\,z + \var{2}{}\,z^ 2 + \ldots,
\end{equation}
where $\varphi_{d\{\gamma\}}$ denotes the analytic continuation of $\varphi_d$ along the curve $\gamma$.

Note that the fundamental group of the leaf $\LF\cong\C\setminus\{-1,1\}$ is the free group on two generators.
\begin{definition}\label{def:geometricgenerators}
 Let $\mu_1$ and $\mu_2$ be loops in $\LF$ based at the origin which go around the singular points $w=-1$ and $w=1$ respectively, once in the positive direction. We call these loops the \emph{standard geometric generators} of $\pi_1(\LF,0)$. 
\end{definition}

Now, consider the commutators
\begin{equation}\label{def:gamma_i}
 \gamma_1=\mu_2\mu_1\mu_2^{-1}\mu_1^{-1} \quad\text{and}\quad \gamma_2=\mu_2\mu_1^2\mu_2^{-1}\mu_1^{-2}, 
\end{equation}
and let $f_1,f_2$ be the holonomy maps corresponding to the above loops, this is, $f_j=\Delta_{\gamma_j}$, $j=1,2$. We call this germs \emph{distinguished parabolic germs}; they play a key role in this paper.

\begin{remark}\label{rmk:genericity(iii)}
Genericity assumption (iii) in Subsection \ref{subsec:genericity} means that the distinguished parabolic germ 
\[f_1=[\Delta_{\mu_1},\Delta_{\mu_2}]\]
has a non--zero quadratic term. 
\end{remark}

Suppose $\Ft\in\A_2$ is topologically equivalent to $\F$. The genericity assumptions imposed on these foliations imply that both $\F$ and $\Ft$ have the same characteristic numbers at infinity and so we may write $\Ft=\F(\lambda,\beta)$. Define $\tilde{f}_j$ to be the holonomy map of $\Ft$ along $\gamma_j$. The topological conjugacy gives raise to a conformal germ in $\Diff$ and a geometric automorphism of $\pi_1(\LF,0)$ which conjugate the holonomy groups. 

\begin{remark}\label{rmk:strongae}
It follows from \cite{Ramirez2014} that such geometric automorphism may be assumed to be the identity map. We therefore conclude the existence of a germ $h\in\Diff$ such that
\begin{equation}\label{eq:conjugatef-structure}
 h\circ f_j-\tilde{f}_j\circ h=0, \qquad j=1,2.
\end{equation}
Because of the above, from now on we will always assume that any given analytic conjugacy between holonomy groups is given by some germ $h\in\Diff$ and the identity automorphism of the fundamental group of $\LF$. In \cite{Ramirez2014} such a conjugacy is called \emph{strong analytic equivalence}. However, since this is the only type of conjugacy we will consider in this work, we shall not use this term. 
\end{remark}

The essence of the proof of Theorem \ref{thm:main} may be summarized as follows: If the holonomy groups of $\F$ and $\Ft$ are analytically conjugate then there exits $h\in\Diff$ such that (\ref{eq:conjugatef-structure}) holds. We can compute the first terms in the power series expansions of the distinguished parabolic germs in terms of the parameters $\lambda$, $\alpha$ and $\beta$ as explicit iterated integrals using the variation equations of the differential equation (\ref{eq:ratdifeq0}) with respect to the solution $z=0$. We also expand $h$ as a power series with unknown coefficients and substitute all these series into equation (\ref{eq:conjugatef-structure}) to obtain an expression of the form
\[ h\circ f_j-\tilde{f}_j\circ h=\sum_{d=1}^{\infty}\kappa_{d,j}\,z^d. \]
Equating each $\kappa_{d,j}$ to zero should impose some conditions on the parameter $\beta$. However, since we do not know the coefficients in the power series expansion of $h$, we must consider for each $d$ the system of equations
\begin{equation}\label{eq:systPsi}
 \kappa_{d,1}=0, \qquad \kappa_{d,2}=0.
\end{equation} 
A careful analysis of such a system will allow us to compute the coefficient of degree $d-1$ in the power series of $h$ and at the same time to obtain a concrete condition imposed on the parameter $\beta$ by (\ref{eq:systPsi}). We do this for $d=3,4,5,6$. We will first obtain  conditions imposed on $\beta$ expressed in terms of the vanishing of certain integrals. Even though these conditions are polynomial in $\beta$, the coefficient of such polynomials are transcendental functions on $\lambda$ and $\alpha$. A crucial step in the proof of Theorem \ref{thm:main} is that we are actually able to translate these conditions into algebraic ones. This is done using a Lemma \ref{lemma:Pyartli}, which is proved in \cite{Pyartli2000}. We lastly prove that for generic $\lambda$ and $\alpha$ the polynomial system of equations we obtain has a unique solution given by $\beta=\alpha$. This proves that these normalized foliations having conjugate holonomy groups are in fact one and the same. This shows in particular that two foliations, not necessarily normalized, with conjugate holonomy groups must be affine equivalent. Moreover, in order to obtain such affine map taking one foliation into the other we consider first the affine maps taking each foliation to its normal form and compose one of these maps with the inverse of the other.

The proof outlined above is carried out in a series of lemmas whose formal statements are given below.

\subsection{Three fundamental lemmas}\label{subsec:mainlemmas}

The most elaborate part of the proof of Theorem \ref{thm:main} is to obtain explicit conditions imposed on $\beta$ by the conjugacy of the holonomy groups of $\F(\lambda,\alpha)$ and $\F(\lambda,\beta)$. We do this following closely the constructions presented in \cite{Pyartli2006}.

\begin{key}\label{lemma:key}
For $d=3,4,5,6$ there exists a polynomial $P_d(w)$, whose coefficients are polynomials in $\beta$, such that the existence of a germ $h\in\Diff$ that conjugates the holonomy groups of $\F(\lambda,\alpha)$ and $\F(\lambda,\beta)$ up to jets of order $d$ implies
\begin{equation}\label{eq:keylemma}
 \int_{\gamma_1}\frac{P_d(w)}{r(w)^{d}}\,\varphi_1(w)^{d-1}\,dw=0. 
\end{equation}
\end{key} 
In the lemma above $\varphi_1(w)$ is the first variation of the solution $z=0$ of equation (\ref{eq:ratdifeq0}) and $r(w)=w^2-1$. Before proving this lemma it is necessary to obtain explicit expressions for the coefficients in the power series expansions of the distinguished parabolic germs. These computations are carried out in Section \ref{sec:analysis}.

\begin{remark}\label{rmk:dependence}
Note that the vanishing of the integral in the \hyperref[lemma:key]{Key lemma} imposes one linear condition on the coefficients of the polynomial $P_d(w)$. The polynomials $P_d(w)$ \emph{do} depend on the foliation $\F(\lambda,\alpha)$. In fact, the coefficients of these polynomials depend polynomially on $\alpha$ and rationally on $\lambda$. The main content of the next lemma is that, in virtue of Lemma \ref{lemma:Pyartli}, the linear condition imposed on the coefficients of $P_d$ by the vanishing of the integral is not trivial. This implies rightaway that such condition is a \emph{polynomial} condition on the parameters $\beta$. This is discussed in detail in Subsection \ref{subsec:keytomain}.
\end{remark}

\begin{main}\label{lemma:main}
For $d=3,4,5,6$ there exists a non--zero polynomial $F_d\in\C[\beta]$ such that the existence of a germ $h\in\Diff$ that conjugates the holonomy groups of $\F(\lambda,\alpha)$ and $\F(\lambda,\beta)$ up to jets of order $d$ implies $F_d(\beta)=0$.
\end{main}

Suppose now that $\F(\lambda,\alpha)$ and $\F(\lambda,\beta)$ have conjugate holonomy groups. The above lemma implies that $\beta\in\C^3$ satisfies the polynomial system of equations
\begin{equation}\label{eq:polysyst}
 F_3(\beta)=0,\;\ldots ,\;F_6(\beta)=0.
\end{equation}
This is a system of four equations on three variables. Generically such a system will have no solutions at all. However, because of the defining property of $F_d$ we see that $\beta=\alpha$ will always be a solution. The proof of Theorem \ref{thm:main} is completed by the following lemma.

\begin{elimination}\label{lemma:elimination}
There exists a dense Zariski open set $U\subset\C^5$ such that if $(\lambda,\alpha)\in U$ then the polynomial system \textnormal{(\ref{eq:polysyst})} has a unique solution given by $\beta=\alpha$.
\end{elimination}

\subsection{Two lemmas about integrals}\label{subsec:lemmasintegrals}

The following lemmas were proved and used by Pyartli in \cite{Pyartli2000} and \cite{Pyartli2006}. They play a major role in our proof and will be used frequently.

Recall that $\gamma_1$ and $\gamma_2$ have been defined to be the commutators $\gamma_1=\mu_2\mu_1\mu_2^{-1}\mu_1^{-1}$ and $\gamma_2=\mu_2\mu_1^2\mu_2^{-1}\mu_1^{-2}$ where $\mu_1$, $\mu_2$ are standard geometric generators of the fundamental group of the punctured line $\C\setminus\{1,-1\}$.

\begin{lemma}\label{lemma:integrals1}
Let $P(w)$ be a polynomial and let $\zeta(w)=(1+w)^{u_1}(1-w)^{u_2}$ where $u_1,u_2$ are complex numbers and $\zeta(0)=1$. Then
\[ \int_{\gamma_2}P(w)\zeta(w)\,dw = (1+\exp{(2\pi i\,u_1)})\int_{\gamma_1}P(w)\zeta(w)\,dw. \]
\end{lemma}

The proof of this lemma is straightforward: we decompose the loops $\gamma_1,\gamma_2$ into pieces and write down each integral as a sum of integrals along these pieces to verify that the equality holds. 

The next lemma is the fundamental step for deducing the \hyperref[lemma:main]{Main lemma} from the \hyperref[lemma:key]{Key lemma}.

\begin{lemma}\label{lemma:Pyartli}
Let $\zeta(w)=(1+w)^{u_1}(1-w)^{u_2}$, $\zeta(0)=1$, $u_1,u_2\notin\Z$, $r(w)=w^2-1$ and $P(w)$ a polynomial of degree at most $m$. The equality $\int_{\gamma_1}P(w)\zeta(w)\,dw=0$ holds if and only if there exists a polynomial $R(w)$ of degree at most $\max{(m-1,-2-\operatorname{Re}{(u_1+u_2)})}$ and a constant $C\in\C$ such that
\[ \int_0^w P(t)\zeta(t)\,dt=R(w)r(w)\zeta(w)+C. \]
\end{lemma}

In this paper we will only use the above lemma in the case $m-1>-2-\operatorname{Re}{(u_1+u_2)}$ so that, if it exists, $R_d(w)$ will have degree at most $m-1$. Note that both the vanishing of the integral and the existence of $R(w)$ impose one linear condition on the coefficients of the polynomial $P(w)$. Clearly the existence of such an $R$ implies the vanishing of the integral since we are integrating along the commutator loop $\gamma_1$ and so $\zeta_{\{\gamma_1\}}(0)=\zeta(0)=1$. We conclude that both linear conditions are equivalent. A detailed proof can be found in \cite{Pyartli2000}.

Recall that we have numbered the singular points at infinity of $\F$ in such a way that $\Re{\lambda_1}\geq\Re{\lambda_2}\geq\Re{\lambda_3}$. It follows from the fact that $\lambda_1+\lambda_2+\lambda_3=1$ that
\begin{equation}\label{eq:inequalityRelambda}
\Re{\lambda_1}+\Re{\lambda_2}\geq 2/3,
\end{equation}
This remark will be frequently used as a complement to Lemma \ref{lemma:Pyartli}. Indeed, in Section \ref{sec:keylemma} we will apply Lemma \ref{lemma:Pyartli} to integrals of the form (\ref{eq:keylemma}) taking $u_i=(d-1)\lambda_i-d$, for $d=3,4,5,6$. This is one of the instances where it is important that genericity assumption $\lambda_i\notin\frac{1}{3}\Z\cup\frac{1}{4}\Z\cup\frac{1}{5}\Z$ holds.

%% file: 3-sketch.tex

\subsection{Key lemma: the strategy}\label{subsec:keylemma}

Suppose there exists a germ $h\in\Diff$ that conjugates the holonomy groups of $\F=\F(\lambda,\alpha)$ and $\Ft=\F(\lambda,\beta)$. We expand the distinguished parabolic germs in power series
\begin{equation}
  f_j(z)=z+a_{2j}z^2+a_{3j}z^3+\ldots, \quad \tilde{f}_j(z)=z+\tilde{a}_{2j}z^2+\tilde{a}_{3j}z^3+\ldots, \quad j=1,2,
\end{equation}
as well as the germ $h$,
\[ h(z)=h_1z+h_2z^2+h_3z^3+\ldots. \]
Note that the first variations satisfy $\varphi_1=\tilde{\varphi}_1$, since these functions are completely determined by $\lambda$. Throughout this work we will omit the tilde on $\tilde{\varphi}_1$.

The coefficients $a_{dj}$ are computed in Section \ref{sec:analysis} in terms of the parameters $\lambda$ and $\alpha$. In particular, it will be shown that
\begin{equation}\label{eq:a_2-v0}
 a_{2j}=\tilde{a}_{2j}=\int_{\gamma_j}\frac{1}{r(t)}\,\varphi_1(t)\,dt, \quad j=1,2.
\end{equation}

The \hyperref[lemma:key]{Key lemma} for degree $d=3$ will be easily deduced from the fact that equation (\ref{eq:a_2-v0}) holds, which in turn is a direct consequence of the particular normal form (\ref{eq:normalform}) that we shall be using. In particular it will be shown that the equality $a_{2j}=\tilde{a}_{2j}$ forces the germ $h$ to be parabolic; that is, $h_1=1$. The \hyperref[lemma:key]{Key lemma} for all higher degrees is proved following a strategy which we now present.

Suppose we have computed all the coefficients $h_2,...,h_{d-2}$ in terms of $\lambda,\alpha,\beta$. Since the germs $f_j$, $\tilde{f}_j$ and $h$ are parabolic, the coefficient of degree $d$ in the power series expansion of $h\circ f_j-\tilde{f}_j\circ h$ is of the form 
\begin{equation}\label{eq:sketchKL1}
\frac{1}{d!}(h\circ f_j-\tilde{f}_j\circ h)^{(d)}(0)=(h_d+a_{dj})-(\tilde{a}_{dj}+h_d)+\ldots =a_{dj}-\tilde{a}_{dj}+\ldots, 
\end{equation}
where the multiple dots denote those terms that depend only on $a_{kj}$, $\tilde{a}_{kj}$ and $h_k$ with $k<d$. Since $h\circ f_j-\tilde{f}_j\circ h=0$, the above equation yields an expression for $\tilde{a}_{dj}-a_{dj}$ in terms of $a_{kj}$, $\tilde{a}_{kj}$, $h_k$, $k=2,...,d-1$. On the other hand, we have explicit formulas for the coefficients $a_{dj}$, and thus for $\tilde{a}_{dj}-a_{dj}$, from Section \ref{sec:analysis} (cf.~Propositions \ref{prop:secondvar} to \ref{prop:sixthvar}). We equate this formula for $\tilde{a}_{dj}-a_{dj}$ to the formula we deduced from (\ref{eq:sketchKL1}). This method yields an equation involving the index $j$ and thus by making $j=1$ and $j=2$ we obtain a system of two equations. A priori, it is not at all clear what conditions this system of equations imposes on the parameter $\beta$. The fundamental fact about this system, proved throughout Section \ref{sec:keylemma}, is that it can be simplified to take the form
\[ a_{2j}\,\mathcal{C}_d+\mathcal{I}_{dj}=0, \quad j=1,2, \]
where $a_{2j}$ is as in (\ref{eq:a_2-v0}), $\mathcal{C}_d$ is an expression involving the coefficients $h_2,\ldots,h_{d-1}$ that does not depend on the index $j$, $\mathcal{I}_{dj}=\int_{\gamma_j}\frac{P_d}{r^d}\,\varphi_1^{d-1}\,dw$ and $P_d$ is a polynomial which will be computed explicitly. The \hyperref[lemma:key]{Key lemma} for degree $d$ is completed by the following proposition.

\begin{proposition}\label{prop:key}
 Let $d\geq3$. If $\lambda_1\notin\frac{1}{d-2}\Z$ and the polynomial $P_d(w)$ satisfies a system of equations of the form 
\begin{align}
 a_{21}\,\mathcal{C}_d+\mathcal{I}_{d1} &= 0 \nonumber \\
 a_{22}\,\mathcal{C}_d+\mathcal{I}_{d2} &= 0  \label{eq:sketchKL2}
\end{align}
where $\mathcal{C}_d$ is a complex number,
\begin{equation}\label{eq:defId}
 \mathcal{I}_{dj}=\int_{\gamma_j}\frac{P_d(w)}{r(w)^d}\,\varphi_1(w)^{d-1}\,dw, 
\end{equation}
and $a_{2j}$ is as in \textnormal{(\ref{eq:a_2-v0})} then 
\[ \mathcal{C}_d=\mathcal{I}_{dj}=0. \]
\end{proposition}

\begin{proof}
We can regard (\ref{eq:sketchKL2}) as a linear system on three unknowns: $\mathcal{C}_d$, $\mathcal{I}_{d1}$ and $\mathcal{I}_{d2}$. Note that the integrand that appears in (\ref{eq:defId}) can be rewritten as $P_d(w)\zeta_d(w)$, where
\[\zeta_d(w)=\frac{1}{r(w)^d}\varphi_1(w)^{d-1}=(1+w)^{(d-1)\lambda_1-d}(1-w)^{(d-1)\lambda_2-d}, \]
since $\varphi_1(w)=(1+w)^{\lambda_1}(1-w)^{\lambda_2}$. Applying Lemma \ref{lemma:integrals1} we can express $\mathcal{I}_{d2}$ as a scalar multiple of $\mathcal{I}_{d1}$,
\[ \mathcal{I}_{d2}=(1+\nu_1^{d-1})\,\mathcal{I}_{d1}, \quad \nu_1=\exp{(2\pi i\lambda_1)}. \]
Since $a_{2j}$ is given in terms of the integral in (\ref{eq:a_2-v0}), Lemma \ref{lemma:integrals1} also implies that
\[  a_{22}=(1+\nu_1)\,a_{21}. \]
In this way system (\ref{eq:sketchKL2}) becomes
\begin{align}\label{eq:sketchKL3}
 a_{21}\,\mathcal{C}_d+\mathcal{I}_{d1} &= 0, \nonumber \\
 (1+\nu_1)\,a_{21}\,\mathcal{C}_d+(1+\nu_1^{d-1})\,\mathcal{I}_{d1} &= 0,
\end{align}
whose unknowns are $\mathcal{C}_d$ and $\mathcal{I}_{d1}$. The determinant of this system is
\[\begin{vmatrix}
    a_{21} & 1\\
    (1+\nu_1)\,a_{21} & 1+\nu_1^{d-1}
  \end{vmatrix} = a_{21}\nu_1(\nu_1^{d-2}-1), \]
which is not zero. Indeed, $\nu_1^{d-2}\neq 1$ since $\nu_1=\exp{(2\pi i\lambda_1)}$ and $\lambda_1\notin\frac{1}{d-2}\Z$, and by our genericity assumptions $a_{21}\neq 0$. This implies that $\mathcal{I}_{d1}=0$ and $\mathcal{C}_d=0$. 
\end{proof}

Note that the fact that $\mathcal{I}_{d1}=0$ proves the \hyperref[lemma:key]{Key lemma} for degree $d$ since the expression for $\mathcal{I}_d$ given in (\ref{eq:defId}) coincides with the lefthand side of (\ref{eq:keylemma}) in the \hyperref[lemma:key]{Key lemma}. On the other hand, $\mathcal{C}_d$ is given in terms of $h_2,\ldots h_{d-1}$ and so the fact that $\mathcal{C}_d=0$ allows us to find an expression for the coefficient $h_{d-1}$. In this way we are able to repeat the process now for degree $d+1$. That is, at every step $d$ we will prove the \hyperref[lemma:key]{Key lemma} for degree $d$ and compute $h_{d-1}$.

\subsection{Deducing Main lemma from Key lemma}\label{subsec:keytomain}

As pointed out in Remark \ref{rmk:dependence}, the equation
\begin{equation}\label{eq:keylemmabis}
\int_{\gamma_1}\frac{P_d}{r^d}\varphi_1^{d-1}\,dw=0
\end{equation}
imposes one linear condition on the coefficients of the polynomial $P_d(w)$. Since these coefficients are polynomials on $\beta$, we need only prove that this linear condition is non--trivial to conclude the \hyperref[lemma:main]{Main Lemma}. We prove this fact using Lemma \ref{lemma:Pyartli}. Indeed, Lemma \ref{lemma:Pyartli} claims that equation (\ref{eq:keylemmabis}) is equivalent to the existence of a polynomial $R_d(w)$ such that
\[ \int_0^w \frac{P_d}{r^d}\varphi_1^{d-1}\,dt=\frac{R_d(w)}{r(w)^{d-1}}\,\varphi_1(w)^{d-1}+C. \]
This means that
\[ \left(\frac{R_d(w)}{r(w)^{d-1}}\,\varphi_1(w)^{d-1}\right)^{\prime}=\frac{P_d(w)}{r(w)^d}\,\varphi_1(w)^{d-1}, \]
on the other hand a short computation shows that (cf.~the variation equation (\ref{eq:vareq1}) in Section \ref{sec:normalizations}),
\[ \left(\frac{R_d(w)}{r(w)^{d-1}}\,\varphi_1(w)^{d-1}\right)^{\prime}=\frac{R'_d(w)r(w)+(d-1)(s(w)-r'(w))R_d(w)}{r(w)^{d}}\,\varphi_1(w)^{d-1}, \]
where $s(w)=\lambda_1(w-1)+\lambda_2(w+1)$. This implies that
\begin{equation}\label{eq:P_dR_d}
 P_d=R'_dr+(d-1)(s-r')R_d. 
\end{equation}
We will see in Subsection \ref{subsec:mainlemmarevisited} that the polynomials $P_d$ have degree $2(d-1)$ and that $
\deg{R_d}\leq\deg{P_d}-1$. This fact, together with equation (\ref{eq:P_dR_d}), implies that the linear condition imposed on the coefficients of $P_d(w)$ by equation (\ref{eq:keylemmabis}) is non--trivial. The \hyperref[lemma:main]{Main Lemma} now follows immediately.

\begin{remark}
In Subsection \ref{subsec:computingFd} we will explain how to obtain explicit expressions for the polynomials $R_d(w)$ and $F_d(\beta)$ in terms of the coefficients of the polynomials $P_d(w)$. These will be later needed in order to prove the \hyperref[lemma:elimination]{Elimination Lemma}.
\end{remark}

\subsection{The Elimination lemma}\label{subsec:eliminationlemma}

The last step in the proof of Theorem \ref{thm:main} is to prove that the system
\[F_3(\beta)=0,\,\ldots\, ,F_6(\beta)=0, \]
has no solutions other than $\beta=\alpha$. This is done taking resultants of the polynomials $F_d$ with respect to successive variables $\beta_2,\beta_1,\beta_0$. Consider for the time being the parameters $\lambda$, $\alpha$ to be fixed, thus the coefficients of of the polynomials $F_d$ are also fixed complex numbers. 

Recall that if $f(x)=a_0x^n+\ldots+a_n$ and $g(x)=b_0x^m+\ldots+b_m$ are polynomials in $x$ with coefficients in some field $\mathbb{F}$, the resultant of $f$ and $g$ is defined to be
\[\Res_x(f(x),g(x))=a_0^m b_0^n \prod_{i,j}(u_i-v_j), \]
where $u_i$ and $v_j$ are the roots of $f(x)$ and $g(x)$, respectively, in $\overline{\mathbb{F}}$. The resultant can be defined for polynomials over any commutative ring. Over an integral domain it has the fundamental property that $\Res_x(f(x),g(x))=0$ if and only if $f(x)$ and $g(x)$ have a common factor of positive degree.

We will first take several resultants of the polynomials $F_d$ with respect to $\beta_2$. Second, we take resultants of these previously obtained resultants with respect to $\beta_1$. The final step has a twist; if we take now a last resultant with respect to $\beta_0$ we are guaranteed to get 0, since $\beta=\alpha$ is a solution to system (\ref{eq:polysyst}). We avoid this by dividing one of these resultants by the linear polynomial $\beta_0-\alpha_0$. More precisely, let us define
\begin{align*}
 \operatorname{Res}^1_j(\beta_0,\beta_1) &= \Res_{\beta_2}\big(F_3(\beta_0,\beta_1,\beta_2),F_j(\beta_0,\beta_1,\beta_2)\big), & j &= 4,5,6, \\
 \operatorname{Res}^2_j(\beta_0) &= \Res_{\beta_1}\big(\operatorname{Res}^1_4(\beta_0,\beta_1),\operatorname{Res}^1_j(\beta_0,\beta_1)\big), & j &= 5,6, \\
 \operatorname{Res}^3_6   &= \Res_{\beta_0}\big(\operatorname{Res}^2_5(\beta_0)/(\beta_0-\alpha_0),\operatorname{Res}^2_6(\beta_0)\big). & &
\end{align*}
Note that as long as we fix $\alpha$ and $\lambda$ we have that
\[ \operatorname{Res}^1_j\in\C[\beta_0,\beta_1],\quad \operatorname{Res}_j^2\in\C[\beta_0],\quad \operatorname{Res}^3_6\in\C. \]

\begin{proposition}\label{prop:elimination1}
 If $\operatorname{Res}^3_6\neq0$ then any solution $(u_0,u_1,u_2)$ of the polynomial system \textnormal{(\ref{eq:polysyst})} satisfies $u_0=\alpha_0$.
\end{proposition}

\begin{proof}
 Suppose on the contrary that $\operatorname{Res}^3_6\neq 0$ but $(u_0,u_1,u_2)$ is a solution of (\ref{eq:polysyst}) such that $u_0\neq\alpha_0$. Note that $F_3(u_0,u_1,\beta_2)$ and $F_j(u_0,u_1,\beta_2)$ have a common root $\beta_2=u_2$ for any $j=4,5,6$ and so
\[0=\Res_{\beta_2}\big(F_3(u_0,u_1,\beta_2),F_j(u_0,u_1,\beta_2)\big)=\operatorname{Res}^1_j(u_0,u_1), \qquad j=4,5,6. \]
In particular $\operatorname{Res}^1_4(u_0,\beta_1)$ has a common root, $\beta_1=u_1$, with both $\operatorname{Res}_5^1(u_0,\beta_1)$ and $\operatorname{Res}_6^1(u_0,\beta_1)$. We deduce that $\operatorname{Res}_5^2(u_0)=0$ and $\operatorname{Res}_6^2(u_0)=0$. Now, since $u_0\neq \alpha_0$ it is still true that $\operatorname{Res}_5^2(\beta_0)/(\beta_0-\alpha_0)$ and $\operatorname{Res}_6^2(\beta_0)$ have $\beta_0=u_0$ as a common root; in particular, $\operatorname{Res}^3_6=0$, a contradiction.
\end{proof}

We would like to be able to guarantee that $\operatorname{Res}^3_6$ is never zero, no matter the choice of $\lambda$ and $\alpha$. This need not be true. However, we can guarantee that for \emph{almost every} choice of $\lambda$ and $\alpha$ the resultant $\operatorname{Res}^3_6$ is not zero. Indeed, as mentioned in Remark \ref{rmk:dependence}, the coefficients of the polynomials $F_d$ depend polynomially on $\alpha$ and rationally on $\lambda$. In this way, if we allow $\alpha$ and $\lambda$ to vary, the coefficients of $F_d$ belong to the ring $\C(\lambda)[\alpha]$, in particular $\operatorname{Res}^3_6\in\C(\lambda)[\alpha]$. Let us thus introduce the notation $\operatorname{Res}^3_6(\lambda,\alpha)$. If $\operatorname{Res}^3_6(\lambda,\alpha)$ is not identically zero then the union of its divisors of zeroes and poles defines a proper algebraic subset of affine space $\C^5$. The complement $U$ of this algebraic set is a Zariski--open subset of $\C^5$ with the property that for any $(\lambda,\alpha)\in U$ we have $\operatorname{Res}^3_6(\lambda,\alpha)\neq 0$. Finally we will prove that $\operatorname{Res}^3_6(\lambda,\alpha)\not\equiv 0$ by exhibiting an explicit point $(\lambda,\alpha)\in\C^5$, given in (\ref{values}), for which $\operatorname{Res}^3_6$ does not vanish. 

The above argument shows that if $\F=\F(\lambda,\alpha)$ and $\Ft=\F(\lambda,\beta)$ have conjugate monodromy groups, then we must have $\alpha_0=\beta_0$. The polynomial $F_3(\beta)$ is linear and $F_4(\beta)$ is linear on $\beta_1,\beta_2$ yet quadratic on $\beta_0$. However, if we replace $\beta_0$ by $\alpha_0$ we obtain a linear system on $\beta_1,\beta_2$ (this is verified by direct inspection of the polynomials $F_3$ and $F_4$ which can be found in the \hyperref[sec:appendix]{Appendix}). The proof of the \hyperref[lemma:elimination]{Elimination lemma} is completed by the following proposition.

\begin{proposition}\label{prop:elimination2}
The pair of equations
\begin{equation}\label{eq:elimination2}
F_3(\alpha_0,\beta_1,\beta_2)=0,\qquad F_4(\alpha_0,\beta_1,\beta_2)=0,
\end{equation}
forms a linear inhomogeneous system on $\beta_1$ and $\beta_2$. Its determinant is a non--zero element of $\C(\lambda)[\alpha]$ and therefore for almost every $(\lambda,\alpha)\in\C^5$ the system has a unique solution which is necessarily given by 
\[ \beta_1=\alpha_1, \qquad \beta_2=\alpha_2. \]
\end{proposition}

The proof of this proposition is discussed in Subsection \ref{subsec:concludingelimination}. Propositions \ref{prop:elimination1} and \ref{prop:elimination2} together imply the \hyperref[lemma:elimination]{Elimination lemma}.

\begin{remark}
In the proof of the \hyperref[lemma:main]{Main lemma} and \hyperref[lemma:key]{Key lemma} all computations are carried out in terms of the rational functions $K_d(w)$ defined by the formula
\[ \frac{dz}{dw}=\frac{z\,P(z,w)}{Q(z,w)}=\sum_{d=1}^{\infty}K_d(w)\,z^d, \]
whose explicit dependence on $(\lambda,\alpha)$ is not provided until Section \ref{sec:elimination}. This has been done to avoid excessively large expressions and make the proof more transparent. However, in order to prove the \hyperref[lemma:elimination]{Elimination lemma} (more precisely, that the final resultant $\operatorname{Res}^3_6$ and the determinant of (\ref{eq:elimination2}) do not vanish identically) we do need to compute expressions for the polynomials $F_d$ in terms of the parameters $\lambda,\alpha,\beta$ explicitly. Obtaining these expressions and evaluating the resultant $\operatorname{Res}^3_6$ and the determinant of (\ref{eq:elimination2}) at a particular point has been done with computer assistance using \emph{Mathematica 9} \cite{Mathematica9}. This procedure is discussed in Section \ref{sec:elimination} and the program script is presented in the \hyperref[sec:appendix]{Appendix}.
\end{remark}

%% file: 4-normalizations.tex

A foliation $\F\in\A_2$ has three singular points at infinity. These can be brought to any other three different points on the infinite line by the action of the affine group of $\C^2$. We wish to normalize a foliation in such a way that the singular points are given by $w_1=-1$, $w_2=1$ and $w_3=\infty$ in coordinates $(z,w)=(1/x,\,y/x)$. If the characteristic numbers are pairwise different we can do this unambiguously by numbering the singular points in such a way that $\Re{\lambda_1}\geq\Re{\lambda_2}\geq\Re{\lambda_3}$ and if $\Re{\lambda_i}=\Re{\lambda_j}$ then $\Im{\lambda_i}>\Im{\lambda_j}$ provided $i<j$. 

Since the characteristic numbers are not integer numbers it follows from \cite{Pyartli2000} that we can find an affine change of coordinates such that in the chart $(z,w)$ the foliation is induced by 
\[ \frac{dz}{dw}=z\,\frac{s(w)(1+\alpha_0 z)+\kappa z+\eta z^2}{r(w)(1+\alpha_0\sigma z)+p(w)z^2}, \]
where $r(w)=w^2-1$, $s(w)=\lambda_1(w-1)+\lambda_2(w+1)$, $p(w)=\alpha_1(w-1)+\alpha_2(w+1)$, $\sigma=\lambda_1+\lambda_2$ and $\eta=\alpha_1+\alpha_2$.

It follows from \cite{Pyartli2006} that if $\lambda_1,\lambda_2\notin\Z$ then the parameter $\kappa$ above is non--zero, provided that the germ $f_1$ constructed in Definition \ref{def:geometricgenerators} as the commutator of the holonomy maps along the standard geometric generators has non--zero quadratic part. Moreover, if $\kappa\neq 0$ we can further normalize the above equation in such a way that $\kappa=1$. By one of our genericity hypotheses, $f_1$ has a non--trivial quadratic part and moreover this property also holds for any foliation whose holonomy group is analytically conjugate to that of $\F$. Therefore all foliations considered in this work may be normalized in such a way that $\kappa=1$. We arrive to the following normal form:
\begin{equation}\label{eq:normalform}
 \frac{dz}{dw}=z\,\frac{s(w)(1+\alpha_0 z)+z+\eta z^2}{r(w)(1+\alpha_0\sigma z)+p(w)z^2}
\end{equation}
In this way any generic foliation $\F\in\A_2$ is uniquely defined by five complex parameters: $\lambda_1$, $\lambda_2$, $\alpha_0$, $\alpha_1$, $\alpha_2$. We write $\F=\F(\lambda,\alpha)$ to emphasize this fact. In what follows $\Ft$ will denote a foliation from $\A_2$ whose holonomy group at infinity is analytically equivalent to that of $\F$. We deduce from such conjugacy and from the non--solvability of the holonomy groups that $\Ft$ has the same characteristic numbers at infinity. Therefore we may write $\Ft=\F(\lambda,\beta)$ where $\beta\in\C^3$.

Let us denote the right hand side of (\ref{eq:normalform}) by $\Psi(z,w)$. The rational function $\Psi$ has a power series expansion with respect to $z$ of the form
\begin{equation}\label{eq:expansionF}
\Psi(z,w)=\sum_{d=1}^{\infty}K_d(w)\,z^d,
\end{equation}
where $K_d$ is a rational function in $w$. Since $\Psi(0,w)$ has denominator $r(w)$ we can expect that the rational functions $K_d(w)$ to have $r(w)$ to some power as denominator. We will see in Proposition \ref{prop:c_dS_d} that this is in fact the case and that moreover such power can always be taken to be equal to $d$.

In particular the first coefficient $K_1(w)$ is the rational function
\[ K_1(w)=\frac{s(w)}{r(w)}=\frac{\lambda_1}{w+1}+\frac{\lambda_2}{w-1}. \]
The first variation of the solution $z=0$ to equation (\ref{eq:normalform}) satisfies the linear equation
\begin{equation}\label{eq:vareq1}
 \frac{d\varphi_1}{dw}=K_1(w)\,\varphi_1, \quad \varphi_1(0)=1,
\end{equation}
and so $\varphi_1(w)=(1+w)^{\lambda_1}(1-w)^{\lambda_2}$. The higher variations $\varphi_d$, $d\geq 2$, satisfy an inhomogeneous linear equation whose associate homogeneous equation is \textnormal{(\ref{eq:vareq1})}:
\[ \frac{d\varphi_d}{dw}=K_1(w)\,\varphi_d+b_d(w), \quad \varphi_d(0)=0. \]
Let us write $B_d(t)=\varphi_1(t)^{-1}\,b_d(t)$ so that the solution to the above equation is given by
\[ \varphi_d(w)=\varphi_1(w)\int_{0}^{w}B_d(t)\,dt. \]
Let us define $\phi_d(w)=\int_{0}^{w}B_d(t)\,dt$ and call these functions the \emph{reduced variations}. In this way $\varphi_d=\varphi_1\phi_d$. The non--linear terms $b_d(w)$ are well known for an equation of the form (\ref{eq:ratdifeq0}). The following proposition gives an explicit expression for $B_d=\varphi_1^{-1}\,b_d$.

\begin{proposition}\label{prop:formulasB_d}
The functions $B_d$ defined above are given by the following formulas:
\begin{align*}
B_2 & = K_2\varphi_1, \\
B_3 & = 2K_2\phi_2\varphi_1 + K_3\varphi_1^2, \\
B_4 & = K_2(2\phi_3\varphi_1+\phi_2^2\varphi_1) + 3K_3\phi_2\varphi_1^2 + K_4\varphi_1^3, \\
B_5 & = 2K_2(\phi_4\varphi_1+\phi_3\phi_2\varphi_1) + 3K_3(\phi_3\varphi_1^2+\phi_2^2\varphi_1^2) + 4K_4\phi_2\varphi_1^3 + K_5\varphi_1^4, \\
B_6 & = K_2(2\phi_5\varphi_1+2\phi_4\phi_2\varphi_1+\phi_3^2\varphi_1) + K_3(3\phi_4\varphi_1^2+6\phi_3\phi_2\varphi_1^2+\phi_2^3\varphi_1^2)\\ 
    &\phantom{=}+K_4(4\phi_3\varphi_1^3+6\phi_2^2\varphi_1^3) + 5K_5\phi_2\varphi_1^4 + K_6\varphi_1^5.
\end{align*}
\end{proposition}

In order to compute the reduced variations $\phi_d(w)=\int_0^wB_d\,dt$ it will be convenient to split each of the rational functions $K_d(w)$ into two pieces, one of these a scalar multiple of $K_1(w)$. Computations are simplified since, in virtue of (\ref{eq:vareq1}), we can compute explicitly an integral of the form $\int_{0}^{w}K_1\varphi_1^m\,dt$.

\begin{definition}\label{def:splittingF}
Given a rational differential equation $\displaystyle\frac{dz}{dw}=\Psi(z,w)$ normalized as in \textnormal{(\ref{eq:normalform})} we define the rational function
\[ C(z,w)=z\frac{s(w)(1+\alpha_0 z)}{r(w)(1+\alpha_0\sigma z)}, \]
where $s(w),r(w),\sigma$ are as in \textnormal{(\ref{eq:normalform})}. We also define $S(z,w)$ by the formula 
\begin{equation}\label{eq:splittingF}
\Psi(z,w)=C(z,w)+S(z,w). 
\end{equation}
\end{definition}

\begin{remark}
It is proved in \cite{Pyartli2000} that a foliation given by 
\[ \frac{dz}{dw}=C(z,w), \]
with $C(z,w)$ as above has a commutative holonomy group. This holonomy group is in fact linearizable but it is not linear unless $\alpha_0=0$.
\end{remark}

Note that
\[ C(z,w)=K_1(w)\vartheta(z), \]
where $\vartheta(z)$ is the rational function $\vartheta(z)=z(1+\alpha_0 z)(1+\alpha_0\sigma z)^{-1}$. 

\begin{proposition}\label{prop:c_dS_d}
The splitting of $\Psi(z,w)$ given in equation \textnormal{(\ref{eq:splittingF})} implies that for each $d\geq1$,
\begin{equation}\label{eq:c_dS_d}
K_d(w)=c_d\,K_1(w)+\frac{S_d(w)}{r(w)^d},
\end{equation}
where the polynomials $S_d(w)$ are given by the formula
\[ S(z,w)=\sum_{d=2}^{\infty}\frac{S_d(w)}{r(w)^{d}}\,z^d, \]
and the constants $c_d$ are given by $\vartheta(z)=\sum_{d=1}^\infty c_dz^d.$
\end{proposition}

Explicit expressions for $c_d$ and $S_d$ in terms of the parameters $\lambda$ and $\alpha$ are given at the beginning of Section \ref{sec:elimination}.

\begin{remark}\label{rmk:a_dj=phi_d} 
We have expanded the distinguished parabolic germs in power series
\[ f_j(z)=z+a_{2j}z^2+a_{3j}z^3+\ldots. \]
According to (\ref{eq:powerseriesDelta}) we have $a_{dj}=\var{d}{j}$, and we also know that $\var{1}{j}=1$ since the loops $\gamma_1$, $\gamma_2$ are commutators. The equality $\varphi_d=\varphi_1\phi_d$ implies that in fact
\[ a_{dj}=\phi_{d\{\gamma_j\}}(0). \]
This fact will be used in the next section when computing the coefficients $a_{dj}$.
\end{remark}

%% file: 5-analysis.tex

In this section we compute the coefficients $a_{dj}$ in the power series expansion of the distinguished parabolic germ $f_j$. These computations follow very closely computations carried out in \cite{Pyartli2006}. However, in \cite{Pyartli2006} it is assumed that the holonomy group at infinity of the foliation in question is solvable, and thus several simplifications take place. The computations provided here are completely general.

\subsection{Analysis of the terms of low degree}

\begin{proposition}\label{prop:secondvar}
The reduced second variation is given by
\[ \phi_2(w)=c_2(\varphi_1(w)-1)+\psi_2(w), \]
where 
\[ \psi_2(w)=\int_{0}^{w}\frac{S_2(t)}{r(t)^2}\,\varphi_1(t)\,dt, \]
and $c_2$, $S_2$ are as in Proposition \textnormal{\ref{prop:c_dS_d}}. In particular we have
\[ a_{2j}=\psi_{2j} \quad\text{with}\quad \psi_{2j}=\int_{\gamma_j}\frac{S_2(w)}{r(w)^2}\,\varphi_1(w)\,dw, \quad j=1,2. \]
\end{proposition}

\begin{proof}
The reduced variation is given by $\phi_2(w)=\int_0^wB_2\,dt$. It follows from Proposition \ref{prop:formulasB_d} and equation (\ref{eq:c_dS_d}) that
\[ \phi_2(w)=\int_{0}^{w} K_2(t)\varphi_1(t)\,dt=\int_{0}^{w} \left(c_2K_1(t)+\frac{S_2(t)}{r(t)^2}\right)\varphi_1(t)\,dt. \]
Note that
\[ \int_{0}^{w}K_1(t)\varphi_1(t)\,dt=\int_0^w\frac{d\varphi_1}{dt}\,dt=\varphi_1(w)-1, \]
and so
\[ \phi_2(w)=c_2(\varphi_1(w)-1)+\int_{0}^{w}\frac{S_2(t)}{r(t)^2}\,\varphi_1(t)\,dt, \]
as claimed.
\end{proof}

\begin{proposition}\label{prop:thirdvar}
The reduced third variation is given by
\[ \phi_3(w)=\phi_2(w)^2\varphi_1(w)+c_3\frac{\varphi_1(w)^2-1}{2}+\psi_3(w), \]
where 
\[ \psi_3(w)=\int_0^w\frac{S_3(t)}{r(t)^3}\varphi_1(t)^2\,dt, \]
and $c_3$, $S_3$ are as in Proposition \textnormal{\ref{prop:c_dS_d}}. In particular 
\[ a_{3j}=a_{2j}^2+\psi_{3j} \quad\text{with}\quad \psi_{3j}=\int_{\gamma_j}\frac{S_3(w)}{r(w)^3}\varphi_1(w)^2\,dw, \quad j=1,2. \]
\end{proposition}

\begin{proof}
By Proposition \ref{prop:formulasB_d} we have that $\phi_3$ is given by
\[ \int_0^w B_3\,dt=2\int_0^w K_2\phi_2\varphi_1\,dt+\int_0^w K_3\varphi_1^2\,dt. \]
The first integral on the right--hand side can be easily computed:
\[ \int_0^w K_2\phi_2\varphi_1\,dt=\int_0^w B_2\phi_2\,dt=\int_0^w\frac{d\phi_2}{dt}\phi_2\,dt=\frac{1}{2}\phi_2^2. \]
For the second integral we split $K_3$ according to (\ref{eq:c_dS_d}):
\[ \int_0^w K_3\varphi_1^2\,dt=c_3\int_0^wK_1\varphi_1^2\,dt+\int_0^w\frac{S_3}{r^3}\varphi_1^2\,dt=c_3\frac{\varphi_1^2-1}{2}+\psi_3. \]
Adding up both integrals gives the desired result.
\end{proof}

\subsection{Analysis of the terms of higher degree}

For degrees higher than the third we shall not need an explicit expression for the reduced variation $\phi_d(w)$, so we focus only on the coefficient $a_{dj}=\phi_{d\{\gamma_j\}}(0)$. 

We stress that for any given exponent $n$ we have
\[ \int_{\gamma_j}K_1\varphi_1^n\,dw=0, \quad j=1,2, \]
since $\frac{d\varphi_1}{dw}=K_1\varphi_1$ and $\varphi_{1\{\gamma_j\}}(0)=\varphi_1(0)=1$.

\begin{proposition}\label{prop:fourthvar}
The coefficient of degree $4$ in the power series expansion of $f_j$ is given by
\[ a_{4j}=2a_{3j}a_{2j}-a_{2j}^3+\frac{c_3}{2}a_{2j}-c_2\psi_{3j}+\Delta_{1j}+\psi_{4j}, \]
where
\[ \Delta_{1j}=\int_{\gamma_j}\frac{S_3(w)}{r(w)^3}\,\psi_2(w)\varphi_1(w)^2\,dw, \quad \psi_{4j}=\int_{\gamma_j}\frac{q_4(w)}{r(w)^4}\,\varphi_1(w)^3\,dw,   \]
and the polynomial $q_4(w)$ is defined to be
\[ q_4(w)=S_4(w)+c_2S_3(w)r(w)-\frac{c_3}{2}S_2(w)r(w)^2, \]
with the terms $c_d$, $S_d$ as in Proposition \textnormal{\ref{prop:c_dS_d}}.
\end{proposition}

\begin{proof}
By Proposition \ref{prop:formulasB_d} we know that $a_{4j}$ is given by
\begin{equation}\label{eq:phi_4-v1}
 a_{4j}=\int_{\gamma_j} B_4\,dw = 2H_{1j}+H_{2j}+3H_{3j}+H_{4j}, 
\end{equation}
where 
\[ H_{1j}=\int_{\gamma_j} K_2\phi_3\varphi_1\,dw,\quad H_{2j}=\int_{\gamma_j} K_2\phi_2^2\varphi_1\,dw,\quad H_{3j}=\int_{\gamma_j} K_3\phi_2\varphi_1^2\,dw,\quad H_{4j}=\int_{\gamma_j} K_4\varphi_1^3\,dw. \]
We now proceed to compute these integrals. It is straight forward that $H_{2j}=\int_{\gamma_j} B_2\phi_2^2\,dw$, hence
\[ H_{2j}=\int_{\gamma_j}\left(\frac{1}{3}\phi_2^3\right)^{\prime}\,dw=\frac{1}{3}a_{2j}^3. \] 
Note that  $H_{1j}=\int_{\gamma_j} B_2\phi_3\,dw$, so integration by parts yields
\[ \int_{\gamma_j}\frac{d\phi_2}{dt}\phi_3\,dw = a_{3j}a_{2j}-\int_{\gamma_j} B_3\phi_2\,dw. \]
Using the expression for $B_3$ provided in Proposition \ref{prop:formulasB_d} we see that
\[ H_{1j}=a_{3j}a_{2j}-2\int_{\gamma_j}K_2\varphi_1\phi_2^2\,dw-\int_{\gamma_j}K_3\phi_2\varphi_1^2\,dw=a_{3j}a_{2j}-\frac{2}{3}a_2^3-H_{3j}. \]
Equation (\ref{eq:phi_4-v1}) becomes 
\begin{equation}\label{eq:phi_4-v2}
a_{4j}=2a_{3j}a_{2j}-a_{2j}^3+H_{3j}+H_{4j}. 
\end{equation}

We split $K_3$ using equation (\ref{eq:c_dS_d}), thus
\begin{equation}\label{eq:I3-v0}
 H_{3j}=\int_{\gamma_j}c_3K_1\phi_2\varphi_1^2\,dw + \int_{\gamma_j}\frac{S_3}{r^3}\phi_2\varphi_1^2\,dw. 
\end{equation}
By Proposition \ref{prop:secondvar} we have $\phi_2=c_2(\varphi_1-1)+\psi_2$ therefore the first of the integrals above is given by
\[ \int_{\gamma_j}c_3K_1( c_2(\varphi_1-1)+\psi_2)\varphi_1^2\,dw=\int_{\gamma_j}c_3K_1\psi_2\varphi_1^2\,dw, \]
since 
\[ \int_{\gamma_j}K_1(\varphi_1-1)\varphi_1^2\,dw=0. \]
Note that integration by parts yields
\[ c_3\int_{\gamma_j}K_1\psi_2\varphi_1^2\,dw=c_3\int_{\gamma_j}\left(\frac{1}{2}\varphi_1^2\right)^{\prime} \psi_2\,dw=\frac{c_3}{2}a_{2j}-\int_{\gamma_j}\frac{\frac{c_3}{2}S_2}{r^2}\varphi_1^3\,dw. \]
On the other hand, the last integral in (\ref{eq:I3-v0}) is given by
\begin{align*} 
 \int_{\gamma_j}\frac{S_3}{r^3}( c_2(\varphi_1-1)+\psi_2)\varphi_1^2\,dw &= \int_{\gamma_j}\frac{c_2S_3}{r^3}\varphi_1^3\,dw-c_2\int_{\gamma_j}\frac{S_3}{r^3}\varphi_1^2\,dw+\int_{\gamma_j}\frac{S_3}{r^3}\psi_2\varphi_1^2\,dw. \\
								  &= \int_{\gamma_j}\frac{c_2S_3}{r^3}\varphi_1^3\,dw-c_2\psi_3+\Delta_{1j}.
\end{align*}
Therefore
\[ H_{3j}= \frac{c_3}{2}a_{2j}-\int_{\gamma_j}\frac{\frac{c_3}{2}S_2}{r^2}\varphi_1^3\,dw+\int_{\gamma_j}\frac{c_2S_3}{r^3}\varphi_1^3\,dw-c_2\psi_3+\Delta_{1j}. \]

Lastly, splitting $K_4$ according to equation (\ref{eq:c_dS_d}) we get
\[ H_{4j}=\int_{\gamma_j}\left(c_4K_1+\frac{S_4}{r^4}\right)\varphi_1^3\,dw=\int_{\gamma_j}\frac{S_4}{r^4}\varphi_1^3\,dw. \]

Substituting the above expressions for $H_{3j}$ and $H_{4j}$ in (\ref{eq:phi_4-v2}) and taking into account that we have defined $q_4=S_4+c_2S_3r-\frac{c_3}{2}S_2r^2$ we obtain the desired expression for $a_{4j}$.
\end{proof}

\begin{proposition}\label{prop:fifthvar}
The coefficient of degree $5$ in the power series expansion of $f_j$ is given by
\begin{align*}
a_{5j}	&= 2a_{4j}a_{2j}+\frac{3}{2}a_{3j}^2-4a_{3j}a_{2j}^2+\frac{3}{2}a_{2j}^4+ \frac{c_3}{2}a_{2j}^2+\frac{2c_4-c_3c_2}{3}a_{2j} \\
	&\phantom{=} +\,c_2^2\psi_{3j}-2c_2\psi_{4j}-2c_2\Delta_{1j}+\Delta_{2j}+2\Gamma_{1j}+\psi_{5j}, 
\end{align*}
where 
\[ \Delta_{2j} = \int_{\gamma_j}\frac{S_3(w)}{r(w)^3}\,\psi_2(w)^2\varphi_1(w)^2\,dw, \qquad \Gamma_{1j} = \int_{\gamma_j}\frac{q_4(w)}{r(w)^4}\,\psi_2(w)\varphi_1(w)^3\,dw, \]
\[\psi_{5j} = \int_{\gamma_j}\frac{q_5(w)}{r(w)^5}\,\varphi_1(w)^4\,dw, \]
and the polynomial $q_5(w)$ is defined to be
\[ q_5=S_5+2c_2S_4r+c_2^2S_3r^2-\frac{2}{3}(c_4+c_3c_2)S_2r^3  \]
with the terms $c_d$, $S_d$ as in Proposition \textnormal{\ref{prop:c_dS_d}}.
\end{proposition}

\begin{proof}
According to Proposition \ref{prop:formulasB_d} we know that $a_{5j}$ is given by
\begin{equation}\label{eq:a5-pt0} 
\phi_{5\{\gamma_j\}}(0)=\int_{\gamma_j}B_5\,dw = 2I_{1j}+2I_{2j}+3I_{3j}+3I_{4j}+4I_{5j}+I_{6j}, 
\end{equation}
where 
\begin{align*}
 I_{1j}&=\int_{\gamma_j}K_2\phi_4\varphi_1\,dw, 	& I_{2j}&=\int_{\gamma_j}K_2\phi_3\phi_2\varphi_1\,dw, 	& I_{3j}&=\int_{\gamma_j}K_3\phi_3\varphi_1^2\,dw, \\
 I_{4j}&=\int_{\gamma_j}K_3\phi_2^2\varphi_1^2\,dw, 	& I_{5j}&=\int_{\gamma_j}K_4\phi_2\varphi_1^3\,dw, 	& I_{6j}&=\int_{\gamma_j}K_5\varphi_1^4\,dw.
\end{align*}
The first integrals are computed as follows:
\begin{align*}
I_{2j}	&= \int_{\gamma_j}\left(\frac{1}{2}\phi_2^2\right)^\prime\phi_3\,dw=\frac{1}{2}a_{2j}^2a_{3j}-\frac{1}{2}\int_{\gamma_j}B_3\phi_2^2\,dw \\
	&= \frac{1}{2}a_{3j}a_{2j}^2-\frac{1}{2}\int_{\gamma_j}2K_2\phi_2^3\varphi_1\,dw-\frac{1}{2}\int_{\gamma_j}K_3\phi_2^2\varphi_1^2\,dw \\
	&= \frac{1}{2}a_{3j}a_{2j}^2-\frac{1}{4}a_{2j}^4-\frac{1}{2}I_{4j}.
\end{align*}
\begin{align*}
I_{3j}	&= \int_{\gamma_j}(K_3\varphi_1^2)\phi_3\,dw=\int_{\gamma_j}(B_3-2K_2\phi_2\varphi_1)\phi_3\,dw=\frac{1}{2}a_{3j}^2-2I_{2j} \\
	&= \frac{1}{2}a_{3j}^2-a_{3j}a_{2j}^2+\frac{1}{2}a_{2j}^4+I_{4j}.
\end{align*}
\begin{align*}
I_{1j}	&= \int_{\gamma_j}\frac{d\phi_2}{dw}\phi_4\,dw=a_{4j}a_{2j}-\int_{\gamma_j}B_4\phi_2\,dw \\
	&= a_{4j}a_{2j}-2\int_{\gamma_j}K_2\phi_3\phi_2\varphi_1\,dw-\int_{\gamma_j}K_2\phi_2^3\varphi_1\,dw-3\int_{\gamma_j}K_3\phi_2^2\varphi_1^2\,dw-\int_{\gamma_j}K_4\phi_2\varphi_1^3\,dw \\
	&= a_{4j}a_{2j}-2I_{2j}-\frac{1}{4}a_{2j}^4-3I_{4j}-I_{5j}\\
	&= a_{4j}a_{2j}-a_{3j}a_{2j}^2+\frac{1}{4}a_{2j}^4-2I_{4j}-I_{5j}.
\end{align*}
Therefore equation (\ref{eq:a5-pt0}) becomes
\begin{equation}\label{eq:a5-pt1}
a_{5j} = 2a_{4j}a_{2j}+\frac{3}{2}a_{3j}^2-4a_{3j}a_{2j}^2+\frac{3}{2}a_{2j}^4+I_{4j}+2I_{5j}+I_{6j}.
\end{equation}

Next we break $K_3$ according to (\ref{eq:c_dS_d}), so $I_{4j}=\int_{\gamma_j}\left(\frac{S_3}{r^3}+c_3K_1\right)\phi_2^2\varphi_1^2\,dw$. Now, using Proposition \ref{prop:secondvar} we get
\begin{align*} 
\int_{\gamma_j}\frac{S_3}{r^3}\phi_2^2\varphi_1^2\,dw	&= \int_{\gamma_j}\frac{S_3}{r^3}(c_2(\varphi_1-1)+\psi_2)^2\varphi_1^2\,dw \\
							&= \int_{\gamma_j}\frac{S_3}{r^3}(c_2^2(\varphi_1^4-2\varphi_1^3+\varphi_1^2)+2c_2(\varphi_1^3-\varphi_2^2)\psi_2+\psi_2^2\varphi_1^2)\,dw.
\end{align*}
Let us group under a same integral those terms having the same exponent on $\varphi_1$,
\begin{align}\label{eq:a5-I4-pt1}
\int_{\gamma_j}\frac{S_3}{r^3}\phi_2^2\varphi_1^2\,dw	&= \int_{\gamma_j}\frac{c_2^2S_3}{r^3}\varphi_1^4\,dw + \int_{\gamma_j}\frac{-2c_2^2S_3+2c_2S_3\psi_2}{r^3}\varphi_1^3\,dw \nonumber \\
							&\phantom{=} + \int_{\gamma_j}\frac{c_2^2S_3-2c_2S_3\psi_2+S_3\psi_2^2}{r^3}\varphi_1^2\,dw.
\end{align}
On the other hand,
\[ \int_{\gamma_j}c_3K_1\phi_2^2\varphi_1^2\,dw=c_3\int_{\gamma_j}\left(\frac{1}{2}\varphi_1^2\right)^{\prime}\phi_2^2\,dw=\frac{c_3}{2}a_{2j}^2-c_3\int_{\gamma_j}B_2\phi_2\varphi_1^2\,dw. \]
The last integral above is given by $\int_{\gamma_j}K_2\phi_2\varphi_1^3\,dw=\int_{\gamma_j}\left(\frac{S_2}{r^2}+c_2K_1\right)\phi_2\varphi_1^3\,dw$.
\[ \int_{\gamma_j}\frac{S_2}{r^2}\phi_2\varphi_1^3\,dw = \int_{\gamma_j}\frac{S_2}{r^2}(c_2(\varphi_1-1)+\psi_2)\varphi_1^3\,dw = \int_{\gamma_j}\frac{c_2S_2}{r^2}\varphi_1^4\,dw+\int_{\gamma_j}\frac{-c_2S_2+S_2\psi_2}{r^2}\varphi_1^3\,dw \]
and
\[ \int_{\gamma_j}c_2K_1\phi_2\varphi_1^3\,dw = \int_{\gamma_j}c_2K_1(c_2(\varphi_1-1)+\psi_2)\varphi_1^3\,dw = c_2\int_{\gamma_j}K_1\psi_2\varphi_1^3\,dw, \]
since $\int_{\gamma_j}K_1(\varphi_1^4-\varphi_1^3)\,dw=0$. The integral on the right--hand side above can be integrated by parts to obtain
\[ \int_{\gamma_j}\left(\frac{1}{3}\varphi_1^3\right)^{\prime}\psi_2\,dw=\frac{1}{3}a_{2j}-\frac{1}{3}\int_{\gamma_j}\frac{S_2}{r^2}\varphi_1^4\,dw. \]
We conclude that
\begin{equation}\label{eq:a5-I4-pt2}
\int_{\gamma_j}c_3K_1\phi_2^2\varphi_1^2\,dw=\frac{c_3}{2}a_{2j}^2-\frac{c_3c_2}{3}a_{2j}+\int_{\gamma_j}\frac{c_3c_2S_2-c_3S_2\psi_2}{r^2}\varphi_1^3\,dw+\int_{\gamma_j}\frac{-\frac{2}{3}c_3c_2S_2}{r^2}\varphi_1^4\,dw.
\end{equation}

We proceed in a similar way to compute $2I_{5j}=2\int_{\gamma_j}\left(\frac{S_4}{r^4}+c_4K_1\right)\phi_2\varphi_1^3\,dw$. By Proposition \ref{prop:secondvar}, $\phi_2=c_2(\varphi_1-1)+\psi_2$, so
\begin{equation}\label{eq:a5-I5-pt1}
2\int_{\gamma_j}\frac{S_4}{r^4}\phi_2\varphi_1^3\,dw=2\int_{\gamma_j}\frac{c_2S_4}{r^4}\varphi_1^4\,dw+2\int_{\gamma_j}\frac{-c_2S_4+S_4\psi_2}{r^4}\varphi_1^3\,dw. 
\end{equation}
On the other hand, 
\begin{align}\label{eq:a5-I5-pt2}
2\int_{\gamma_j}c_4K_1\phi_2\varphi_1^3\,dw	&= 2c_4\int_{\gamma_j}\left(\frac{1}{3}\varphi_1^3\right)^{\prime}\phi_2\,dw=\frac{2c_4}{3}a_{2j}-\frac{2c_4}{3}\int_{\gamma_j}B_2\varphi_1^3\,dw \nonumber \\
						&= \frac{2c_4}{3}a_{2j}+\int_{\gamma_j}\frac{-\frac{2}{3}c_4S_2}{r^2}\varphi_1^4\,dw, 
\end{align}
since $B_2=\left(\frac{S_2}{r^2}+c_2K_1\right)\varphi_1$ and $\int_{\gamma_j}K_1\varphi_1^4\,dw=0$.

Lastly, note that writing $K_5=\frac{S_5}{r^5}+c_5K_1$ immediately yields
\begin{equation}\label{eq:a5-I6}
I_{6j}=\int_{\gamma_j}\frac{S_5}{r^5}\varphi_1^4\,dw.
\end{equation}

The formula claimed for $a_{5j}$ is obtained by combining equations (\ref{eq:a5-pt1}) to (\ref{eq:a5-I6}). Indeed, substituting in (\ref{eq:a5-pt1}) the expressions found in (\ref{eq:a5-I4-pt1}) -- (\ref{eq:a5-I6}) yields
\[ a_{5j}=2a_{4j}a_{2j}+\frac{3}{2}a_{3j}^2-4a_{3j}a_{2j}^2+\frac{3}{2}a_{2j}^4+ \frac{c_3}{2}a_{2j}^2+\frac{2c_4-c_3c_2}{3}a_{2j}+E_{2j}+E_{3j}+E_{4j}, \]
where we have grouped all integrals containing $\varphi_1$ to the $k$--th power in a single integral $E_{kj}$ given by the following expressions:
\begin{align*}
E_{2j}	&= \int_{\gamma_j}\frac{c_2^2S_3-2c_2S_3\psi_2+S_3\psi_2^2}{r^3}\varphi_1^2\,dw = c_2^2\psi_{3j}-2c_2\Delta_{1j}+\Delta_{2j}, \\[4pt]
E_{3j}	&= \int_{\gamma_j}\frac{-2c_2^2S_3r+c_3c_2S_2r^2-2c_2S_4+(2c_2S_3r-c_3S_2r^2+2S_4)\psi_2}{r^4}\varphi_1^3\,dw\\[4pt]
	&= \int_{\gamma_j}\frac{-c_2q_4+2q_4\psi_2}{r^4}\varphi_1^3\,dw = -2c_2\psi_{4j}+2\Gamma_{1j}, \\[4pt]
E_{4j}	&= \int_{\gamma_j}\frac{c_2^2S_3r^2-\frac{2}{3}c_3c_2S_2r^3+2c_2S_4r-\frac{2}{3}c_4S_2r^3+S_5}{r^5}\varphi_1^4\,dw\\[4pt]
	&= \int_{\gamma_j}\frac{q_5}{r^5}\varphi_1^4\,dw = \psi_{5j}.
\end{align*}
This is exactly the expression claimed by Proposition \ref{prop:fifthvar}.
\end{proof}

\begin{proposition}\label{prop:sixthvar}
The coefficient of degree $6$ in the power series expansion of $f_j$ is given by
\begin{align*}
a_{6j} 	&=	2a_{5j}a_{2j}+3a_{4j}a_{3j}-4a_{4j}a_{2j}^2-5a_{3j}^2a_{2j}+7a_{3j}a_{2j}^3-2a_{2j}^5 \\
&\phantom{=}	+\frac{c_3}{2}a_{2j}^3+\left(c_4-\frac{c_3c_2}{2}\right)a_{2j}^2+\left(\frac{3c_5}{4}-\frac{c_4c_2}{2}-\frac{c_3^2}{8}+\frac{c_3c_2^2}{4}+\frac{c_3}{2}\psi_{3j}\right)a_{2j} \\
&\phantom{=}	-\frac{c_2}{2}\psi_{3j}^2+\left(\frac{c_4}{3}+\frac{c_3c_2}{3}-c_2^3\right)\psi_{3j}+\left(-\frac{c_3}{2}+3c_2^2\right)\Delta_{1j}-3c_2\Delta_{2j}+\Delta_{3j}+\Delta_{(1,1)j} \\
&\phantom{=}	+\left(-\frac{c_3}{2}+3c_2^2\right)\psi_{4j}-6c_2\Gamma_{1j}+3\Gamma_{2j}+\Gamma_{(0,1)j}-3c_2\psi_{5j}+3\mathrm{B}_{1j}+\psi_{6j}.
\end{align*}
where
\begin{align*}
 \Delta_{3j} &= \int_{\gamma_j}\frac{S_3}{r^3}\,\psi_2^3\varphi_1^2\,dw, & 
 \Delta_{(1,1)j} &= \int_{\gamma_j}\frac{S_3}{r^3}\,\psi_2\psi_3\varphi_1^2\,dw, \\
 \Gamma_{2j} &= \int_{\gamma_j}\frac{q_4}{r^4}\,\psi_2^2\varphi_1^3\,dw, & 
 \Gamma_{(0,1)j} &= \int_{\gamma_j}\frac{q_4}{r^4}\,\psi_3\varphi_1^3\,dw, \\
 \mathrm{B}_{1j} &= \int_{\gamma_j}\frac{q_5}{r^5}\,\psi_2\varphi_1^4\,dw, &
 \psi_{6j} &= \int_{\gamma_j}\frac{q_6}{r^6}\,\varphi_1^5\,dw, 
\end{align*}
and the polynomial $q_6(w)$ is defined to be
\begin{align*}
q_6 	&= 	S_6+3c_2S_5r+\left(\frac{c_3}{2}+3c_2^2\right)S_4r^2+\left(-\frac{c_4}{3}+\frac{c_3c_2}{6}+c_2^3\right)S_3r^3 \\
&\phantom{=}	+\left(-\frac{3c_5}{4}-\frac{3c_4c_2}{2}-\frac{c_3^2}{8}-\frac{3c_3c_2^2}{4}\right)S_2r^4,
\end{align*}
with the terms $c_d$, $S_d$ as in Proposition \textnormal{\ref{prop:c_dS_d}}.
\end{proposition}

\begin{proof}
According to Proposition \ref{prop:formulasB_d}, $a_{6j}$ is given by
\[ \phi_{6\{\gamma_j\}}(0)=\int_{\gamma_j}B_6\,dw=2J_{1j}+2J_{2j}+J_{3j}+3J_{4j}+6J_{5j}+J_{6j}+4J_{7j}+6J_{8j}+5J_{9j}+J_{10j}, \]
where
\begin{align*}
J_{1j}&=\int_{\gamma_j}K_2\phi_5\varphi_1\,dw, &
J_{2j}&=\int_{\gamma_j}K_2\phi_4\phi_2\varphi_1\,dw, & 
J_{3j}&=\int_{\gamma_j}K_2\phi_3^2\varphi_1\,dw, \\
J_{4j}&=\int_{\gamma_j}K_3\phi_4\varphi_1^2\,dw, &
J_{5j}&=\int_{\gamma_j}K_3\phi_3\phi_2\varphi_1^2\,dw, &
J_{6j}&=\int_{\gamma_j}K_3\phi_2^3\varphi_1^2\,dw, \\
J_{7j}&=\int_{\gamma_j}K_4\phi_3\varphi_1^3\,dw, &
J_{8j}&=\int_{\gamma_j}K_4\phi_2^2\varphi_1^3\,dw, &
J_{9j}&=\int_{\gamma_j}K_5\phi_2\varphi_1^4\,dw, \\
J_{10j}&=\int_{\gamma_j}K_6\varphi_1^5\,dw. & & 
\end{align*}

Let us compute some of these integrals. First of all let us define $J_{0j}=\int_{\gamma_j}K_2\phi_3\phi_2^2\varphi_1\,dw$. Taking into account the expression for $B_4$ presented in Proposition \ref{prop:formulasB_d}, we have
\begin{align*}
J_{2j}	&= \int_{\gamma_j}\left(\frac{1}{2}\phi_2^2\right)^{\prime}\phi_4\,dw=\frac{1}{2}a_{4j}a_{2j}^2-\frac{1}{2}\int_{\gamma_j}B_4\phi_2^2\,dw \\
	&= \frac{1}{2}a_{4j}a_{2j}^2-J_{0j}-\frac{1}{10}a_{2j}^5-\frac{3}{2}J_{6j}-\frac{1}{2}J_{8j}.
\end{align*}
Similarly, taking into account the expression found for $B_5$ we obtain
\begin{align*}
J_{1j} 	&= \int_{\gamma_j}B_2\phi_5\,dw=a_{5j}a_{2j}-\int_{\gamma_j}B_5\phi_2\,dw \\
	&= a_{5j}a_{2j}-2J_{2j}-2J_{0j}-3J_{5j}-3J_{6j}-4I_{8j}-J_{9j} \\
	&= a_{5j}a_{2j}-a_{4j}a_{2j}^2+\frac{1}{5}a_{2j}^5-3J_{5j}-3J_{8j}-J_{9j}.
\end{align*}
We also have
\[ J_{3j}=\int_{\gamma_j}B_2\phi_3^2\,dw=a_{3j}^2a_{2j}-2\int_{\gamma_j}B_3\phi_3\phi_2\,dw=a_{3j}^2a_{2j}-4J_{0j}-2J_{5j}, \]
and
\begin{align*}
J_{4j}	&= \int_{\gamma_j}(B_3-2K_2\phi_2\varphi_1)\phi_4\,dw=a_{4j}a_{3j}-\int_{\gamma_j}B_4\phi_3\,dw-2J_{2j} \\
	&= a_{4j}a_{3j}-2J_{3j}-J_{0j}-3J_{5j}-J_{7j}-2J_{2j}.
\end{align*}
Taking into account the expressions for $J_{3j}$ and $J_{2j}$ above we obtain
\[ J_{4j}=a_{4j}a_{3j}-a_{4j}a_{2j}^2-2a_{3j}^2a_{2j}+\frac{1}{5}a_{2j}^5+9J_{0j}+J_{5j}+3J_{6j}-J_{7j}+J_{8j}. \]

We conclude that 
\begin{align*} 
a_{6j} &=	2a_{5j}a_{2j}+3a_{4j}a_{3j}-4a_{4j}a_{2j}^2-5a_{3j}^2a_{2j}+\frac{4}{5}a_{2j}^5 \\
&\phantom{=}	+21J_{0j}+J_{5j}+7J_{6j}+J_{7j}+2J_{8j}+3J_{9j}+J_{10j}.
\end{align*}
Note that $J_{0j}=\int_{\gamma_j}\left(\frac{1}{3}\phi_2^3\right)^{\prime}\phi_3\,dw=\frac{1}{3}a_{3j}a_{2j}^3-\frac{1}{3}\int_{\gamma_j}B_3\phi_2^3\,dw$ and $\int_{\gamma_j}B_3\phi_2^3\,dw=\frac{2}{5}a_{2j}^5+J_{6j}$. This shows that
\[ J_{0j}=\frac{1}{3}a_{3j}a_{2j}^3-\frac{2}{15}a_{2j}^5-\frac{1}{3}J_{6j}. \]
We arrive to the following formula for $a_{6j}$,
\begin{align}\label{eq:a6-v1}
a_{6j}	&= 2a_{5j}a_{2j}+3a_{4j}a_{3j}-4a_{4j}a_{2j}^2-5a_{3j}^2a_{2j}+7a_{3j}a_{2j}^3-2a_{2j}^5\nonumber \\
&\phantom{=}	+J_{5j}+J_{7j}+2J_{8j}+3J_{9j}+J_{10j}.
\end{align}

Let us now compute $J_{5j}=\int_{\gamma_j}K_3\phi_3\phi_2\varphi_1^2\,dw$. We split $K_3$ according to (\ref{eq:c_dS_d}) and write $J_{5j}=J_{5j}^{(1)}+J_{5j}^{(2)}$, where
\[ J_{5j}^{(1)}=\int_{\gamma_j}\frac{S_3}{r^3}\phi_3\phi_2\varphi_1^2\,dw, \quad J_{5j}^{(2)}=\int_{\gamma_j}c_3K_1\phi_3\phi_2\varphi_1^2\,dw. \]
Note first that using Proposition \ref{prop:secondvar} and Proposition \ref{prop:thirdvar} we can write $\phi_3\phi_2\varphi_1^2$ as
\[ \big( c_2^2(\varphi_1^2-2\varphi_1+1)+\frac{1}{2}c_3(\varphi_1^2-1)+2c_2(\varphi_1-1)\psi_2+ \psi_2^2+\psi_3\big)\big( c_2(\varphi_1-1)+\psi_2\big)\varphi_1^2, \]
therefore we obtain
\begin{align}
\phi_3\phi_2\varphi_1^2 &= c_2^3(\varphi_1^5-3\varphi_1^4+3\varphi_1^3-\varphi_1^2)+\frac{1}{2}c_3c_2(\varphi_1^5-\varphi_1^4-\varphi_1^3+\varphi_1^2) \nonumber \\
 &\phantom{=} +3c_2^2(\varphi_1^4-2\varphi_1^3+\varphi_1^2)\psi_2+\frac{1}{2}c_3(\varphi_1^4-\varphi_1^2)\psi_2 +3c_2(\varphi_1^3-\varphi_1^2)\psi_2^2 \label{eq:J5v0} \\
&\phantom{=} +\psi_2^3\varphi_1^2 +c_2(\varphi_1^3-\varphi_1^2)\psi_3 +\psi_3\psi_2\varphi_1^2. \nonumber 
\end{align}
We substitute the above expression for $\phi_3\phi_2\varphi_1^2$ in $J_{5j}^{(1)}$ and regroup under the same integral those terms having the same power of $\varphi_1$ to obtain
\begin{align}
 J_{5j}^{(1)} &= \int_{\gamma_j}\frac{(c_2^3+\frac{1}{2}c_3c_2)S_3}{r^3}\varphi_1^5\,dw +\int_{\gamma_j}\frac{(-3c_2^3-\frac{1}{2}c_3c_2)S_3+(3c_2^2+\frac{1}{2}c_3)S_3\psi_2}{r^3}\,\varphi_1^4\,dw \nonumber \\
&\phantom{=} +\int_{\gamma_j}\frac{(3c_2^3-\frac{1}{2}c_3c_2)S_3-6c_2^2S_3\psi_2+3c_2S_3\psi_2^2+c_2S_3\psi_3}{r^3}\varphi_1^3\,dw \label{eq:J5(1)v0} \\
&\phantom{=} +\int_{\gamma_j}\frac{(-c_2^3+\frac{1}{2}c_3c_2)S_3+(3c_2^2-\frac{1}{2}c_3)S_3\psi_2-3c_2S_3\psi_2^2+S_3\psi_2^3-c_2S_3\psi_3+S_3\psi_3\psi_2}{r^3}\varphi_1^2\,dw. \nonumber
\end{align}
We shall simplify only one of the above terms: Note that 
\[ \int_{\gamma_j}\frac{-c_2S_3\psi_3}{r^3}\varphi_1^2\,dw =-c_2\int_{\gamma_j}\frac{d\psi_3}{dw}\psi_3\,dw=-\frac{1}{2}c_2\psi_{3j}^2.\]
We thus obtain
\begin{align}
 J_{5j}^{(1)} &= -\frac{1}{2}c_2\psi_{3j}^2 + \int_{\gamma_j}\frac{(c_2^3+\frac{1}{2}c_3c_2)S_3}{r^3}\varphi_1^5\,dw +\int_{\gamma_j}\frac{(-3c_2^3-\frac{1}{2}c_3c_2)S_3+(3c_2^2+\frac{1}{2}c_3)S_3\psi_2}{r^3}\,\varphi_1^4\,dw \nonumber \\
&\phantom{=} +\int_{\gamma_j}\frac{(3c_2^3-\frac{1}{2}c_3c_2)S_3-6c_2^2S_3\psi_2+3c_2S_3\psi_2^2+c_2S_3\psi_3}{r^3}\varphi_1^3\,dw \label{eq:J5(1)} \\
&\phantom{=} +\int_{\gamma_j}\frac{(-c_2^3+\frac{1}{2}c_3c_2)S_3+(3c_2^2-\frac{1}{2}c_3)S_3\psi_2-3c_2S_3\psi_2^2+S_3\psi_2^3+S_3\psi_3\psi_2}{r^3}\varphi_1^2\,dw. \nonumber
\end{align}

For computing $J_{5j}^{(2)}=\int_{\gamma_j}c_3K_1\phi_3\phi_2\varphi_1^2\,dw$ we also substitute the expression for $\phi_3\phi_2\varphi_1^2$ found in (\ref{eq:J5v0}). Splitting the integral into individual terms, we get expressions of the form $\int_{\gamma_j}K_1\psi_3^s\psi_2^t\varphi_1^k\,dw$. For each of these terms we use one of the following integration by parts formulas:
\begin{align}
\int_{\gamma_j}K_1\psi_2^s\varphi_1^k\,dw &= \frac{1}{k}a_{2j}^s-\frac{s}{k}\int_{\gamma_j} \frac{S_2}{r^2}\psi_2^{s-1}\varphi_1^{k+1}\,dw, \nonumber \\
\int_{\gamma_j}K_1\psi_3^s\varphi_1^k\,dw &= \frac{1}{k}\psi_{3j}^s-\frac{s}{k}\int_{\gamma_j} \frac{S_3}{r^3}\psi_3^{s-1}\varphi_1^{k+2}\,dw, \label{eq:intbypartsformulas} \\
\int_{\gamma_j}K_1\psi_3\psi_2\varphi_1^2\,dw &= \frac{1}{2}\psi_{3j}a_{2j}-\frac{1}{2}\int_{\gamma_j}\frac{S_3}{r^3}\psi_2\varphi_1^4\,dw -\frac{1}{2}\int_{\gamma_j}\frac{S_2}{r^2}\psi_3\varphi_1^3\,dw, \nonumber
\end{align}
or the fact that $\int_{\gamma_j}K_1\varphi_1^k\,dw=0$. After regrouping we obtain an expression
\begin{align}
 J_{5j}^{(2)} &= \frac{1}{2}c_3a_{2j}^3-\frac{1}{2}c_3c_2a_{2j}^2 +(-\frac{1}{8}c_3^2+\frac{1}{4}c_3c_2^2)a_{2j} +\frac{1}{2}c_3\psi_{3j}a_{2j}-\frac{1}{6}c_3c_2\psi_{3j} \nonumber \\
&\phantom{=} +\int_{\gamma_j}\frac{(-\frac{1}{8}c_3^2-\frac{3}{4}c_3c_2^2)S_2r-\frac{1}{3}c_3c_2S_3}{r^3}\,\varphi_1^5\,dw, \nonumber \\
&\phantom{=} +\int_{\gamma_j}\frac{2c_3c_2^2S_2r+\frac{1}{2}c_3c_2S_3-2c_3c_2S_2\psi_2r-\frac{1}{2}c_3S_3\psi_2}{r^3}\,\varphi_1^4\,dw, \label{eq:J5(2)} \\
&\phantom{=} +\int_{\gamma_j}\frac{(\frac{1}{4}c_3^2-\frac{3}{2}c_3c_2^2)S_2+3c_3c_2S_2\psi_2-\frac{3}{2}c_3S_2\psi_2^2-\frac{1}{2}c_3S_2\psi_3}{r^2}\,\varphi_1^3\,dw. \nonumber
\end{align}

Note that by Proposition \ref{prop:thirdvar}, $J_{7j}=\int_{\gamma_j}K_4(\phi_2^2+\frac{1}{2}c_3(\varphi_1^2-1)+\psi_3)\varphi_1^3\,dw $. Regrouping we get $J_{7j}=\int_{\gamma_j}K_4(\phi_2^2\varphi_1^3-\frac{1}{2}c_3\varphi_1^3+\psi_3\varphi_1^3+\frac{1}{2}c_3\varphi_1^5)\,dw$. Since the integral $J_{8j}$ is defined to be $\int_{\gamma_j}K_4\phi_2^2\varphi_1^3\,dw$ we see that 
\[ J_{7j}=J_{8j}+\int_{\gamma_j}\left(\frac{S_4}{r^4}+c_4K_1\right)\left(-\frac{1}{2}c_3\varphi_1^3+\psi_3\varphi_1^3+\frac{1}{2}c_3\varphi_1^5\right)\,dw. \]
Expanding the above product and using the integration by parts formula (\ref{eq:intbypartsformulas}) we obtain
\begin{equation}\label{eq:J7}
 J_{7j}=J_{8j}+\frac{1}{3}c_4\psi_{3j}+\int_{\gamma_j}\frac{\frac{1}{2}c_3S_4-\frac{1}{3}c_4S_3r}{r^4}\,\varphi_1^5\,dw + \int_{\gamma_j}\frac{-\frac{1}{2}c_3S_4+S_4\psi_3}{r^4}\,\varphi_1^3\,dw.
\end{equation}

Now, let us also split $J_{8j}=\int_{\gamma_j}K_4\phi_2^2\varphi_1^3\,dw$ as $J_{8j}=J_{8j}^{(1)}+J_{8j}^{(2)}$ with
\[ J_{8j}^{(1)}=\int_{\gamma_j}\frac{S_4}{r^4}\phi_2^2\varphi_1^3\,dw, \quad 
   J_{8j}^{(2)}=\int_{\gamma_j}c_4K_1\phi_2^2\varphi_1^3\,dw. \]
Expanding and substituting $\phi_2^2=\big(c_2(\varphi_1-1)+\psi_2\big)^2$ into the above expressions we obtain
\begin{align}
 J_{8j}^{(1)} &= 
\int_{\gamma_j}\frac{c_2^2S_4}{r^4}\,\varphi_1^5\,dw + \int_{\gamma_j}\frac{-2c_2^2S_4+2c_2S_4\psi_2}{r^4}\,\varphi_1^4\,dw + \int_{\gamma_j}\frac{c_2^2S_4-2c_2S_4\psi_2+S_4\psi_2^2}{r^4}\,\varphi_1^3\,dw, \label{eq:J8(1)} \\
 J_{8j}^{(2)} &= 
\frac{1}{3}c_4a_{2j}^2-\frac{1}{6}c_4c_2a_{2j}+\int_{\gamma_j}\frac{-\frac{1}{2}c_4c_2S_2}{r^2}\,\varphi_1^5\,dw+\int_{\gamma_j}\frac{\frac{2}{3}c_4c_2S_2-\frac{2}{3}c_4S_2\psi_2}{r^2}\,\varphi_1^4\,dw. \label{eq:J8(2)}
\end{align}

For the last integrals $J_{9j}$ and $J_{10j}$ we proceed in an analogous way. We obtain
\begin{align}
 J_{9j}&=\frac{1}{4}c_5a_{2j}+\int_{\gamma_j}\frac{c_2S_5-\frac{1}{4}c_5S_2r^3}{r^5}\,\varphi_1^5\,dw+\int_{\gamma_j}\frac{-c_2S_5+S_5\psi_2}{r^5}\,\varphi_1^4\,dw, \label{eq:J9} \\
 J_{10j}&=\int_{\gamma_j}\frac{S_6}{r^6}\,\varphi_1^5\,dw. \label{eq:J10}
\end{align}

If we now substitute in (\ref{eq:a6-v1}) the expressions we have found for $J_{5j},\ldots,J_{10j}$ given by equations (\ref{eq:J5(1)}) to (\ref{eq:J10}), we obtain
\begin{align*}
 a_{6j} &= 
2a_{5j}a_{2j}+3a_{4j}a_{3j}-4a_{4j}a_{2j}^2-5a_{3j}^2a_{2j}+7a_{3j}a_{2j}^3-2a_{2j}^5 \\
&\phantom{=}+\frac{c_3}{2}a_{2j}^3+\left(c_4-\frac{c_3c_2}{2}\right)a_{2j}^2+\left(\frac{3c_5}{4}-\frac{c_4c_2}{2}-\frac{c_3^2}{8}+\frac{c_3c_2^2}{4}+\frac{c_3}{2}\psi_{3j}\right)a_{2j} \\
&\phantom{=}-\frac{1}{2}c_2\psi_{3j}^2+\left(\frac{c_4}{3}-\frac{c_3c_2}{6}\right)\psi_{3j} +D_{2j}+D_{3j}+D_{4j}+D_{5j},
\end{align*}
where we have grouped all integrals containing $\varphi_1^k$ into a single expression $D_{kj}$. This expressions are given explicitly below.
\begin{align*}
 D_{2j} &= \int_{\gamma_j}\frac{(-c_2^3+\frac{1}{2}c_3c_2)S_3+(3c_2^2-\frac{1}{2}c_3)S_3\psi_2-3c_2S_3\psi_2^2+S_3\psi_2^3+S_3\psi_3\psi_2}{r^3}\,\varphi_1^2\,dw  \\
	&=  (-c_2^3+\frac{1}{2}c_3c_2)\psi_{3j} +(3c_2^2-\frac{1}{2}c_3)\Delta_{1j}-3c_2\Delta_{2j} +\Delta_{3j} +\Delta_{(1,1)j}.
\end{align*}

Recall that $q_4=S_4+c_2S_3r-\frac{1}{2}c_3S_2r^2$. We have:
\begin{align*}
 D_{3j} &= \int_{\gamma_j}\frac{(3c_2^2-\frac{1}{2}c_3)(S_4+c_2S_3r-\frac{1}{2}c_3S_2r^2)-6c_2(S_4+c_2S_3r-\frac{1}{2}c_3S_2r^2)\psi_2}{r^4}\,\varphi_1^3\,dw  \\
&\phantom{=} +\int_{\gamma_j}\frac{3(S_4+c_2S_3r-\frac{1}{2}c_3S_2r^2)\psi_2^2+(S_4+c_2S_3r-\frac{1}{2}c_3S_2r^2)\psi_3}{r^4}\,\varphi_1^3\,dw  \\
	&= (3c_2^2-\frac{1}{2}c_3)\psi_{4j}-6c_2\Gamma_{1j}+3\Gamma_{2j}+\Gamma_{(0,1)j}.
\end{align*}

Recall also that $q_5=S_5+2c_2S_4r+c_2^2S_3r^2-\frac{2}{3}(c_4+c_3c_2)S_2r^3$. Thus,
\begin{align*}
 D_{4j} &= \int_{\gamma_j}\frac{-3c_2\left(S_5+2c_2S_4r+c_2^2S_3r^2-\frac{2}{3}(c_4+c_3c_2)S_2r^3\right)}{r^5}\,\varphi_1^4\,dw \\
&\phantom{=} +\int_{\gamma_j}\frac{3\left(S_5+2c_2S_4r+c_2^2S_3r^2-\frac{2}{3}(c_4+c_3c_2)S_2r^3\right)\psi_2}{r^5}\,\varphi_1^4\,dw  \\
	&= -3c_2\psi_{5j}+3\mathrm{B}_{1j}.
\end{align*}

Lastly, we obtain
\begin{align*}
 D_{5j} &= \int_{\gamma_j}\frac{S_6+3c_2S_5r+\left(\frac{1}{2}c_3+3c_2^2\right)S_4r^2+\left(-\frac{1}{3}c_4+\frac{1}{6}c_3c_2+c_2^3\right)S_3r^3}{r^6}\,\varphi_1^5\,dw \\
&\phantom{=} +\int_{\gamma_j}\frac{\left(-\frac{3}{4}c_5-\frac{3}{2}c_4c_2-\frac{1}{8}c_3^2-\frac{3}{4}c_3c_2^2\right)S_2r^4}{r^6}\,\varphi_1^5\,dw,  \\
\end{align*}
which is exactly $\psi_{6j}$, by definition of $q_6(w)$.

In this way we obtain exactly the expression claimed by Proposition \ref{prop:sixthvar}, hence concluding its proof.
\end{proof}

%% file: 6-keylemma.tex

We now proceed to prove the \hyperref[lemma:key]{Key lemma}. Let us consider now a normalized foliation $\Ft$ whose holonomy group at infinity is analytically conjugate to the holonomy group of $\F$. The genericity assumptions imposed on $\F$ and the way we have normalized imply that both foliations have the same characteristic numbers at infinity at the same singular points. Therefore if $\F=\F(\lambda,\alpha)$, we may write $\Ft=\F(\lambda,\beta)$. For every object we have defined for foliation $\F$ we define the analogous object for $\Ft$ and denote it by the same symbol with a tilde on top. In particular $\tilde{f}_1$ and $\tilde{f}_2$ denote the corresponding distinguished parabolic germs which are defined as the holonomy maps along the same loops $\gamma_1$ and $\gamma_2$ from Definition \ref{def:geometricgenerators}. By the conjugacy of the holonomy groups, and in virtue of Remark \ref{rmk:strongae}, there exists a conformal germ $h\in\Diff$ such that
\begin{equation}\label{eq:conjugatef}
h\circ f_j-\tilde{f}_j\circ h=0,\quad j=1,2. 
\end{equation}

We reemphasize that the idea of the \hyperref[lemma:key]{Key lemma} is to show that the above equation imposes certain conditions on the parameter $\beta$. We do this by proving the existence of polynomials $P_d(w)$, whose coefficients depend on $\lambda$, $\alpha$ and $\beta$, with the property that if equation (\ref{eq:conjugatef}) holds up to jets of order $d$ then 
\[ \int_{\gamma_1}\frac{P_d(w)}{r(w)^d}\,\varphi_1(w)^{d-1}\,dw = 0. \]

We will first compare the terms of degree 2 in equation (\ref{eq:conjugatef}) and prove that the normal form (\ref{eq:normalform}) that we have chosen forces the germ $h$ to be parabolic. The \hyperref[lemma:key]{Key lemma} for degree $d=3$ will be a corollary of this fact. Once we have done this we will prove the \hyperref[lemma:key]{Key lemma} for higher degrees, one degree at the time, following the strategy explained in Subsection \ref{subsec:keylemma}.

\subsection{Comparison of the terms of low degree}

We start with an important observation about the normal form (\ref{eq:normalform}).

\begin{proposition}\label{prop:S2=r}
 The polynomial $S_2(w)$ defined in Proposition \textnormal{\ref{prop:c_dS_d}} by the property $K_2=c_2K_1+\frac{S_2}{r^2}$ is exactly $S_2(w)=r(w)$. In particular the function  
\[ \psi_2(w)=\int_0^w\frac{S_2}{r^2}\varphi_1\,dt=\int_0^w\frac{1}{r}\varphi_1\,dt \]
depends only on the characteristic numbers $\lambda_1,\lambda_2$ and not on the parameter $\alpha$, and so we have $\tilde{\psi}_2(w)=\psi_2(w)$.
\end{proposition}

This proposition is proved by just expanding $F(z,w)$ in a power series and computing the quadratic coefficient $K_2$. We omit the proof here since we shall give explicit expression for all the terms $S_d$ and $c_d$ at the begining of Section \ref{sec:elimination}.

\begin{proposition}\label{prop:h2}
If $h\in\Diff$ conjugates the holonomy groups of $\F$ and $\Ft$ then $h$ is necessarily a parabolic germ and its quadratic coefficient $h_2=\frac{1}{2}h''(0)$ is given by $h_2=\tilde{c}_2-c_2$, with $c_2,\tilde{c}_2$ as in Proposition \textnormal{\ref{prop:c_dS_d}}
\end{proposition}

\begin{proof}
If the germ $h$ conjugates the holonomy groups it conjugates the distinguished parabolic germs, which by genericity hypothesis have non--zero quadratic part. By Proposition \ref{prop:secondvar} the quadratic coefficient in the power series of $f_j$ is $a_{2j}=\psi_{2j}$, and by Proposition \ref{prop:S2=r} $\psi_2(w)$ depends only on the characteristic numbers $\lambda_1$, $\lambda_2$. This implies that $a_{2j}=\tilde{a}_{2j}$. Any germ that conjugates two parabolic germs with equal non--zero quadratic part must be parabolic itself, hence $h$ is parabolic.

We now prove the second claim. This is the only instance in this paper where we will consider holonomy maps other than the distinguished parabolic germs. Choose any holonomy map $\Delta_\gamma$ that is not parabolic (for example, choose $\gamma=\mu_1$, a standard geometric generator) and consider its power series expansion: $\Delta_\gamma=\var{1}{}\,z+\var{2}{}\,z^2+O(z^3)$. We also consider the corresponding power series expansion for $\widetilde{\Delta}_\gamma$. Taking into account that $\widetilde{\varphi}_1=\varphi_1$, an easy computation shows that $h\circ\Delta_\gamma-\widetilde{\Delta}_\gamma\circ h$ has a power series expansion of the form
\[ \left(\var{2}{}-\vart{2}{}+h_2\var{1}{}(\var{1}{}-1)\right)\,z^2+O(z^3), \]
which implies
\[ h_2=\frac{\vart{2}{}-\var{2}{}}{\var{1}{}(\var{1}{}-1)} \]
since $h\circ\Delta_\gamma-\widetilde{\Delta}_\gamma\circ h\equiv 0$. Now, we use the relation $\varphi_2=\varphi_1\phi_2$ and Proposition \ref{prop:secondvar} to simplify the numerator. Taking into account that $\psi_2(w)=\tilde{\psi}_2(w)$, we get that $h_2=\tilde{c}_2-c_2$.
\end{proof}

We remark that the fact that $h$ is forced to be parabolic depends strongly on the fact that both $\F$ and $\Ft$ have been normalized as in (\ref{eq:normalform}). Without this normalization the above proposition need not hold.

In virtue of the above proposition we may write
\[ h(z)=z+\sum_{d=2}^{\infty}h_dz^d. \]

\begin{proposition}\label{prop:key3}
Define $P_3(w)=\widetilde{S}_3(w)-S_3(w)$. If a germ $h\in\Diff$ conjugates corresponding pairs of distinguished parabolic germs up to $3$--jets then
\[ \int_{\gamma_1}\frac{P_3(w)}{r(w)^3}\,\varphi_1(w)^2\,dw=0. \]
\end{proposition}

\begin{proof}
It is easy to check that the commutator of any two parabolic germs is of the form $z+O(z^4)$. This implies that the group of 3--jets of parabolic germs is commutative, in particular $f_j$ and $\tilde{f}_j$ have the same 3--jet since $h\circ f_j=\tilde{f}_j\circ h$ and all these germs are parabolic. This tells us that $a_{3j}=\tilde{a}_{3j}$ and moreover $\psi_{3j}=\widetilde{\psi}_{3j}$ since, by Proposition \ref{prop:thirdvar},  $a_{3j}=a_{2j}^2+\psi_{3j}$, and $\tilde{a}_{2j}=a_{2j}$. Recall that we have defined $\psi_{3j}=\int_{\gamma_j}\frac{S_3}{r^3}\,\varphi_1^2\,dw$. Hence,
\[0=\widetilde{\psi}_{31}-\psi_{31}=\int_{\gamma_1}\frac{\widetilde{S}_3-S_3}{r^3}\,\varphi_1^2\,dw=\int_{\gamma_1}\frac{P_3}{r^3}\,\varphi_1^2\,dw. \]
\end{proof}

Before moving on to the \hyperref[lemma:key]{Key Lemma} for degree four, we will use Lemma \ref{lemma:Pyartli} to introduce a polynomial $R_3(w)$ needed in the next subsection (see Subsection \ref{subsec:keytomain} for the general description of the polynomials $R_d(w)$).

\begin{proposition}\label{prop:R3}
If $\lambda_1,\lambda_2\notin\frac{1}{2}\Z$ there exists a polynomial $R_3(w)$ such that 
\[ \int_0^w\frac{P_3(t)}{r(t)^3}\,\varphi_1(t)^2\,dt=\frac{R_3(w)}{r(w)^2}\,\varphi_1(w)^2-R_3(0). \]
\end{proposition}

\begin{proof}
The above proposition is exactly Lemma \ref{lemma:Pyartli} with $P(w)=P_3(w)$ and $u_j=2\lambda_j-3$.
\end{proof}

\subsection{Key lemma for degree four}

In Subsection \ref{subsec:keylemma} we have reduced the proof of the \hyperref[lemma:key]{Key lemma} on degree 4 to the proof of existence of a polynomial $P_4(w)$ and a complex number $\mathcal{C}_4$ such that
\[ a_{2j}\,\mathcal{C}_4+\mathcal{I}_{4j}=0, \quad j=1,2, \]
where $\mathcal{I}_{4j}=\int_{\gamma_j}\frac{P_4}{r^4}\,\varphi_1^{3}\,dw$. Thus, in order to prove the next proposition we shall prove the existence of a polynomial $P_4$ and a number $\mathcal{C}_4$ satisfying the above conditions and cite Proposition \ref{prop:key}.

\begin{proposition}\label{prop:key4}
Let $P_4(w)=\tilde{q}_4(w)-q_4(w)-S_2(w)R_3(w)$ with the polynomials $q_4(w)$ as in Proposition \textnormal{\ref{prop:fourthvar}} and $R_3(w)$ as in Proposition \textnormal{\ref{prop:R3}}. If a germ $h\in\Diff$ conjugates corresponding pairs of distinguished parabolic germs up to $4$--jets then
\[ \int_{\gamma_1}\frac{P_4(w)}{r(w)^4}\,\varphi_1(w)^3\,dw=0. \]

Moreover the cubic coefficient in the power series of $h$ is given by
\begin{equation}\label{eq:h3}
 h_3=h_2^2+\frac{\tilde{c}_3-c_3}{2}+R_3(0). 
\end{equation}
\end{proposition}

\begin{proof}
Taking into account that we know $\tilde{a}_{2j}=a_{2j}$ and $\tilde{a}_{3j}=a_{3j}$, a short computation shows that the coefficient of degree 4 in the power series expansion of $h\circ f_j-\tilde{f}_j\circ h$ is given by $(h_3-h_2^2)a_{2j}-h_2(a_{3j}-a_{2j}^2)-\tilde{a}_{4j}+a_{4j}$. This implies
\begin{equation}\label{eq:a4t-a4_v1}
 \tilde{a}_{4j}-a_{4j}=(h_3-h_2^2)a_{2j}-h_2(a_{3j}-a_{2j}^2), \quad j=1,2. 
\end{equation}

On the other hand, it follows from Proposition \ref{prop:fourthvar} that
\[ \tilde{a}_{4j}-a_{4j}=\frac{\tilde{c_3}-c_3}{2}a_{2j}-(\tilde{c}_2-c_2)\psi_{3j}+\widetilde{\Delta}_{1j}-\Delta_{1j}+\widetilde{\psi}_{4j}-\psi_{4j}. \]
In the above expression we are using the fact that $\tilde{a}_{2j}=a_{2j}$, $\tilde{a}_{3j}=a_{3j}$ and also that $\widetilde{\psi}_{3j}=\psi_{3j}$. Now, using the fact that $\widetilde{\psi}_2(w)=\psi_2(w)$ we see that
\[ \widetilde{\Delta}_{1j}-\Delta_{1j}=\int_{\gamma_j}\frac{\widetilde{S}_3-S_3}{r^3}\psi_2\varphi_1^2\,dw=\int_{\gamma_j}\frac{P_3}{r^3}\psi_2\varphi_1^2\,dw. \]
Using Proposition \ref{prop:R3} we can integrate by parts the last integral above to obtain
\begin{equation}\label{formula:Delta1t-Delta1} \widetilde{\Delta}_{1j}-\Delta_{1j}=\int_{\gamma_j}\left(\frac{R_3}{r^2}\varphi_1^2\right)^\prime\psi_2\,dw= R_3(0)a_{2j}-\int_{\gamma_j}\frac{R_3S_2}{r^4}\varphi_1^3\,dw, 
\end{equation}
Taking into account that we have defined $P_4=\tilde{q}_4-q_4-S_2R_3$ we see that
\begin{equation}\label{eq:a4t-a4_v2}
\tilde{a}_{4j}-a_{4j}=\frac{\tilde{c_3}-c_3}{2}a_{2j}-(\tilde{c}_2-c_2)\psi_{3j}+R_3(0)a_{2j}+\int_{\gamma_j}\frac{P_4}{r^4}\varphi_1^3\,dw, \quad j=1,2.
\end{equation}

We now substitute the right hand side of (\ref{eq:a4t-a4_v1}) into (\ref{eq:a4t-a4_v2}) to obtain an expression
\[ (h_3-h_2^2)a_{2j}-h_2(a_{3j}-a_{2j}^2)=\frac{\tilde{c_3}-c_3}{2}a_{2j}-(\tilde{c}_2-c_2)\psi_{3j}+R_3(0)a_{2j}+\int_{\gamma_j}\frac{P_4}{r^4}\varphi_1^3\,dw. \]
Recall that $h_2=\tilde{c}_2-c_2$ by Proposition \ref{prop:h2}, and recall also that $a_{3j}=a_{2j}^2+\psi_{3j}$ by Proposition \ref{prop:thirdvar}, therefore $(\tilde{c}_2-c_2)\psi_{3j}=h_2(a_{3j}-a_{2j}^2)$. The equation above is thus simplified to
\[ (h_3-h_2^2)a_{2j}=\left(\frac{\tilde{c_3}-c_3}{2}+R_3(0)\right)a_{2j}+\int_{\gamma_j}\frac{P_4}{r^4}\varphi_1^3\,dw, \]
which can be rewritten in the form
\[ a_{2j}\,\mathcal{C}_4+\mathcal{I}_{4j}=0, \]
where
\[ \mathcal{C}_4=\frac{\tilde{c_3}-c_3}{2}+R_3(0)+h_2^2-h_3, \]
and
\[ \mathcal{I}_{4j}=\int_{\gamma_j}\frac{P_4}{r^4}\,\varphi_1^3\,dw. \]
By Proposition \ref{prop:key} we have 
\[ \mathcal{I}_{41}=\int_{\gamma_1}\frac{P_4}{r^4}\varphi_1^3\,dw=0, \qquad \mathcal{C}_4=0. \]
This proves the \hyperref[lemma:key]{Key lemma} for degree four. Note that $\mathcal{C}_4=0$ implies
\[ h_3=h_2^2+\frac{\tilde{c}_3-c_3}{2}+R_3(0). \]
\end{proof}

We conclude this subsection by introducing the polynomial $R_4(w)$.
\begin{proposition}\label{prop:R4}
If $\lambda_1,\lambda_2\notin\frac{1}{3}\Z$ there exists a polynomial $R_4(w)$ such that 
\[ \int_0^w\frac{P_4(t)}{r(t)^4}\,\varphi_1(t)^3\,dt=\frac{R_4(w)}{r(w)^3}\,\varphi_1(w)^3+R_4(0). \]
\end{proposition}

\begin{proof}
Apply Lemma \ref{lemma:Pyartli} with $P(w)=P_4(w)$ and $u_j=3\lambda_j-4$. 
\end{proof}

\subsection{Key lemma for degree five}

We proceed in exactly the same way as we did in the previous subsection.

\begin{proposition}\label{prop:key5}
Let $P_5(w)=\tilde{q}_5(w)-q_5(w)-2S_2(w)R_4(w)$ with the polynomials $q_5(w)$ as in Proposition \textnormal{\ref{prop:fifthvar}} and $R_4(w)$ as in Proposition \textnormal{\ref{prop:R4}}. If a germ $h\in\Diff$ conjugates corresponding pairs of distinguished parabolic germs up to $5$--jets then
\[ \int_{\gamma_1}\frac{P_5(w)}{r(w)^5}\,\varphi_1(w)^4\,dw=0. \]
Moreover, the coefficient of degree four in the power series expansion of $h$ is given by
\begin{equation}\label{eq:h4}
 h_4=\frac{\tilde{c}_4-c_4}{3}-\frac{\tilde{c}_3\tilde{c}_2-c_3c_2}{6}-R_4(0)-\tilde{c}_2R_3(0)+3h_3h_2-2h_2^3+\frac{c_3}{2}h_2. 
\end{equation}
\end{proposition}

\begin{proof}
Taking into account that $\tilde{a}_{2j}=a_{2j}$ and $\tilde{a}_{3j}=a_{3j}=a_{2j}^2+\psi_{3j}$, a straightforward computation shows that the coefficient of degree 5 in the power series expansion of $h\circ f_j-\tilde{f}_j\circ h$ is given by
\begin{align}\label{eq:a5t-a5-v0}
&- \tilde{a}_{5j}+a_{5j}-4h_2(\tilde{a}_{4j}-a_{4j})-2h_2a_{4j}+2h_2a_{2j}^3+3(h_3-h_2^2)a_{2j}^2 \nonumber \\
&\phantom{-}+(2h_4-2h_3h_2+2h_2\psi_{3j})a_{2j}-3h_2^2\psi_{3j}. 
\end{align}
By Proposition \ref{prop:fourthvar}, 
\[ a_{4j}=2(a_{2j}^2+\psi_{3j})a_{2j}-a_{2j}^3+\frac{c_3}{2}a_{2j}-c_2\psi_{3j}+\Delta_{1j}+\psi_{4j}, \] 
and equation (\ref{eq:a4t-a4_v1}) implies 
\[ \tilde{a}_{4j}-a_{4j}=(h_3-h_2^2)a_{2j}-h_2\psi_{3j}. \] 
Using the above identities and equating (\ref{eq:a5t-a5-v0}) to zero we obtain
\begin{align}\label{eq:a5t-a5-v1}
\tilde{a}_{5j}-a_{5j} &= 3(h_3-h_2^2)a_{2j}^2+(2h_4-4(h_3-h_2^2)h_2-2h_3h_2-2h_2\psi_{3j}-c_3h_2)a_{2j} \nonumber \\
		      &\phantom{=} +(h_2^2+2c_2h_2)\psi_{3j}-2h_2\Delta_{1j}-2h_2\psi_{4j}.
\end{align}

On the other hand, we can use Proposition \ref{prop:fifthvar} to compute $\tilde{a}_{5j}-a_{5j}$. We use once more the facts $\tilde{a}_{2j}=a_{2j}$, $\tilde{a}_{3j}=a_{3j}$ and $\tilde{\psi}_{3j}=\psi_{3j}$, thus
\begin{align}
 \tilde{a}_{5j}-a_{5j} 
&=		2(\tilde{a}_{4j}-a_{4j})a_{2j}+\frac{\tilde{c}_3-c_3}{2}a_{2j}^2+\left(2\frac{\tilde{c}_4-c_4}{3}-\frac{\tilde{c}_3\tilde{c}_2-c_3c_2}{3}\right)a_{2j} \label{eq:a5t-a5-v2-pt1} \\
&\phantom{=}	+(\tilde{c}_2^2-c_2^2)\psi_{3j}-2\tilde{c}_2\tilde{\psi}_{4j}+2c_2\psi_{4j}-2\tilde{c}_2\widetilde{\Delta}_{1j}+2c_2\Delta_{1j} \label{eq:a5t-a5-v2-pt2} \\
&\phantom{=}	+\widetilde{\Delta}_{2j}-\Delta_{2j}+2\widetilde{\Gamma}_{1j}-2\Gamma_{1j}+\tilde{\psi}_{5j}-\psi_{5j}. \label{eq:a5t-a5-v2-pt3}
\end{align}
First, note that using the expression found for $\tilde{a}_{4j}-a_{4j}$ in (\ref{eq:a4t-a4_v1}) we can rewrite the right--hand side of (\ref{eq:a5t-a5-v2-pt1}) as
\begin{equation}\label{eq:a5t-a5-PQ}
 \left(2(h_3-h_2^2)+\frac{\tilde{c}_3-c_3}{2}\right)a_{2j}^2+\left(-2h_2\psi_{3j}+2\frac{\tilde{c}_4-c_4}{3}-\frac{\tilde{c}_3\tilde{c}_2-c_3c_2}{3}\right)a_{2j}.
\end{equation}
Now, note that $\widetilde{\Delta}_{2j}-\Delta_{2j}=\int_{\gamma_j}\frac{P_3}{r^3}\psi_2^2\varphi_1^2\,dw$, and so integration by parts yields
\begin{equation}\label{formula:Delta2t-Delta2}
 \widetilde{\Delta}_{2j}-\Delta_{2j}=R_3(0)a_{2j}^2-\int_{\gamma_j}\frac{2R_3S_2}{r^4}\psi_2\varphi_1^3\,dw. 
\end{equation}
Recall that $P_4=\tilde{q}_4-q_4-2S_2R_3$, therefore
\begin{align*}
\widetilde{\Delta}_{2j}-\Delta_{2j}+2\widetilde{\Gamma}_{1j}-2\Gamma_{1j}
&= R_3(0)a_{2j}^2-2\int_{\gamma_j}\frac{S_2R_3}{r^4}\psi_2\varphi_1^3\,dw+2\int_{\gamma_j}\frac{\tilde{q}_4-q_4}{r^4}\psi_2\varphi_1^3\,dw \\
&= R_3(0)a_{2j}^2+2\int_{\gamma_j}\frac{P_4}{r^4}\psi_2\varphi_1^3\,dw.
\end{align*}
Integrating by parts the last integral we obtain
\[ \int_{\gamma_j}\frac{P_4}{r^4}\psi_2\varphi_1^3\,dw=-R_4(0)a_{2j}-\int_{\gamma_j}\frac{R_4S_2}{r^5}\varphi_1^4\,dw. \]
We conclude that
\[ \widetilde{\Delta}_{2j}-\Delta_{2j}+2\widetilde{\Gamma}_{1j}-2\Gamma_{1j} = R_3(0)a_{2j}^2-2R_4(0)a_{2j}-\int_{\gamma_j}\frac{2S_2R_4}{r^5}\varphi_1^4\,dw. \]
Since we defined $P_5=\tilde{q}_5-q_5-2S_2R_4$ and $\psi_{5j}=\int_{\gamma_j}\frac{q_5}{r^5}\,\varphi_1^4\,dw$ we see that expression (\ref{eq:a5t-a5-v2-pt3}) is given by
\begin{equation}\label{eq:a5t-a5-D2G1P5}
 R_3(0)a_{2j}^2-2R_4(0)a_{2j}+\int_{\gamma_j}\frac{P_5}{r^5}\varphi_1^4\,dw.
\end{equation}
Let us now analyse expression (\ref{eq:a5t-a5-v2-pt2}). Note that $\tilde{c}_2^2-c_2^2=h_2^2+2c_2h_2$, since $h_2=\tilde{c_2}-c_2$, therefore the first term in (\ref{eq:a5t-a5-v2-pt2}) can be rewritten as $(h_2^2+2c_2h_2)\psi_{3j}$. Next,
\[ -2\tilde{c}_2\tilde{\psi}_{4j}+2c_2\psi_{4j}=-2h_2\psi_{4j}-2\tilde{c}_2(\tilde{\psi}_{4j}-\psi_{4j}), \]
and
\[ -2\tilde{c}_2\widetilde{\Delta}_{1j}+2c_2\Delta_{1j}=-2h_2\Delta_{1j}-2\tilde{c}_2(\widetilde{\Delta}_{1j}-\Delta_{1j}). \]
We've seen already that $\widetilde{\Delta}_{1j}-\Delta_{1j}=R_3(0)a_{2j}-\int_{\gamma_j}\frac{S_2R_3}{r^4}\varphi_1^3\,dw$, so taking into account that $\psi_{4j}=\int_{\gamma_j}\frac{q_4}{r^4}\varphi_1^3\,dw$ and $P_4=\tilde{q}_4-q_4-S_2R_3$ we get that expression (\ref{eq:a5t-a5-v2-pt2}) is given by
\begin{align}\label{eq:a5t-a5-otros}
&  (h_2^2+2c_2h_2)\psi_{3j}-2h_2\psi_{4j}-2h_2\Delta_{1j}-2\tilde{c}_2R_3(0)a_{2j}-2\tilde{c}_2\int_{\gamma_j}\frac{P_4}{r^4}\varphi_1^3\,dw, \nonumber \\
&= (h_2^2+2c_2h_2)\psi_{3j}-2h_2\psi_{4j}-2h_2\Delta_{1j}-2\tilde{c}_2R_3(0)a_{2j},
\end{align}
since, according to Proposition \ref{prop:key4}, $\int_{\gamma_j}\frac{P_4}{r^4}\varphi_1^3\,dw=0$. Adding up expressions (\ref{eq:a5t-a5-PQ}), (\ref{eq:a5t-a5-D2G1P5}) and (\ref{eq:a5t-a5-otros}), and taking into account that 
\[ h_3-h_2^2=\frac{\tilde{c}_3-c_3}{2}+R_3(0), \]
(which also follows from Proposition \ref{prop:key4}) we finally obtain
\begin{align}\label{eq:a5t-a5-v2}
 \tilde{a}_{5j}-a_{5j} &= 
		3(h_3-h_2^2)a_{2j}^2 + \left(-2h_2\psi_{3j}+2\frac{\tilde{c}_4-c_4}{3}-\frac{\tilde{c}_3\tilde{c}_2-c_3c_2}{3}-2R_4(0)-2\tilde{c}_2R_3(0)\right)a_{2j} \nonumber \\
&\phantom{=}	+(h_2^2+2c_2h_2)\psi_{3j}-2h_2\Delta_{1j}-2h_2\psi_{4j}+\int_{\gamma_j}\frac{P_5}{r^5}\varphi_1^4\,dw.
\end{align}
	  
We now equate the right hand sides of (\ref{eq:a5t-a5-v1}) and (\ref{eq:a5t-a5-v2}). Note that we can cancel those terms with $a_{2j}^2$ as well as those terms where $a_{2j}$ does not appear, with the exception of $\int_{\gamma_j}\frac{P_5}{r^5}\varphi_1^4\,dw$. We thus obtain an equation
\[ a_{2j}\,\mathcal{C}_5+\mathcal{I}_{5j}=0, \]
where 
\[ \mathcal{C}_5=2\frac{\tilde{c}_4-c_4}{3}-\frac{\tilde{c}_3\tilde{c}_2-c_3c_2}{3}-2R_4(0)-2\tilde{c}_2R_3(0)+6h_3h_2-4h_2^3+c_3h_2-2h_4, \]
and
\[ \mathcal{I}_{5j}=\int_{\gamma_j}\frac{P_5}{r^5}\,\varphi_1^{4}\,dw. \]
By Proposition \ref{prop:key},
\[ \mathcal{I}_{51}=\int_{\gamma_1}\frac{P_5}{r^5}\,\varphi_1^4\,dw=0, \qquad \mathcal{C}_5=0. \]
This proves the \hyperref[lemma:key]{Key lemma} for degree five. Moreover, it follows from $\mathcal{C}_5=0$ that
\[ h_4=\frac{\tilde{c}_4-c_4}{3}-\frac{\tilde{c}_3\tilde{c}_2-c_3c_2}{6}-R_4(0)-\tilde{c}_2R_3(0)+3h_3h_2-2h_2^3+\frac{c_3}{2}h_2. \]
Proposition \ref{prop:key5} is now proved.
\end{proof}

We now introduce the polynomial $R_5(w)$.

\begin{proposition}\label{prop:R5} 
If $\lambda_1,\lambda_2\notin\frac{1}{4}\Z$ there exists a polynomial $R_5(w)$ such that
\[ \int_0^w\frac{P_5(t)}{r(t)^5}\,\varphi_1(t)^4\,dt=\frac{R_5(w)}{r(w)^4}\,\varphi_1^4-R_5(0). \]
\end{proposition}

\begin{proof}
Apply Lemma \ref{lemma:Pyartli} with $P(w)=P_5(w)$ and $u_j=4\lambda_j-5$. 
\end{proof}

\subsection{Key lemma for degree six}

\begin{proposition}\label{prop:key6}
Let us define
\[ P_6=\tilde{q}_6-q_6+\tilde{q}_4R_3-\frac{1}{2}S_2R_3^2-S_3R_4-3S_2R_5, \]
with the polynomials $q_6$ as in Proposition \textnormal{\ref{prop:sixthvar}} and $R_5$ as in Proposition \textnormal{\ref{prop:R5}}. If a germ $h\in\Diff$ conjugates corresponding pairs of distinguished parabolic germs up to $6$--jets then
\[ \int_{\gamma_1}\frac{P_6(w)}{r(w)^6}\,\varphi_1(w)^5\,dw=0. \]
\end{proposition}

\begin{proof}
Let us start by using Proposition \ref{prop:sixthvar} to obtain an expression for $\tilde{a}_{6j}-a_{6j}$. Using that  $\tilde{a}_{2j}=a_{2j}$ and $\tilde{a}_{3j}=a_{3j}$ we obtain the following formula for $\tilde{a}_{6j}-a_{6j}$,
\begin{align}
&  2(\tilde{a}_{5j}-a_{5j})a_{2j}+3(\tilde{a}_{4j}-a_{4j})a_{3j}-4(\tilde{a}_{4j}-a_{4j})a_{2j}^2 \label{a6t-a6_v2_1} \\
& +\frac{\tilde{c}_3-c_3}{2}a_{2j}^3+\left(\tilde{c}_4-c_4-\frac{\tilde{c}_3\tilde{c}_2-c_3c_2}{2}\right)a_{2j}^2 \label{a6t-a6_v2_2} \\
& +\left(\frac{3\tilde{c}_5-3c_5}{4}-\frac{\tilde{c}_4\tilde{c}_2-c_4c_2}{2}-\frac{\tilde{c}_3^2-c_3^2}{8}+\frac{\tilde{c}_3\tilde{c}_2^2-c_3c_2^2}{4}+\frac{\tilde{c}_3-c_3}{2}\psi_{3j}\right)a_{2j} \label{a6t-a6_v2_3} \\
& -\frac{\tilde{c}_2-c_2}{2}\psi_{3j}^2+\left(\frac{\tilde{c}_4-c_4}{3}+\frac{\tilde{c}_3\tilde{c}_2-c_3c_2}{3}-\tilde{c}_2^3+c_2^3\right)\psi_{3j} \label{a6t-a6_v2_4} \\
& +\left(-\frac{\tilde{c}_3}{2}+3\tilde{c}_2^2\right)\widetilde{\Delta}_{1j}-\left(-\frac{c_3}{2}+3c_2^2\right)\Delta_{1j} \label{a6t-a6_v2_5} \\
& -3\tilde{c}_2\widetilde{\Delta}_{2j}+3c_2\Delta_{2j} +\widetilde{\Delta}_{3j}-\Delta_{3j} +\widetilde{\Delta}_{(1,1)j}-\Delta_{(1,1)j} \label{a6t-a6_v2_6} \\
& +\left(-\frac{\tilde{c}_3}{2}+3\tilde{c}_2^2\right)\tilde{\psi}_{4j}-\left(-\frac{c_3}{2}+3c_2^2\right)\psi_{4j} -6\tilde{c}_2\widetilde{\Gamma}_{1j}+6c_2\Gamma_{1j} +3\widetilde{\Gamma}_{2j}-3\Gamma_{2j} \label{a6t-a6_v2_7} \\
& +\widetilde{\Gamma}_{(0,1)j}-\Gamma_{(0,1)j} -3\tilde{c}_2\tilde{\psi}_{5j}+3c_2\psi_{5j} +3\widetilde{\mathrm{B}}_{1j}-3\mathrm{B}_{1j} +\tilde{\psi}_{6j}-\psi_{6j}. \label{a6t-a6_v2_8}
\end{align}

We now shall rewrite several of the terms in the above expression for $\tilde{a}_{6j}-a_{6j}$. For (\ref{a6t-a6_v2_1}) we can use the expression for $\tilde{a}_{5j}-a_{5j}$ found in (\ref{eq:a5t-a5-v1}) and that for $\tilde{a}_{4j}-a_{4j}$ from (\ref{eq:a4t-a4_v1}), and write $a_{3j}=a_{2j}^2+\psi_{3j}$. We obtain the following expression after these substitutions:
\begin{align}
& 5(h_3-h_2^2)a_{2j}^3+\big(4h_4-12h_3h_2^2+8h_2^3-2c_3h_2-3h_2\psi_{3j}\big)a_{2j}^2 \nonumber \\
& +\big(3h_3\psi_{3j}-h_2^2\psi_{3j}+4c_2h_2\psi_{3j}-4h_2\Delta_{1j}-4h_2\psi_{4j}\big)a_{2j} -3h_2\psi_{3j}^2. \label{eq:Pt-P}
\end{align}

Next, equation (\ref{a6t-a6_v2_5}) can be rewritten as
\[ \left(-\frac{\tilde{c}_3-c_3}{2}+3(\tilde{c}_2^2-c_2^2)\right)\Delta_{1j}+\left(-\frac{\tilde{c}_3}{2}+3\tilde{c}_2^2\right)(\widetilde{\Delta}_{1j}-\Delta_{1j}). \]
We have an expression for $\widetilde{\Delta}_{1j}-\Delta_{1j}$ from equation (\ref{formula:Delta1t-Delta1}). Using this, (\ref{a6t-a6_v2_5}) becomes
\begin{equation}\label{eq:Delta1t-Delta1}
\left(-\frac{\tilde{c}_3-c_3}{2}+3(\tilde{c}_2^2-c_2^2)\right)\Delta_{1j} +\left(-\frac{\tilde{c}_3}{2}+3\tilde{c}_2^2\right)R_3(0)a_{2j} -\left(-\frac{\tilde{c}_3}{2}+3\tilde{c}_2^2\right)\int_{\gamma_j}\frac{S_2R_3}{r^4}\,\varphi_1^3\,dw.
\end{equation}

We also have an expression for $\widetilde{\Delta}_{2j}-\Delta_{2j}$ from (\ref{formula:Delta2t-Delta2}), so the first two terms in (\ref{a6t-a6_v2_6}) can be rewritten as
\begin{align}
 -3\tilde{c}_2\widetilde{\Delta}_{2j}+3c_2\Delta_{2j} &= -3(\tilde{c}_2-c_2)\Delta_{2j}-3\tilde{c}_2(\widetilde{\Delta}_{2j}-\Delta_{2j}) \nonumber \\
 &= -3(\tilde{c}_2-c_2)\Delta_{2j}-3\tilde{c}_2R_3(0)a_{2j}^2+6\tilde{c}_2\int_{\gamma_j}\frac{S_2R_3}{r^4}\,\psi_{2j}\varphi_1^3\,dw. \label{eq:Delta2t-Delta2}
\end{align}

In the same way as we deduced the formulas for $\widetilde{\Delta}_{1j}-\Delta_{1j}$ and $\widetilde{\Delta}_{2j}-\Delta_{2j}$, we integrate $\widetilde{\Delta}_{3j}-\Delta_{3j}=\int_{\gamma_j}\frac{P_3}{r^3}\,\psi_2^3\varphi_1^2\,dw$ by parts to obtain
\begin{equation}\label{eq:Delta3t-Delta3}
 \widetilde{\Delta}_{3j}-\Delta_{3j}=R_3(0)a_{2j}^3-3\int_{\gamma_j}\frac{S_2R_3}{r^4}\,\psi_2^2\varphi_1^3\,dw.
\end{equation}

We now wish to express $\widetilde{\Delta}_{(1,1)j}-\Delta_{(1,1)j}$ in terms of simpler objects. We proceed as follows. By definition,
\[ \widetilde{\Delta}_{(1,1)j}-\Delta_{(1,1)j} = \int_{\gamma_j}\frac{\widetilde{S}_3}{r^3}\,\tilde{\psi}_2\tilde{\psi}_3\varphi_1^2\,dw-\int_{\gamma_j}\frac{S_3}{r^3}\,\psi_2\psi_3\varphi_1^2\,dw,\]
which, taking into account that $\tilde{\psi}_2=\psi_2$, may be rewritten as
\begin{equation}\label{eq:Delta1,1t-Delta1,1_1}
\int_{\gamma_j}\frac{\widetilde{S}_3-S_3}{r^3}\,\psi_2\psi_3\varphi_1^2\,dw +\int_{\gamma_j}\frac{\widetilde{S}_3}{r^3}\,\psi_2(\tilde{\psi}_3-\psi_3)\varphi_1^2\,dw.
\end{equation}
The first integral in the above equation is given by $\int_{\gamma_j}\frac{P_3}{r^3}\,\psi_2\psi_3\varphi_1^2\,dw$, and so integration by parts yields
\begin{equation}\label{eq:Delta1,1t-Delta1,1_2}
R_3(0)a_{2j}\psi_{3j}-\int_{\gamma_j}\frac{S_2R_3}{r^4}\psi_3\varphi_1^3\,dw -\int_{\gamma_j}\frac{S_3R_3}{r^5}\,\psi_2\varphi_1^4\,dw.
\end{equation}
On the other hand, note that 
\begin{equation}\label{formula:psi3t-psi3}
 \tilde{\psi}_3(w)-\psi_3(w)=\int_0^w\frac{P_3}{r^3}\,\varphi_1^2\,dt=\frac{R_3(w)}{r(w)^2}\varphi_1(w)^2-R_3(0).
\end{equation}
Thus, the second integral in (\ref{eq:Delta1,1t-Delta1,1_1}) can be rewritten as
\[ \int_{\gamma_j}\frac{\widetilde{S}_3R_3}{r^5}\,\psi_2\varphi_1^4\,dw-R_3(0)\widetilde{\Delta}_{1j}, \]
since, by definition, $\widetilde{\Delta}_{1j}=\int_{\gamma_j}\frac{\widetilde{S}_3}{r^3}\,\tilde{\psi}_2\varphi_1^2\,dw$. In fact, taking into account (\ref{eq:Delta1,1t-Delta1,1}) and writing $\widetilde{\Delta}_{1j}=\Delta_{1j}+R_3(0)a_{2j}-\int_{\gamma_j}\frac{S_2R_3}{r^4}\,\varphi_1^3\,dw$ we obtain that the second integral in (\ref{eq:Delta1,1t-Delta1,1_1}) is given by
\begin{equation}\label{eq:Delta1,1t-Delta1,1_3}
\int_{\gamma_j}\frac{\widetilde{S}_3R_3}{r^5}\,\psi_2\varphi_1^4\,dw -R_3(0)\Delta_{1j}-R_3(0)^2a_{2j}+R_3(0)\int_{\gamma_j}\frac{S_2R_3}{r^4}\,\varphi_1^3\,dw.
\end{equation}
We claim that the following equality holds:
\begin{equation}\label{formula:psi4t-psi4}
 \int_{\gamma_j}\frac{S_2R_3}{r^4}\,\varphi_1^3\,dw=\tilde{\psi}_{4j}-\psi_{4j}. 
\end{equation}
Indeed, by definition, $\psi_{4j}=\int_{\gamma_j}\frac{q_4}{r^4}\,\varphi_1^3\,dw$ and so
\[ \tilde{\psi}_{4j}-\psi_{4j}=\int_{\gamma_j}\frac{\tilde{q}_4-q_4}{r^4}\,\varphi_1^3\,dw= \int_{\gamma_j}\frac{P_4}{r^4}\,\varphi_1^3\,dw+\int_{\gamma_j}\frac{S_2R_3}{r^4}\,\varphi_1^3\,dw, \]
since we have defined $P_4$ to be exacly $P_4=\tilde{q}_4-q_4-S_2R_3$. But acording to Proposition \ref{prop:key4} $\int_{\gamma_j}\frac{P_4}{r^4}\,\varphi_1^3\,dw=0$. This proves our claim and so we deduce that expression (\ref{eq:Delta1,1t-Delta1,1_3}), which is the second integral in (\ref{eq:Delta1,1t-Delta1,1_1}), equals
\begin{equation}\label{eq:Delta1,1t-Delta1,1_3.5}
\int_{\gamma_j}\frac{\widetilde{S}_3R_3}{r^5}\,\psi_2\varphi_1^4\,dw -R_3(0)\Delta_{1j}-R_3(0)^2a_{2j}+R_3(0)(\tilde{\psi}_{4j}-\psi_{4j}).
\end{equation}
Combining the last integral from (\ref{eq:Delta1,1t-Delta1,1_2}) and the first one from (\ref{eq:Delta1,1t-Delta1,1_3.5}) into a single integral we get
\begin{align}
\int_{\gamma_j}\frac{(\widetilde{S}_3-S_3)R_3}{r^5}\,\psi_2\varphi_1^4\,dw &= \int_{\gamma_j}\frac{P_3R_3}{r^5}\,\psi_2\varphi_1^4\,dw \nonumber \\
&= \frac{1}{2}\int_{\gamma_j}\left(\left(\frac{R_3}{r^2}\,\varphi_1^2\right)^2\right)^{\prime}\psi_2\,dw \nonumber \\
&= \frac{1}{2}R_3(0)^2a_{2j}-\int_{\gamma_j}\frac{\frac{1}{2}S_2R_3^2}{r^6}\,\varphi_1^5\,dw. \label{eq:Delta1,1t-Delta1,1_4}
\end{align}
Combining (\ref{eq:Delta1,1t-Delta1,1_2}) and (\ref{eq:Delta1,1t-Delta1,1_3.5}), and taking into account (\ref{eq:Delta1,1t-Delta1,1_4}) we obtain the following final expression:
\begin{align}
\widetilde{\Delta}_{(1,1)j}-\Delta_{(1,1)j} &= \left(-R_3(0)\psi_{3j}-\frac{1}{2}R_3(0)^2\right)a_{2j}-R_3(0)\Delta_{1j} +R_3(0)\tilde{\psi}_{4j}-R_3(0)\psi_{4j} \nonumber \\
&\phantom{=} -\int_{\gamma_j}\frac{S_2R_3}{r^4}\psi_3\varphi_1^3\,dw -\int_{\gamma_j}\frac{\frac{1}{2}S_2R_3^2}{r^6}\,\varphi_1^5\,dw. \label{eq:Delta1,1t-Delta1,1}
\end{align}

Next, we rewrite the first two terms of (\ref{a6t-a6_v2_7}) as
\[ \left(-\frac{\tilde{c}_3-c_3}{2}+3(\tilde{c}_2^2-c_2^2)\right)\psi_{4j} +\left(-\frac{\tilde{c}_3}{2}+3\tilde{c}_2^2\right)(\tilde{\psi}_{4j}-\psi_{4j}), \]
and use equation (\ref{formula:psi4t-psi4}) to obtain
\begin{equation}\label{eq:psi4t-psi4}
 \left(-\frac{\tilde{c}_3-c_3}{2}+3(\tilde{c}_2^2-c_2^2)\right)\psi_{4j}+\left(-\frac{\tilde{c}_3}{2}+3\tilde{c}_2^2\right)\int_{\gamma_j}\frac{S_2R_3}{r^4}\,\varphi_1^3\,dw.
\end{equation}

Similarly,
\[ -6\tilde{c}_2\widetilde{\Gamma}_{1j}+6c_2\Gamma_{1j}=-6(\tilde{c}_2-c_2)\Gamma_{1j}-6\tilde{c}_2(\widetilde{\Gamma}_{1j}-\Gamma_{1j}). \]
This time we claim
\begin{equation}\label{formula:Gamma1t-Gamma1}
 \widetilde{\Gamma}_{1j}-\Gamma_{1j}=-R_4(0)a_{2j}-\int_{\gamma_j}\frac{S_2R_4}{r^5}\,\varphi_1^4\,dw+\int_{\gamma_j}\frac{S_2R_3}{r^4}\,\psi_2\varphi_1^3\,dw.
\end{equation}
Indeed, since $P_4=\tilde{q}_4-q_4-S_2R_3$ and $\Gamma_{1j}=\int_{\gamma_j}\frac{q_4}{r^4}\,\psi_2\varphi_1^3\,dw$, we have
\[ \widetilde{\Gamma}_{1j}-\Gamma_{1j}=\int_{\gamma_j}\frac{P_4}{r^4}\,\psi_2\varphi_1^3\,dw+\int_{\gamma_j}\frac{S_2R_3}{r^4}\,\psi_2\varphi_1^3\,dw. \]
The claimed formula is simply obtained by integrating by parts the first integral on the right--hand side of the above equation. We conclude that
\begin{equation}\label{eq:Gamma1t-Gamma1}
 -6\tilde{c}_2\widetilde{\Gamma}_{1j}+6c_2\Gamma_{1j}=-6(\tilde{c}_2-c_2)\Gamma_{1j}+6\tilde{c}_2R_4(0)a_{2j}+6\tilde{c}_2\int_{\gamma_j}\frac{S_2R_4}{r^5}\,\varphi_1^4\,dw-6\tilde{c}_2\int_{\gamma_j}\frac{S_2R_3}{r^4}\,\psi_2\varphi_1^3\,dw.
\end{equation}

The analysis for $3\widetilde{\Gamma}_{2j}-3\Gamma_{2j}$ is analogous:
\[ 3(\widetilde{\Gamma}_{2j}-\Gamma_{2j})=3\int_{\gamma_j}\frac{P_4}{r^4}\,\psi_2^2\varphi_1^3\,dw+3\int_{\gamma_j}\frac{S_2R_3}{r^4}\,\psi_2^2\varphi_1^3\,dw,\]
which, after integration by parts of the first integral, becomes
\begin{equation}\label{eq:Gamma2t-Gamma2}
 3\widetilde{\Gamma}_{2j}-3\Gamma_{2j}=-3R_4(0)a_{2j}^2 -6\int_{\gamma_j}\frac{S_2R_4}{r^5}\,\psi_2\varphi_1^4\,dw +3\int_{\gamma_j}\frac{S_2R_3}{r^4}\,\psi_2^2\varphi_1^3\,dw.
\end{equation}

We now focus on $\widetilde{\Gamma}_{(0,1)j}-\Gamma_{(0,1)j}$. Let us rewite this expression:
\begin{align*} 
 \widetilde{\Gamma}_{(0,1)j}-\Gamma_{(0,1)j} &= \int_{\gamma_j}\frac{\tilde{q}_4}{r^4}\,\tilde{\psi}_3\varphi_1^3\,dw-\int_{\gamma_j}\frac{q_4}{r^4}\,\psi_3\varphi_1^3\,dw\\
&= \int_{\gamma_j}\frac{\tilde{q}_4-q_4}{r^4}\,\psi_3\varphi_1^3\,dw+\int_{\gamma_j}\frac{\tilde{q}_4}{r^4}\,(\tilde{\psi}_3-\psi_3)\varphi_1^3\,dw.
\end{align*}
The first integral in the last expression above is equal to
\[ \int_{\gamma_j}\frac{P_4}{r^4}\,\psi_3\varphi_1^3\,dw+\int_{\gamma_j}\frac{S_2R_3}{r^4}\,\psi_3\varphi_1^3\,dw,\]
and integrating by parts the first term gives us
\[ -R_4(0)\psi_{3j} -\int_{\gamma_j}\frac{S_3R_4}{r^6}\,\varphi_1^5\,dw +\int_{\gamma_j}\frac{S_2R_3}{r^4}\,\psi_3\varphi_1^3\,dw. \]
Taking into account formula (\ref{formula:psi3t-psi3}) we get
\[ \int_{\gamma_j}\frac{\tilde{q}_4}{r^4}\,(\tilde{\psi}_3-\psi_3)\varphi_1^3\,dw=\int_{\gamma_j}\frac{\tilde{q}_4R_3}{r^6}\,\varphi_1^5\,dw-R_3(0)\tilde{\psi}_{4j}. \]
We conclude that
\begin{align}
 \widetilde{\Gamma}_{(0,1)j}-\Gamma_{(0,1)j} &= -R_4(0)\psi_{3j} -\int_{\gamma_j}\frac{S_3R_4}{r^6}\,\varphi_1^5\,dw +\int_{\gamma_j}\frac{S_2R_3}{r^4}\,\psi_3\varphi_1^3\,dw \nonumber \\
&\phantom{=} +\int_{\gamma_j}\frac{\tilde{q}_4R_3}{r^6}\,\varphi_1^5\,dw-R_3(0)\tilde{\psi}_{4j}. \label{eq:Gamma0,1t-Gamma0,1}
\end{align}

The last terms in (\ref{a6t-a6_v2_8}) are as follows: first,
\begin{align}
 -3\tilde{c}_2\tilde{\psi}_{5j}+3c_2\psi_{5j} &= -3(\tilde{c}_2-c_2)\psi_{5j}-3\tilde{c}_2(\tilde{\psi}_{5j}-\psi_{5j}) \nonumber \\
&= -3(\tilde{c}_2-c_2)\psi_{5j}-3\tilde{c}_2\int_{\gamma_j}\frac{\tilde{q}_5-q_5}{r^5}\,\varphi_1^4\,dw \nonumber \\
&= -3(\tilde{c}_2-c_2)\psi_{5j}-6\tilde{c}_2\int_{\gamma_j}\frac{S_2R_4}{r^5}\,\varphi_1^4\,dw, \label{eq:psi5t-psi5}
\end{align}
since $P_5=\tilde{q}_5-q_5-2S_2R_4$ and $\int_{\gamma_j}\frac{P_5}{r^5}\,\varphi_1^4\,dw=0$. Second,
\begin{align}
 3\widetilde{\mathrm{B}}_{1j}-3\mathrm{B}_{1j} &= 3\int_{\gamma_j}\frac{P_5}{r^5}\,\psi_2\varphi_1^4\,dw+6\int_{\gamma_j}\frac{S_2R_4}{r^5}\,\psi_2\varphi_1^4\,dw \nonumber \\
&= 3R_5(0)a_{2j}-3\int_{\gamma_j}\frac{S_2R_5}{r^6}\,\varphi_1^5\,dw+6\int_{\gamma_j}\frac{S_2R_4}{r^5}\,\psi_2\varphi_1^4\,dw, \label{eq:Beta1t-Beta1}
\end{align}
by a simple integration by parts argument.

Under all these modifications we obtain a new expression for $\tilde{a}_{6j}-a_{6j}$. Moreover, a closer look at the newly found expressions shows that all integrals that appear in such expressions will cancel each other out \emph{except} those in which $\varphi_1$ appears raised to the sixth power. Indeed, the integral in (\ref{eq:Delta1t-Delta1}) is canceled out by the integral in (\ref{eq:psi4t-psi4}). Similarly the one in (\ref{eq:Delta2t-Delta2}) and the last integral in (\ref{eq:Gamma1t-Gamma1}), that in (\ref{eq:Delta3t-Delta3}) and the last integral in (\ref{eq:Gamma2t-Gamma2}), the first integral in (\ref{eq:Delta1,1t-Delta1,1}) and the second one in (\ref{eq:Gamma0,1t-Gamma0,1}), the first integral in (\ref{eq:Gamma1t-Gamma1}) and the one in (\ref{eq:psi5t-psi5}) and the first integral on (\ref{eq:Gamma2t-Gamma2}) and the last one in (\ref{eq:Beta1t-Beta1}) cancel each other out. We now group all remaining integrals into a single one. We obtain
\[ \int_{\gamma_j}\frac{\tilde{q}_4R_3-\frac{1}{2}S_2R_3^2-S_3R_4-3S_2R_5}{r^6}\,\varphi_1^5\,dw. \]
But recall that we have defined $P_6=\tilde{q}_6-q_6+\tilde{q}_4R_3-\frac{1}{2}S_2R_3^2-S_3R_4-3S_2R_5$ and $\psi_{6}=\int_{\gamma_j}\frac{q_6}{r^6}\,\varphi_1^5\,dw$. Since the expression $\tilde{\psi}_{6j}-\psi_{6j}$ appears at the end of (\ref{a6t-a6_v2_8}), we can group it with the above integral to obtain a term
\[ \int_{\gamma_j}\frac{P_6}{r^6}\,\varphi_1^5\,dw. \]
Note also that the term $R_3(0)\tilde{\psi}_{4j}$ appears in (\ref{eq:Delta1,1t-Delta1,1}) and (\ref{eq:Gamma0,1t-Gamma0,1}) with opposite signs, so we cancel out these as well. 

We finally obtain a new expression for $\tilde{a}_{6j}-a_{6j}$ from equations (\ref{eq:Pt-P}), (\ref{a6t-a6_v2_2}), (\ref{a6t-a6_v2_3}), (\ref{a6t-a6_v2_4}), (\ref{eq:Delta1t-Delta1}), (\ref{eq:Delta2t-Delta2}), (\ref{eq:Delta3t-Delta3}), (\ref{eq:Delta1,1t-Delta1,1}), (\ref{eq:psi4t-psi4}), (\ref{eq:Gamma1t-Gamma1}), (\ref{eq:Gamma2t-Gamma2}), (\ref{eq:Gamma0,1t-Gamma0,1}), (\ref{eq:psi5t-psi5}) and (\ref{eq:Beta1t-Beta1}) and taking into account the above considerations. 

\begin{formula}\label{formula1}
The difference $\tilde{a}_{6j}-a_{6j}$ is given by the following expression:
\begin{align}
& 5(h_3-h_2^2)a_{2j}^3+\big(4h_4-12h_3h_2^2+8h_2^3-2c_3h_2-3h_2\psi_{3j}\big)a_{2j}^2 \label{a6t-a6_v3_1} \\
& +\big(3h_3\psi_{3j}-h_2^2\psi_{3j}+4c_2h_2\psi_{3j}-4h_2\Delta_{1j}-4h_2\psi_{4j}\big)a_{2j} -3h_2\psi_{3j}^2 \label{a6t-a6_v3_2} \\
& +\frac{\tilde{c}_3-c_3}{2}a_{2j}^3+\left(\tilde{c}_4-c_4-\frac{\tilde{c}_3\tilde{c}_2-c_3c_2}{2}\right)a_{2j}^2 \label{a6t-a6_v3_3} \\
& +\left(\frac{3\tilde{c}_5-3c_5}{4}-\frac{\tilde{c}_4\tilde{c}_2-c_4c_2}{2}-\frac{\tilde{c}_3^2-c_3^2}{8}+\frac{\tilde{c}_3\tilde{c}_2^2-c_3c_2^2}{4}+\frac{\tilde{c}_3-c_3}{2}\psi_{3j}\right)a_{2j} \label{a6t-a6_v3_4} \\
& -\frac{\tilde{c}_2-c_2}{2}\psi_{3j}^2+\left(\frac{\tilde{c}_4-c_4}{3}+\frac{\tilde{c}_3\tilde{c}_2-c_3c_2}{3}-\tilde{c}_2^3+c_2^3\right)\psi_{3j} \label{a6t-a6_v3_5} \\
& \left(-\frac{\tilde{c}_3-c_3}{2}+3(\tilde{c}_2^2-c_2^2)\right)\Delta_{1j} +\left(-\frac{\tilde{c}_3}{2}+3\tilde{c}_2^2\right)R_3(0)a_{2j} \label{a6t-a6_v3_6} \\
& -3(\tilde{c}_2-c_2)\Delta_{2j}-3\tilde{c}_2R_3(0)a_{2j}^2+R_3(0)a_{2j}^3 \label{a6t-a6_v3_7} \\
& \left(-R_3(0)\psi_{3j}-\frac{1}{2}R_3(0)^2\right)a_{2j}-R_3(0)\Delta_{1j}-R_3(0)\psi_{4j} \label{a6t-a6_v3_8} \\
& \left(-\frac{\tilde{c}_3-c_3}{2}+3(\tilde{c}_2^2-c_2^2)\right)\psi_{4j} -6(\tilde{c}_2-c_2)\Gamma_{1j}+6\tilde{c}_2R_4(0)a_{2j} \label{a6t-a6_v3_9} \\
& -3R_4(0)a_{2j}^2 -R_4(0)\psi_{3j} -3(\tilde{c}_2-c_2)\psi_{5j} +3R_5(0)a_{2j} \label{a6t-a6_v3_10} \\
& +\int_{\gamma_j}\frac{P_6}{r^6}\,\varphi_1^5\,dw. \label{a6t-a6_v3_11}
\end{align}
\end{formula}

We now deduce a second expression for $\tilde{a}_{6j}-a_{6j}$. The coefficient of degree 6 in the power series expansion of $h\circ f_j-\tilde{f}_j\circ h$ is of the form $a_{6j}-\tilde{a}_{6j}+\ldots$. Let us take into account the formulas for $a_{3j}$, $a_{4j}$ and $a_{5j}$ found in Proposition \ref{prop:thirdvar}, Proposition \ref{prop:fourthvar} and Proposition \ref{prop:fifthvar} respectively. Let us also take into account that $\tilde{a}_{2j}=a_{2j}$, $\tilde{a}_{3j}=a_{3j}$ and let us substitute $\tilde{a}_{4j}$ and $\tilde{a}_{5j}$ by their formulas implied by equations (\ref{eq:a4t-a4_v1}) and (\ref{eq:a5t-a5-v1}), respectively. Under these considerations the explicit expression for the coefficient of degree six in $h\circ f_j-\tilde{f}_j\circ h$ may be easily obtained by a simple computed assisted computation. 

\begin{formula}\label{formula2}
 The difference $\tilde{a}_{6j}-a_{6j}$ is also given by the following expression:
\begin{align}
& -3h_2\psi_{5j}+(-h_3+4h_2^2+6c_2h_2)\psi_{4j}-\frac{7}{2}h_2\psi_{3j}^2 \label{a6t-a6_v1_1} \\
& +(h_4-2h_3h_2+c_2h_3-4c_2h_2^2-3c_2^2h_2)\psi_{3j} \label{a6t-a6_v1_2} \\
& -6h_2\Gamma_{1j}-3h_2\Delta_{2j}+(-h_3+4h_2^2+6c_2h_2)\Delta_{1j} \label{a6t-a6_v1_3} \\
& +\big(-4h_2\psi_{4j}+(4h_3-2h_2^2+4c_2h_2)\psi_{3j}-4h_2\Delta_{1j}\big)a_{2j} \label{a6t-a6_v1_4} \\
& +\left(3h_5-12h_4h_2-5h_3^2-\frac{1}{2}c_3h_3+28h_3h_2^2-14h_2^4+2c_3h_2^2-2c_4h_2+c_3c_2h_2\right)a_{2j} \label{a6t-a6_v1_5} \\
& +\left(-3h_2\psi_{3j}+7h_4-21h_3h_2+14h_2^3-\frac{7}{2}c_3h_2\right)a_{2j}^2 \label{a6t-a6_v1_6} \\
& +6\big(h_3-h_2^2\big)a_{2j}^3. \label{a6t-a6_v1_7} 
\end{align}
\end{formula}

We now proceed to compare the two formulas above. We shall see once again that everything that depends non--trivially on the index $j$ will be canceled out except for those terms which are a scalar multiple of $a_{2j}$, and the integral (\ref{a6t-a6_v3_11}).

Let us start with those terms having $a_{2j}^3$. For our first formula we have such terms on expressions (\ref{a6t-a6_v3_1}), (\ref{a6t-a6_v3_3}) and (\ref{a6t-a6_v3_7}), which add up to
\[ \left(5(h_3-h_2^2)+\frac{\tilde{c}_3-c_3}{2}+R_3(0)\right)a_{2j}^3. \]
It follows from (\ref{eq:h3}) that $h_3-h_2^2=\frac{\tilde{c}_3-c_3}{2}+R_3(0)$, and so the above expression equals $6(h_3-h_2^2)$ which is exactly (\ref{a6t-a6_v1_7}); the unique term in Formula \ref{formula2} having $a_{2j}^3$.

Consider now those terms with $a_{2j}^2$. Gathering those in Formula \ref{formula1} from (\ref{a6t-a6_v3_1}), (\ref{a6t-a6_v3_3}), (\ref{a6t-a6_v3_7}) and (\ref{a6t-a6_v3_10}) we get
\[ 4h_4-12h_3h_2^2+8h_2^3-3h_2\psi_{3j}-2c_3h_2 +\tilde{c}_4-c_4-\frac{\tilde{c}_3\tilde{c}_2-c_3c_2}{2}-3\tilde{c}_2R_3(0)-3R_4(0). \]
Using the formula for $h_4$ from (\ref{eq:h4}) we may transform the above expression into
\[ \left(7h_4-21h_3h_2+14h_2^3-\frac{7}{2}c_3h_2-3h_2\psi_{3j}\right)a_{2j}^2, \]
which is exactly (\ref{a6t-a6_v1_6}).

Let us consider now those terms that have simultaneously $a_{2j}$ \emph{and} something else that depends on the index $j$. Such terms in Formula \ref{formula1} appear in (\ref{a6t-a6_v3_2}), (\ref{a6t-a6_v3_4}) and (\ref{a6t-a6_v3_8}). They add up to the following expression:
\[ \left( 3h_3\psi_{3j}-h_2^2\psi_{3j}+4c_2h_2\psi_{3j}-4h_2\Delta_{1j}-4h_2\psi_{4j} +\frac{\tilde{c}_3-c_3}{2}\psi_{3j}-R_3(0)\psi_{3j} \right)a_{2j}. \]
Substituting $h_3-h_2^2$ instead of $\frac{\tilde{c}_3-c_3}{2}-R_3(0)$, the above turns into
\[ \left( 4h_3\psi_{3j}-2h_2^2\psi_{3j}+4c_2h_2\psi_{3j}-4h_2\Delta_{1j}-4h_2\psi_{4j} \right)a_{2j}, \]
which agrees with (\ref{a6t-a6_v1_4}).

Recall that $h_2=\tilde{c}_2-c_2$. Those terms having $\psi_{3j}^2$ are easily seen to cancel each other out; they are the last term in (\ref{a6t-a6_v3_2}) and the first one in (\ref{a6t-a6_v3_5}) for Formula \ref{formula1}, and the last term in (\ref{a6t-a6_v1_1}) for Formula \ref{formula2}.

Now, let us consider those terms with a single $\psi_{3j}$. In Formula \ref{formula1} they appear only in (\ref{a6t-a6_v3_5}) and (\ref{a6t-a6_v3_10}), and in Formula \ref{formula2} they are exactly those terms in (\ref{a6t-a6_v1_2}). Let us substitute the $h_4$ term in (\ref{a6t-a6_v1_2}) by the expression given in (\ref{eq:h4}). Under this substitution (\ref{a6t-a6_v1_2}) becomes
\begin{equation}\label{F1vsF2_psi3}
 \left(\frac{\tilde{c}_4-c_4}{3}-\frac{\tilde{c}_3\tilde{c}_2-c_3c_2}{6}-R_4(0)-\tilde{c}_2R_3(0)+h_3h_2-2h_2^3+\frac{c_3}{2}h_2+c_2h_3-4c_2h_2^2-3c_2^2h_2\right)\psi_{3j}. 
\end{equation}
According to Proposition \ref{prop:h2} and equation (\ref{eq:h3}) we have
\begin{align*}
 h_2&=\tilde{c}_2-c_2 & h_3&=\tilde{c}_2^2-2\tilde{c}_2c_2+c_2^2+\frac{\tilde{c}_3-c_3}{2}+R_3(0).
\end{align*}
Substituting the above expressions into (\ref{F1vsF2_psi3}) yields, after simplification,
\[ \left(\frac{\tilde{c}_4-c_4}{3}+\frac{\tilde{c}_3\tilde{c}_2-c_3c_2}{3}-R_4(0)-\tilde{c}_2^3+c_2^3 \right)\psi_{3j}, \]
which matches exactly those terms in Formula \ref{formula1} having $\psi_{3j}$.

The term $(-h_3+4h_2^2+6c_2h_2)\Delta_{1j}$ in (\ref{a6t-a6_v1_3}) may be rewritten, after replacing $h_3$ by its formula in (\ref{eq:h3}), as
\[ \left( 3h_2^2-\frac{\tilde{c}_3-c_3}{2}-R_3(0)+6c_2h_2\right)\Delta_{1j}, \]
which is easily seen to match those terms with $\Delta_{1j}$ in (\ref{a6t-a6_v3_6}) and (\ref{a6t-a6_v3_8}), once we replace $h_2$ by $\tilde{c}_2-c_2$.

Note that the terms having $\psi_{4j}$ in (\ref{a6t-a6_v1_1}) are $(-h_3+4h_2^2+6c_2h_2)\psi_{4j}$. The coefficient is the same than the coefficient for the $\Delta_{1j}$ term we just analysed, so the same argument shows that this term cancels out those terms in (\ref{a6t-a6_v3_8}) and (\ref{a6t-a6_v3_9}) having $\psi_{4j}$.

Taking into account that $h_2=\tilde{c}_2-c_2$ it is straight forward that those terms having $\Delta_{2j}$, $\Gamma_{1j}$ or $\psi_{5j}$ in Formula \ref{formula1} will cancel out the corresponding ones in Formula \ref{formula2}.

We conclude that equating Formula \ref{formula1} to Formula \ref{formula2} yields, after simplification, an equation of the form
\[ a_{2j}\,\mathcal{C}_6+\mathcal{I}_{6j}=0, \]
where
\begin{align*}
\mathcal{C}_6 &= \frac{3\tilde{c}_5-3c_5}{4}-\frac{\tilde{c}_4\tilde{c}_2-c_4c_2}{2}-\frac{\tilde{c}_3^2-c_3^2}{8}+\frac{\tilde{c}_3\tilde{c}_2^2-c_3c_2^2}{4} \\
&\phantom{=} +\left(-\frac{\tilde{c}_3}{2}+3\tilde{c}_2^2\right)R_3(0)-\frac{1}{2}R_3(0)^2+6\tilde{c}_2R_4(0)+3R_5(0)\\
&\phantom{=} -3h_5+12h_4h_2+5h_3^2+\frac{1}{2}c_3h_3-28h_3h_2^2+14h_2^4-2c_3h_2^2+2c_4h_2-c_3c_2h_2,
\end{align*}
and
\[ \mathcal{I}_{6j}=\int_{\gamma_j}\frac{P_6}{r^6}\,\varphi_1^5\,dw. \]
By Proposition \ref{prop:key},
\[ \mathcal{I}_{6j}=\int_{\gamma_j}\frac{P_6}{r^6}\,\varphi_1^5\,dw=0, \qquad \mathcal{C}_6=0. \]
This proves the Key lamma for degree six, and completes the proof of Lemma \ref{lemma:key}.
\end{proof}

\begin{proposition}\label{prop:R6}
If $\lambda_1,\lambda_2\notin\frac{1}{5}\Z$ there exists a polynomial $R_6(w)$ such that 
\[ \int_{0}^{w}\frac{P_6(t)}{r(t)^6}\,\varphi_1(t)^5\,dt=\frac{R_6(w)}{r(w)^5}\,\varphi_1(w)^5+R_6(0). \]
\end{proposition}

\begin{proof}
 Apply Lemma \ref{lemma:Pyartli} with $P(w)=P_6(w)$ and $u_j=5\lambda_j-6$.
\end{proof}

%% file: 7-elimination.tex

We have completed the proof of the \hyperref[lemma:main]{Main lemma}, which claims the existence of polynomials $F_d$, $d=3,\ldots,6$, such that if $\F(\lambda,\alpha)$ and $\F(\lambda,\beta)$ have conjugate holonomy groups at infinity then 
\begin{equation}\label{eq:polysystbis}
 F_3(\beta)=0,\,\ldots,\,F_6(\beta)=0.
\end{equation}
The \hyperref[lemma:elimination]{Elimination lemma} claims that for generic $(\lambda,\alpha)\in\C^5$ the above polynomial system of equations has a unique solution given by $\beta=\alpha$. In order to prove such lemma we need to compute explicit expressions for the polynomials $F_d$ in terms of the parameters $\alpha$ and $\lambda$. We can explicitely construct such polynomials $F_d$ following the proof of the \hyperref[lemma:key]{Key lemma} (which is split into  Propositions \ref{prop:key3}, \ref{prop:key4}, \ref{prop:key5}, \ref{prop:key6}) and the ideas presented in Subsection \ref{subsec:keytomain} (Deducing Main lemma from Key lemma). All computations in this section have been carried out using computer assistance.

Recall that we have defined $F(z,w)$ to be the right hand side of the equation
\begin{equation}\label{eq:ratdifeqbis}
 \frac{dz}{dw}=\frac{z\,P(z,w)}{Q(z,w)},
\end{equation}
and that we have defined the rational functions $K_d(w)$ to be the coefficients
\[ F(z,w)=\sum_{d=1}^{\infty}K_d(w)\,z^d. \]
We replace $F(z,w)$ by its explicit expression (\ref{eq:normalform}) and expand it into a power series with respect to $z$ around $z=0$. After this, we split each coefficient $K_d(w)$ into
\[ K_d(w)=c_d\,K_1(w)+\frac{S_d(w)}{r(w)^d}, \]
according to Proposition \ref{prop:c_dS_d}. We obtain the following expressions for the numbers $c_d$,
\begin{align*}
 c_2&=\alpha_0(1-\sigma), & c_3&=-\alpha_0^{2}\sigma(1-\sigma), & c_4&=\alpha_0^{3}\sigma^{2}(1-\sigma),\\
 c_5&=-\alpha_0^{4}\sigma^{3}(1-\sigma), & c_6&=\alpha_0^{5}\sigma^{4}(1-\sigma),
\end{align*}
and for the polynomials $S_d(w)$,
\begin{align*}
S_2(w) &= r(w) \\
S_3(w) &= -s(w)p(w)r(w) + (\eta-\alpha_0\sigma)r(w)^2 \\
S_4(w) &= -p(w)r(w)^2+\alpha_0(2\sigma-1)s(w)p(w)r(w)^2+\alpha_0\sigma(\alpha_0\sigma-\eta)r(w)^3 \\
S_5(w) &= s(w)p(w)^2r(w)^2+(2\alpha_0\sigma-\eta)p(w)r(w)^3+\alpha_0^2\sigma(2-3\sigma)s(w)p(w)r(w)^3\\
  &\phantom{=} +\alpha_0^2\sigma^2(\eta-\alpha_0\sigma)r(w)^4 \\
S_6(w) &= p(w)^2r(w)^3+\alpha_0(1-3\sigma)s(w)p(w)^2r(w)^3+(2\alpha_0\sigma\eta-3\alpha_0^2\sigma^2)p(w)r(w)^4\\
  &\phantom{=} -\alpha_0^3\sigma^2(3-4\sigma)s(w)p(w)r(w)^4+\alpha_0^3\sigma^3(\alpha_0\sigma-\eta)r(w)^5.
\end{align*}

\begin{remark}
 These computations agree with those presented in \cite{Pyartli2006}. We remark that it is a consequence of the normal form (\ref{eq:normalform}) we have adopted, that all the above polynomials are divisible by $r(w)$ to some positive power and that $S_2(w)$ does not depend on the parameter $\alpha$ (cf.~Proposition \ref{prop:S2=r}).
\end{remark}

\subsection{Main lemma revisited}\label{subsec:mainlemmarevisited}

In Subsection \ref{subsec:keytomain} we have proved the \hyperref[lemma:main]{Main lemma} modulo the auxiliary facts that
\[ \deg{P_d(w)}=2(d-1), \qquad \mbox{and} \qquad \deg{R_d(w)}\leq\deg{P_d(w)}-1. \] 

It follows from a direct inspection of the expressions found for the polynomials $P_d(w)$ in Propositions \ref{prop:key3}, \ref{prop:key4}, \ref{prop:key5} and \ref{prop:key6} that for each $d=3,\ldots,6$, and the expressions for $S_d(w)$ above, that the polynomial $P_d(w)$ has degree $2(d-1)$. We now show that $\deg{R_d(w)}\leq\deg{P_d(w)}-1$ using Lemma \ref{lemma:Pyartli}.

\begin{proposition}\label{prop:degRd} 
 For $d=3,4,5,6$, the polynomials $R_d(w)$ have degree at most $\deg{P_d(w)}-1$.
\end{proposition}

\begin{proof}
We know that
\[ \int_{\gamma_1}\frac{P_d(w)}{r(w)^d}\,\varphi_1(w)^{d-1}\,dw = 0, \]
and we have defined the polynomials $R_d(w)$ by applying Lemma \ref{lemma:Pyartli} with $P(w)=P_d(w)$ and $u_j=(d-1)\lambda_j-d$. Lemma \ref{lemma:Pyartli} also implies that
\[ \deg{R_d}(w)\leq \max{\big(\deg{P_d}(w)-1,\,-2-\operatorname{Re}{(u_1+u_2)}\big)}. \]
Since $\Re\lambda_1+\Re\lambda_2\geq 2/3$, we conclude that 
\[ \operatorname{Re}(u_1+u_2)\geq\frac{2}{3}(d-1)-2d, \]
and thus 
\[ -2-\operatorname{Re}{(u_1+u_2)}\big)\leq\frac{4d-4}{3}. \]

On the other hand $\deg{P_d(w)}=2(d-1)$ and 
\[ 2(d-1)-1\geq\frac{4d-4}{3} \]
for any $d\geq3$.
\end{proof}

\subsection{Computing the polynomials \texorpdfstring{$F_d$}{Fd}}\label{subsec:computingFd}

Now that the \hyperref[lemma:main]{Main Lemma} has been fully proved we shall explain how to get explicit expressions for the polynomials $F_d(w)$. In the next subsection we use these explicit expressions to prove the \hyperref[lemma:elimination]{Elimination Lemma}.

Suppose $P_d(w)$ has degree $m$, and so $R_d(w)$ has degree at most $m-1$. Let $V_m$, $V_{m-1}$ denote the vector spaces of polynomials in $w$ of degree at most $m$ and $m-1$, respectively. We have seen in Subsection \ref{subsec:keytomain}, equation (\ref{eq:P_dR_d}), that 
\[ P_d=R'_dr+(d-1)(s-r')R_d. \]
Consider now the linear map
\[ L_d\colon V_{m-1}\longrightarrow V_m, \qquad f(w)\longmapsto f'(w)r(w)+(d-1)(s(w)-r'(w))f(w), \]
where $s(w)$ and $r(w)$ are the polynomials defined in Section \ref{sec:normalizations}. We prove below that the map $L_d$ has maximal rank and so its image $L_d(V_{m-1})$ is a hyperplane in $V_m$. Any hyperplane is given by the kernel of some (fixed) linear functional $T_d$. We have that $\int_{\gamma_1}\frac{P_d}{r^d}\varphi_1^{d-1}\,dt=0$ if and only if $P_d$ belongs to the image of $L_d$, if and only if $T_d(P_d)=0$. Since the coefficients of $P_d$ are polynomials on $\beta$ the expression $T_d(P_d)$ is also a polynomial on $\beta$. In this way we have that $F_d:=T_d(P_d)$ vanishes if $\int_{\gamma_1}\frac{P_d}{r^d}\varphi_1^{d-1}\,dt=0$.

\begin{proposition}\label{prop:rankLd}
The linear map
\[ L_d\colon V_{2d-3}\longrightarrow V_{2d-2}, \qquad f\longmapsto f'r+(d-1)(s-r')f\]
has, with respect to the standard bases $\{1,w,\ldots,w^{2d-3}\}$ and $\{1,w,\ldots,w^{2d-2}\}$, the following matrix representation

\[ M_d=\begin{pmatrix}
A_d & -1 & 0 & \cdots & 0 & 0 & 0\\
B_d-2d+2 & A_d & -2 & \cdots & 0 & 0 & 0\\
0 & B_d-2d+3 & A_d & \cdots & 0 & 0 & 0\\
\vdots & \vdots & \vdots & ~ & \vdots & \vdots & \vdots\\
0 & 0 & 0 & \cdots & B_d-3 & A_d & -2d+3\\
0 & 0 & 0 & \cdots & 0 & B_d-2 & A_d\\
0 & 0 & 0 & \cdots & 0 & 0 & B_d-1\\
   \end{pmatrix},\]
where
\[ A_d=(d-1)(-\lambda_1+\lambda_2), \qquad B_d=(d-1)(\lambda_1+\lambda_2). \]

In particular, if $\lambda_3\notin\frac{1}{3}\Z\cup\frac{1}{4}\Z\cup\frac{1}{5}\Z$ then the linear map $L_d$ has maximal rank for each $d=3,\ldots,6$.
\end{proposition}

\begin{proof}
Obtaining the expression for the above matrix is a straightforward computation. Note that if we drop the first row in the above matrix we obtain an upper--triangular $2(d-1)\times2(d-1)$ matrix whose diagonal entries are of the form $B_d-k=(d-1)(\lambda_1+\lambda_2)-k$ with $k=1,\ldots,2d-2$. Note moreover that such an expression may vanish only if
\[ \lambda_3=1-\lambda_1-\lambda_2\in\frac{1}{d-1}\Z. \]
This shows that under our genericity assumptions the matrix $M_d$, $d=3,\dots,6$, has maximal rank.
\end{proof}

\begin{remark}\label{rmk:Mtilde}
Let $\tilde{M}_d$ be the $2(d-1)\times2(d-1)$ matrix obtained by dropping the first row of $M_d$. Also, let us denote by $\tilde{V}_{2d-2}\subset V_{2d-2}$ the subspace of polynomials without constant term. If we compose the map $L_d$ with the natural projection $V_{2d-2}\to \tilde{V}_{2d-2}$ we obtain a linear map $\tilde{L}_d\colon V_{2d-3}\to\tilde{V}_{2d-2}$ whose matrix representation is precisely $\tilde{M}_d$. Since $\tilde{M}_d$ is invertible we conclude that $\tilde{L}_d$ is an isomorphism.
\end{remark}

In order to compute the polynomials $R_d$ and $F_d$ we input the expressions for $c_k$, $\tilde{c}_k$, $S_k(w)$, $\widetilde{S}_k(w)$ and $R_k(w)$ for each $k<d$. We compute an explicit expression for the polynomial $P_d(w)$ in terms of $\lambda$, $\alpha$, $\beta$ according to the formulas found throughout Section \ref{sec:keylemma}. The polynomial $R_d(w)$ is the unique preimage of $P_d(w)$ under the linear map $L_d$. We can compute this preimage by inverting the isomorphism $\tilde{L}_d$ defined in Remark \ref{rmk:Mtilde}. Indeed, the projection of $P_d(w)$ onto $\tilde{V}_{2d-2}$ is given by $P_d(w)-P_d(0)$ and thus we can find $R_d(w)$ by solving the linear equation
\[ \tilde{L}_d(R_d)(w)=P_d(w)-P_d(0)\in\tilde{V}_{2d-2}. \]
Once an expression for $R_d(w)$ has been found we have that $L_d(R_d)(w)$ and $P_d(w)$ agree on every monomial of positive degree (i.e.~they have the same projections onto $\tilde{V}_{2d-2}$). The condition $L_d(R_d)(w)=P_d(w)$ is thus reduced to the equation
\[ L_d(R_d)(0)=P_d(0). \]
The equation $F_d=L_d(R_d)(0)-P_d(0)$ gives us therefore an explicit expression for $F_d$. Such expressions are quite complicated and so we do not include them here.

\subsection{Concluding the Elimination lemma}\label{subsec:concludingelimination}

Recall that we have defined the series of resultants
\begin{align*}
 \operatorname{Res}^1_j(\beta_0,\beta_1) &= \Res_{\beta_2}\big(F_3(\beta_0,\beta_1,\beta_2),F_j(\beta_0,\beta_1,\beta_2)\big), & j &= 4,5,6, \\
 \operatorname{Res}^2_j(\beta_0) &= \Res_{\beta_1}\big(\operatorname{Res}^1_4(\beta_0,\beta_1),\operatorname{Res}^1_j(\beta_0,\beta_1)\big), & j &= 5,6, \\
 \operatorname{Res}^3_6   &= \Res_{\beta_0}\big(\operatorname{Res}^2_5(\beta_0)/(\beta_0-\alpha_0),\operatorname{Res}^2_6(\beta_0)\big), & &
\end{align*}
and proved in Proposition \ref{prop:elimination1} that if $\operatorname{Res}_6^3\not\equiv0$ as a function of $\lambda$ and $\alpha$ then any solution $(\beta_0,\beta_1,\beta_2)$ to system (\ref{eq:polysystbis}) satisfies $\beta_0=\alpha_0$. 

After finding explicit expressions for the polynomials $F_d$ we have computed the above resultants and verified that $\operatorname{Res}_6^3\not\equiv0$ zero by evaluating it at the values
\begin{equation}\label{values}
  \lambda_1=2-i,\quad \lambda_2=2i,\quad \alpha_0=1,\quad \alpha_1=0,\quad \alpha_2=0,
\end{equation}
and obtaining a non--zero complex number.

The final step in the proof is proving Proposition \ref{prop:elimination2}. The determinant of the linear system 
\[ F_3(\alpha_0,\beta_1,\beta_2)=0,\qquad F_4(\alpha_0,\beta_1,\beta_2)=0 \] 
is also obtained with computer assistance and verified to be non--zero at the values of $\lambda$ and $\alpha$ given in (\ref{values}). All these computations can be found in the \hyperref[sec:appendix]{Appendix}. This completes the proof of the \hyperref[lemma:elimination]{Elimination lemma} and thus complete also the proof of Theorem \ref{thm:main}.

%% file: 8-appendix.tex

\label{sec:appendix}
\begin{changemargin}{-.6in}{-.6in}
\begin{center}
\includegraphics[scale=0.75]{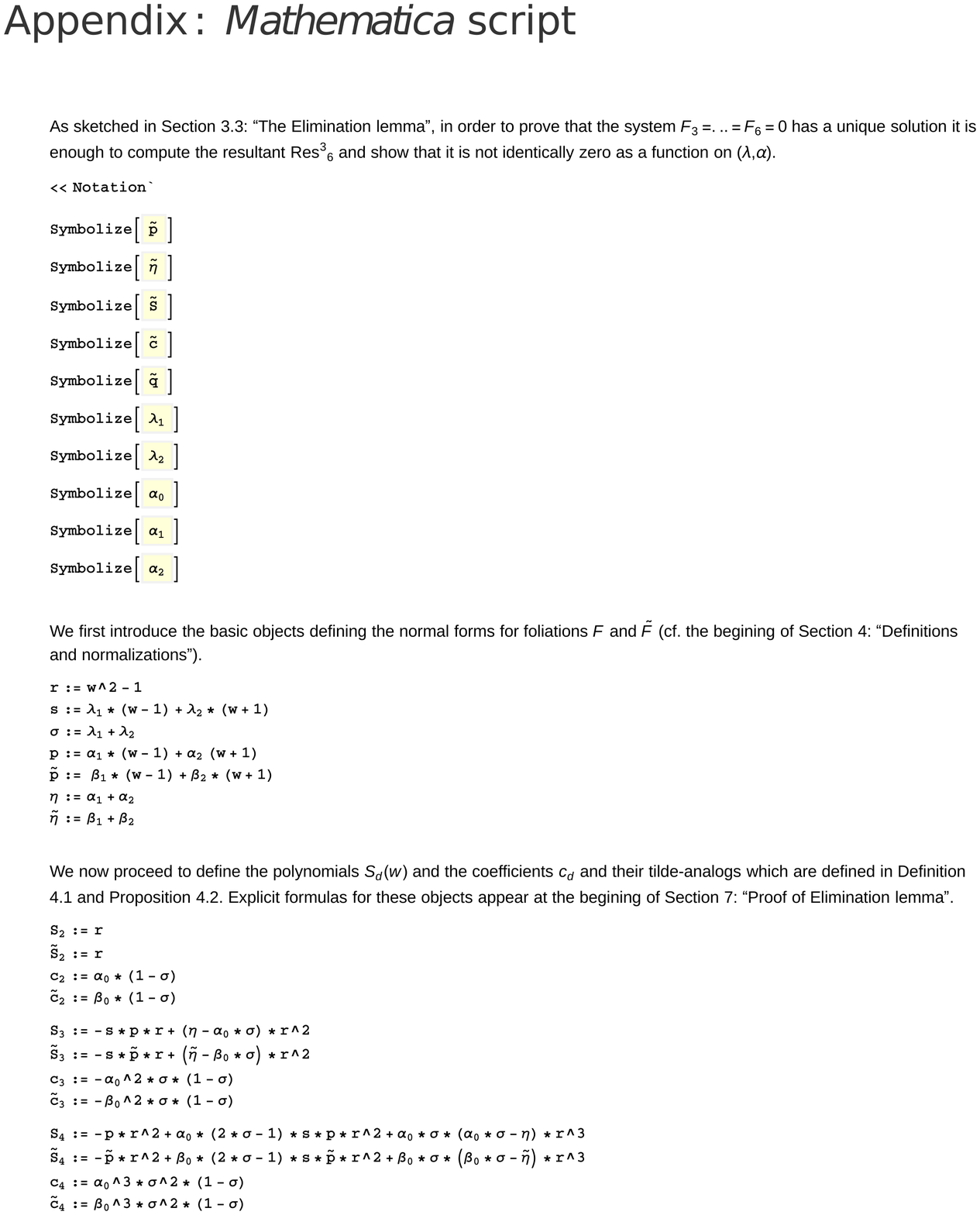}\newpage
\includegraphics[scale=0.75]{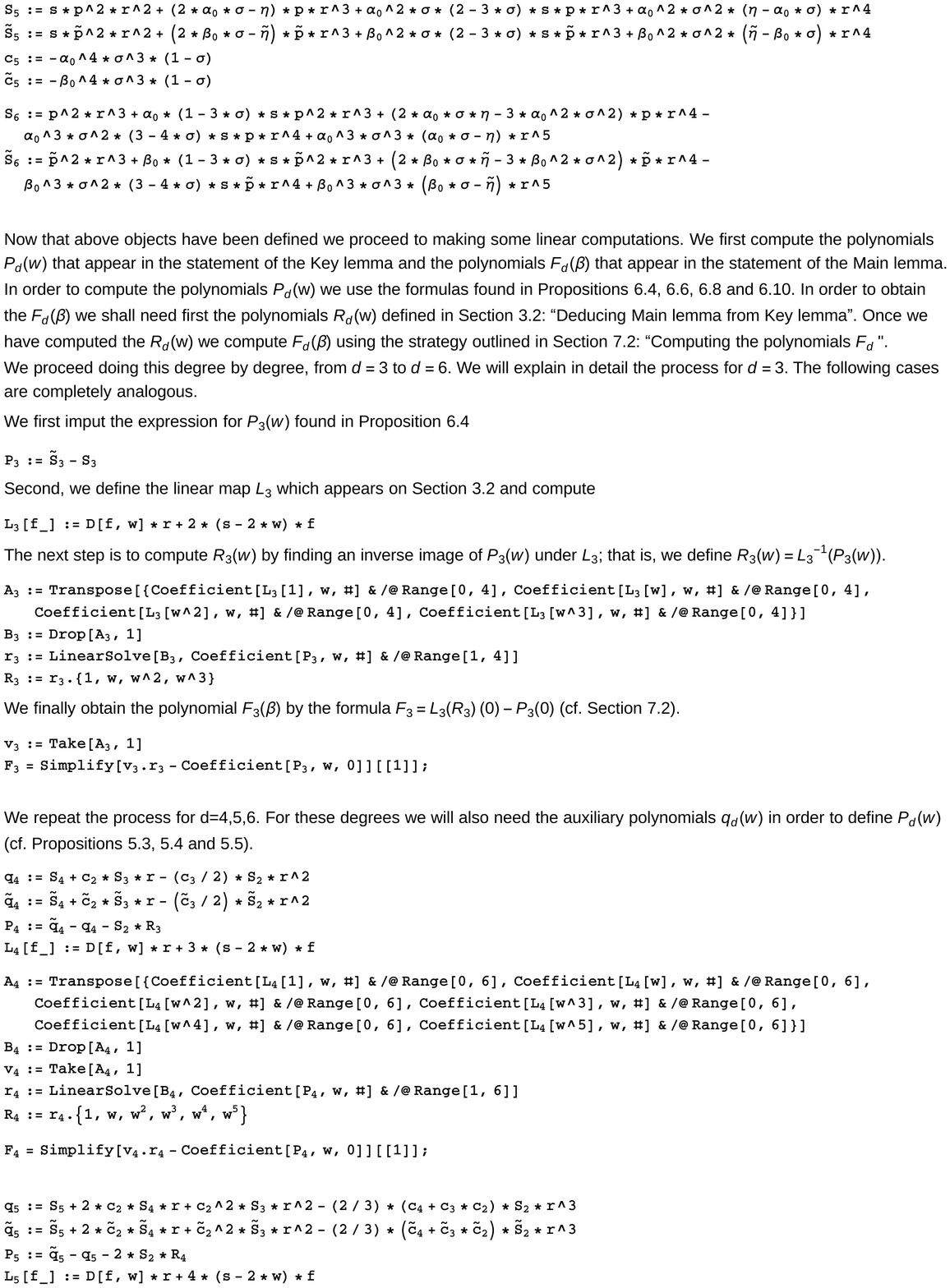}\newpage
\includegraphics[scale=0.75]{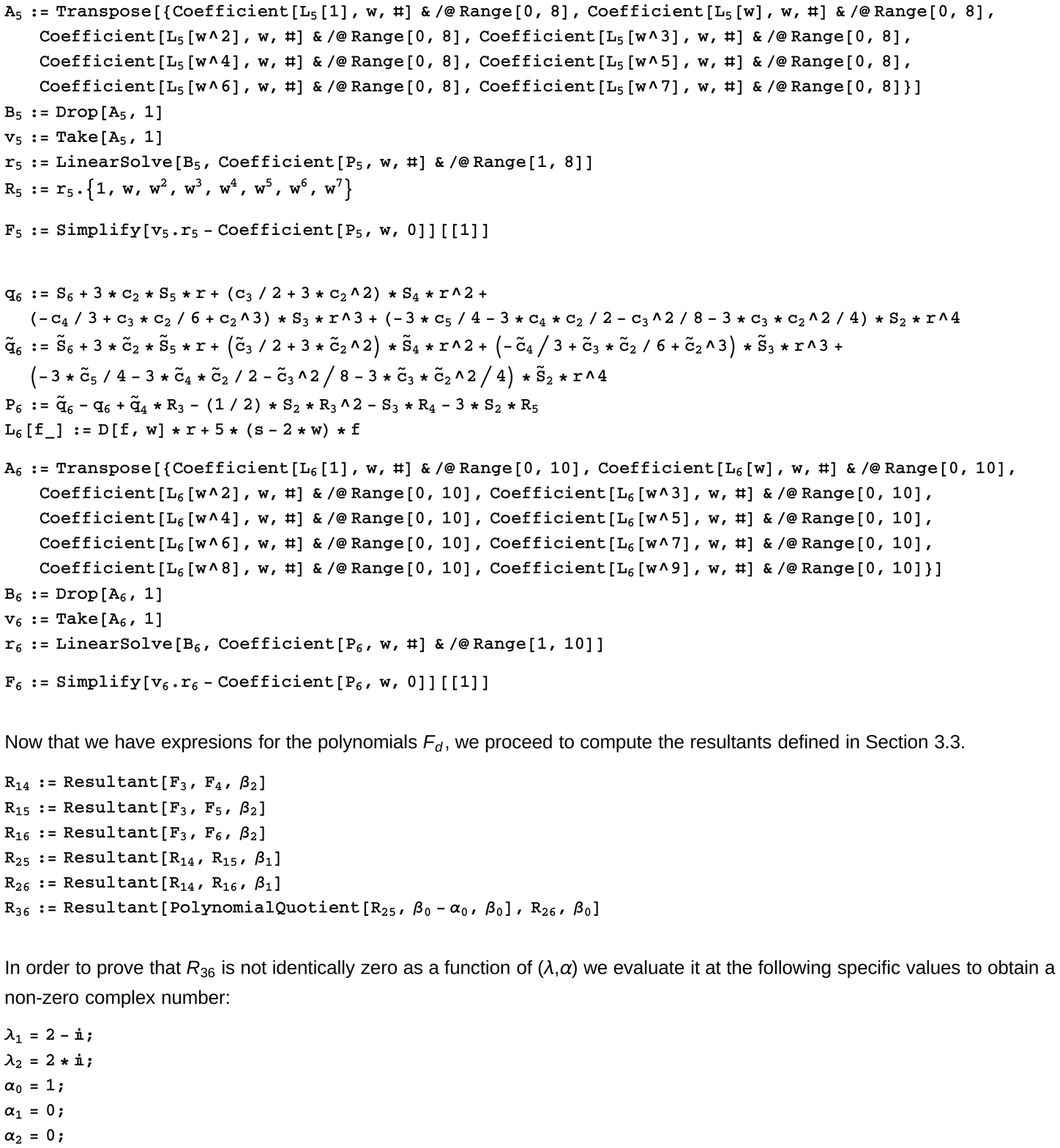}\newpage
\includegraphics[scale=0.75]{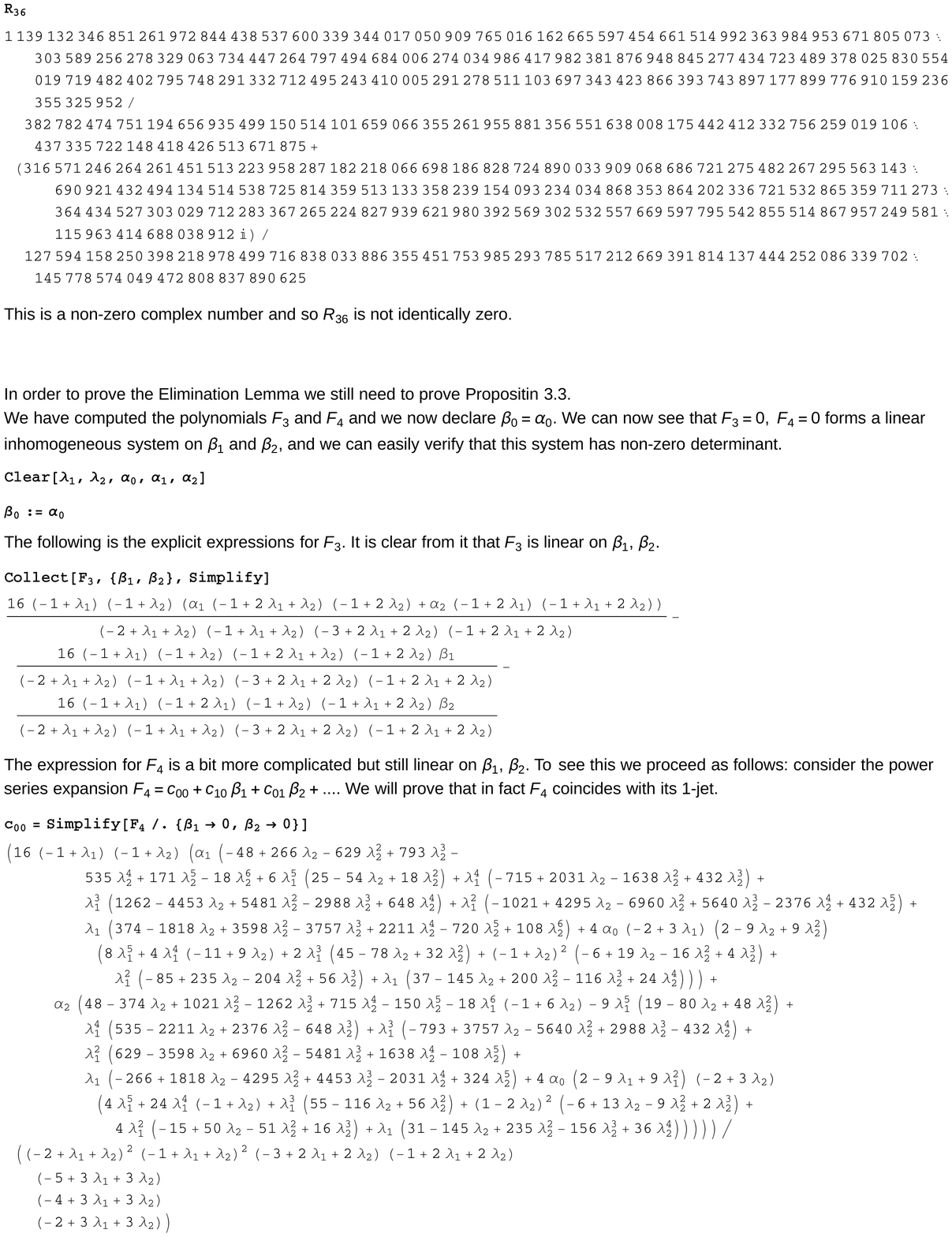}\newpage
\includegraphics[scale=0.75]{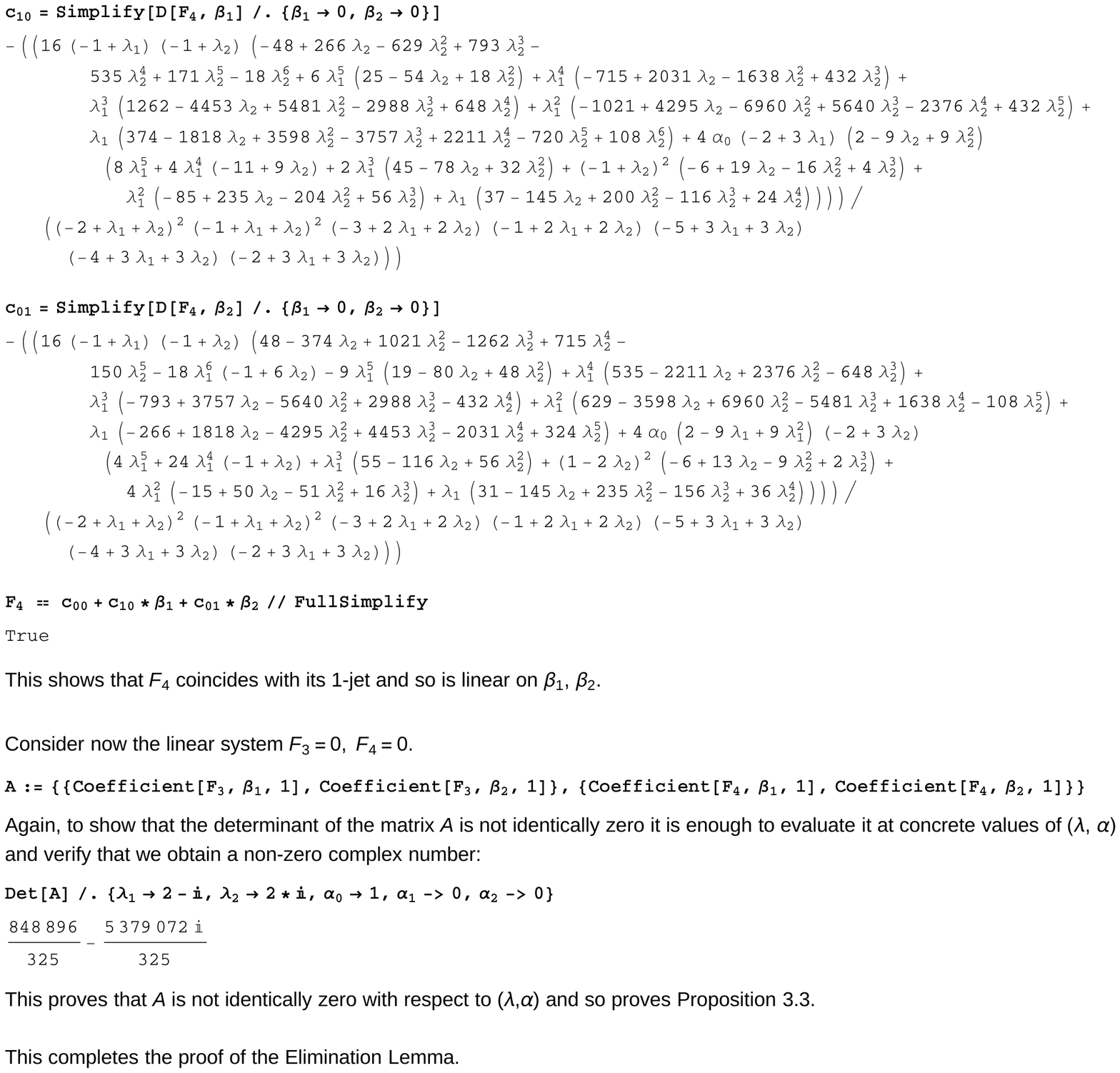}\newpage
\end{center}
\end{changemargin}

%% file: utmost_rigidity_-_arXiv_v2.bbl
\begin{thebibliography}{TGMS13}

\bibitem[Bru04]{Brunella2004}
M.~Brunella.
\newblock {\em Birational geometry of foliations}.
\newblock Publica\c c\~oes Matem\'aticas do IMPA. Instituto de Matem\'atica
  Pura e Aplicada (IMPA), Rio de Janeiro, 2004.
\newblock \href{http://www.ams.org/mathscinet-getitem?mr=1948251}{MR:1948251}.

\bibitem[GMO04]{GomezMontOrtiz2004}
X.~G{\'o}mez-Mont and L.~O{rtiz-Bobadilla}.
\newblock {\em Sistemas din\'amicos holomorfos en superficies
  \textnormal{(Spanish)}}, volume~3 of {\em Aportaciones Matem\'aticas: Notas
  de Investigaci\'on}.
\newblock Sociedad Matem\'atica Mexicana, M\'exico, 2nd edition, 2004.
\newblock \href{http://www.ams.org/mathscinet-getitem?mr=1304495}{MR:1304495}.

\bibitem[Ily78]{Ilyashenko1978}
Yu.~S. Ilyashenko.
\newblock Topology of phase portraits of analytic differential equations on a
  complex projective plane.
\newblock {\em Trudy Sem. Petrovsk.}, (4):83--136, 1978.
\newblock \href{http://www.ams.org/mathscinet-getitem?mr=524528}{MR:524528}.

\bibitem[IM11]{IlyashenkoMoldavskis2011}
Yu.~S. Ilyashenko and V.~Moldavskis.
\newblock Total rigidity of generic quadratic vector fields.
\newblock {\em Mosc. Math. J.}, 11(3):521--530, 630, 2011.
\newblock \href{http://www.ams.org/mathscinet-getitem?mr=2894428}{MR:2894428}.

\bibitem[LNSS98]{LinsNetoSadScardua1998}
A.~Lins~Neto, P.~Sad, and B.~Sc{\'a}rdua.
\newblock On topological rigidity of projective foliations.
\newblock {\em Bull. Soc. Math. France}, 126(3):381--406, 1998.
\newblock
  \href{http://www.numdam.org/item?id=BSMF_1998__126_3_381_0}{NUMDAM:BSMF-1998-126-3-381-0}.

\bibitem[Mar03]{Marin2003}
D.~Mar{\'{\i}}n.
\newblock Moduli spaces of germs of holomorphic foliations in the plane.
\newblock {\em Comment. Math. Helv.}, 78(3):518--539, 2003.
\newblock
  \href{http://dx.doi.org/10.1007/s00014-003-0771-z}{DOI:10.1007/s00014-003-0771-z}.

\bibitem[Nak94]{Nakai1994}
I.~Nakai.
\newblock Separatrices for nonsolvable dynamics on {${\bf C},0$}.
\newblock {\em Ann. Inst. Fourier (Grenoble)}, 44(2):569--599, 1994.
\newblock
  \href{http://www.numdam.org/item?id=AIF_1994__44_2_569_0}{NUMDAM:AIF-1994-44-2-569-0}.

\bibitem[Pya00]{Pyartli2000}
A.~S. Pyartli.
\newblock Rational differential equations with commutative monodromy group at
  infinity.
\newblock {\em Tr. Mosk. Mat. Obs.}, 61:75--106, 2000.
\newblock \href{http://www.ams.org/mathscinet-getitem?mr=1911688}{MR:1911688}.

\bibitem[Pya06]{Pyartli2006}
A.~S. Pyartli.
\newblock Quadratic vector fields on {$\mathbb{C}{\rm P}^2$} with a solvable
  monodromy group at infinity.
\newblock {\em Tr. Mat. Inst. Steklova}, 254(Nelinein. Anal. Differ.
  Uravn.):130--161, 2006.
\newblock \href{http://www.ams.org/mathscinet-getitem?mr=2301002}{MR:2301002}.

\bibitem[Ram14]{Ramirez2014}
V.~Ram\'{i}rez.
\newblock Strong topological invariance of the monodromy group at infinity for
  quadratic vector fields.
\newblock {\em J. Singul.}, 9:193--202, 2014.
\newblock
  \href{http://dx.doi.org/10.5427/jsing.2014.9n}{DOI:10.5427/jsing.2014.9n}.

\bibitem[Shc84]{Shcherbakov1984}
A.~A. Shcherbakov.
\newblock Topological and analytic conjugation of noncommutative groups of
  germs of conformal mappings.
\newblock {\em Trudy Sem. Petrovsk.}, (10):170--196, 238--239, 1984.
\newblock \href{http://www.ams.org/mathscinet-getitem?mr=778885}{MR:778885}.

\bibitem[TGMS13]{TeymuriMafraScardua2013}
M.~Teymuri~Garakani, A.~Mafra, and B.~Sc{\'a}rdua.
\newblock On topological rigidity of phase-portraits in the complex plane.
\newblock {\em Bol. Soc. Mat. Mexicana (3)}, 19(1):65--83, 2013.
\newblock \href{http://www.ams.org/mathscinet-getitem?mr=3097943}{MR:3097943}.

\bibitem[{Wol}]{Mathematica9}
{Wolfram Research, Inc.}
\newblock {{\sl Mathematica}, Version 9.0, Champaign, IL (2012)}.
\newblock
  \href{http://www.wolfram.com/mathematica/}{www.wolfram.com/mathematica}.

\end{thebibliography}
